\documentclass[a4paper,reqno]{amsart}
\usepackage[british]{babel}
\usepackage[T1]{fontenc}
\usepackage{amssymb,amsthm,bm}
\usepackage{mathtools}
\usepackage{nicefrac}
\mathtoolsset{mathic}
\allowdisplaybreaks
\usepackage{courier,mathptmx,stmaryrd}
\usepackage{longtable}
\usepackage[numbers,sort&compress]{natbib}
\usepackage{hyperref,url}
\hypersetup{  
bookmarksnumbered
}
\hypersetup{  
breaklinks=false,
pdfborderstyle={/S/U/W 0},  
citebordercolor=.235 .702 .443,
urlbordercolor=.255 .412 .882,
linkbordercolor=.804 .149 .19,
}
\hypersetup{ 
pdfauthor={Gert de Cooman and Jasper De Bock},
pdftitle={Randomness is inherently imprecise},
pdfkeywords={}
}
\usepackage[all]{hypcap} 
\usepackage{xcolor}
\usepackage{tikz}
\usetikzlibrary{calc,matrix,plotmarks,intersections,arrows,positioning}
\tikzstyle{root}=[rectangle,draw=blue!90]
\tikzstyle{nonterminal}=[rectangle,rounded corners,fill=blue!15,draw=blue!15]
\tikzstyle{terminal}=[rectangle]
\tikzstyle{cut}=[thick,dotted,draw=green!50!black]
\tikzstyle{local}=[color=green!50!black,text=green!25!black]
\usepackage{enumitem}

\DeclarePairedDelimiter{\group}{(}{)}
\DeclarePairedDelimiter{\sqgroup}{[}{]}
\DeclarePairedDelimiter{\set}{\{}{\}}
\DeclarePairedDelimiter{\avg}{\llbracket}{\rrbracket}

\DeclarePairedDelimiter{\norm}{\Vert}{\Vert}
\DeclarePairedDelimiter{\abs}{\vert}{\vert}
\DeclarePairedDelimiter{\dist}{\vert}{\vert}
\DeclarePairedDelimiter{\floor}{\lfloor}{\rfloor}
\DeclarePairedDelimiter{\ceil}{\lceil}{\rceil}

\newcommand{\availables}{\mathcal{A}}
\newcommand{\sells}{\mathcal{S}}

\newcommand{\frcsts}{\sqgroup{0,1}}

\newcommand{\imprecisefrcsts}{\mathcal{I}}

\newcommand{\outcomes}{\set{0,1}}
\newcommand{\pths}{\Omega}
\newcommand{\sits}{\mathbb{S}}

\newcommand{\exact}[1]{\Gamma(#1)}

\newcommand{\constantfrcstsystem}[1][I]{\gamma_{\,#1}}
\newcommand{\lconstantfrcstsystem}[1][I]{\smash{\underline{\gamma}_{\,#1}}}
\newcommand{\uconstantfrcstsystem}[1][I]{\smash{\overline{\gamma}_{\,#1}}}
\newcommand{\average}[2][S]{\avg{#2}_{#1}}

\newcommand{\frcstsystem}{\varphi}
\newcommand{\lfrcstsystem}{\smash{\underline{\frcstsystem}}}
\newcommand{\ufrcstsystem}{\smash{\overline{\frcstsystem}}}
\newcommand{\vacfrcstsystem}{\frcstsystem_{\mathrm{v}}}
\newcommand{\frcstsystems}{\Phi}
\newcommand{\constantfrcstsystems}{\imprecisefrcsts}
\newcommand{\process}{F}
\newcommand{\processtoo}{G}

\newcommand{\selection}{S}
\newcommand{\wangselection}{\hat{S}}
\newcommand{\selectionfunction}{\sigma}
\newcommand{\selectionsum}{\zeta}
\newcommand{\wangselectionsum}{\hat{\zeta}}
\newcommand{\multprocess}{D}
\newcommand{\mint}[1][\multprocess]{#1^{\scriptscriptstyle\circledcirc}}
\newcommand{\supermartin}{M}
\newcommand{\submartin}{M}
\newcommand{\martin}{M}
\newcommand{\wangmartin}{\hat{\martin}}
\newcommand{\submartins}[1][\frcstsystem]{\smash{\underline{\mathbb{M}}^{#1}}}
\newcommand{\supermartins}[1][\frcstsystem]{\smash{\overline{\mathbb{M}}^{#1}}}
\newcommand{\martins}[1][\frcstsystem]{{\mathbb{M}}^{#1}}

\newcommand{\testsupermartins}[1][\frcstsystem]{\smash{\overline{\mathbb{T}}^{#1}}}

\newcommand{\test}{T}
\newcommand{\tests}{{\overline{\mathbb{T}}}}

\newcommand{\allowables}{\mathbb{A}}
\newcommand{\mlallowables}{\smash{\allowables_{\mathrm{ML}}}}
\newcommand{\callowables}{\smash{\allowables_{\mathrm{C}}}}
\newcommand{\poscallowables}{\smash{\allowables^{\scriptscriptstyle+}_{\mathrm{C}}}}
\newcommand{\multmlallowables}{\smash{\allowables^{\scriptscriptstyle\circledcirc}_{\mathrm{ML}}}}
\newcommand{\allowabletests}[1][\frcstsystem]{\smash{\tests^{#1}_{\allowables}}}
\newcommand{\allowableteststoo}[1][\frcstsystem]{\smash{\tests^{#1}_{\allowables'}}}
\newcommand{\mlallowabletests}[1][\frcstsystem]{\smash{\tests^{#1}_{\mathrm{ML}}}}
\newcommand{\multmlallowabletests}[1][\frcstsystem]{\smash{\tests^{{#1},{\scriptscriptstyle\circledcirc}}_{\mathrm{ML}}}}
\newcommand{\callowabletests}[1][\frcstsystem]{\smash{\tests^{#1}_{\mathrm{C}}}}
\newcommand{\poscallowabletests}[1][\frcstsystem]{\smash{\tests^{#1,{\scriptscriptstyle+}}_{\mathrm{C}}}}

\newcommand{\testallowables}[1][]{\allowables\cap\testsupermartins[#1]}
\newcommand{\random}[2][\pth]{\frcstsystems_{#2}(#1)}
\newcommand{\crandom}[1][\pth]{\frcstsystems_{\mathrm{C}}(#1)}
\newcommand{\poscrandom}[1][\pth]{\frcstsystems^{\scriptscriptstyle+}_{\mathrm{C}}(#1)}
\newcommand{\multmlrandom}[1][\pth]{\frcstsystems^{\scriptscriptstyle\circledcirc}_{\mathrm{ML}}(#1)}
\newcommand{\mlrandom}[1][\pth]{\frcstsystems_{\mathrm{ML}}(#1)}
\newcommand{\schnorrrandom}[1][\pth]{\frcstsystems_{\mathrm{S}}(#1)}
\newcommand{\constantrandom}[2][\pth]{\constantfrcstsystems_{#2}(#1)}
\newcommand{\constantcrandom}[1][\pth]{\constantfrcstsystems_{\mathrm{C}}(#1)}
\newcommand{\constantposcrandom}[1][\pth]{\constantfrcstsystems^{\scriptscriptstyle+}_{\mathrm{C}}(#1)}
\newcommand{\constantmultmlrandom}[1][\pth]{\constantfrcstsystems^{\scriptscriptstyle\circledcirc}_{\mathrm{ML}}(#1)}
\newcommand{\constantmlrandom}[1][\pth]{\constantfrcstsystems_{\mathrm{ML}}(#1)}
\newcommand{\constantschnorrrandom}[1][\pth]{\constantfrcstsystems_{\mathrm{S}}(#1)}
\newcommand{\lowerconstantrandom}[2][\pth]{L_{#2}(#1)}
\newcommand{\upperconstantrandom}[2][\pth]{U_{#2}(#1)}
\newcommand{\lowerconstantschnorrrandom}[1][\pth]{L_{\mathrm{S}}(#1)}
\newcommand{\upperconstantschnorrrandom}[1][\pth]{U_{\mathrm{S}}(#1)}

\newcommand{\lowersmallestrandom}[2][\pth]{\lp[#2](#1)}
\newcommand{\uppersmallestrandom}[2][\pth]{\up[#2](#1)}

\newcommand{\lowersmallestschnorrrandom}[1][\pth]{\lp[\mathrm{S}](#1)}
\newcommand{\uppersmallestschnorrrandom}[1][\pth]{\up[\mathrm{S}](#1)}
\newcommand{\varnorm}[1]{\norm{#1}_\mathrm{v}}
\newcommand{\ordering}{\rho}
\newcommand{\realordering}{\tau}

\newcommand{\naturals}{\mathbb{N}}
\newcommand{\naturalswithzero}{\mathbb{N}_0}
\newcommand{\integers}{\mathbb{Z}}

\newcommand{\reals}{\mathbb{R}}

\newcommand{\nonnegreals}{\mathbb{R}_{\geq0}}

\newcommand{\compfrcstsystems}[1][I]{\smash{\mathbb{N}_{0,#1}}}

\newcommand{\ex}{E}
\newcommand{\lex}{\smash{\underline{\ex}}}
\newcommand{\uex}{\smash{\overline{\ex}}}

\newcommand{\pr}{P}
\newcommand{\lpr}{\smash{\underline{\pr}}}
\newcommand{\upr}{\smash{\overline{\pr}}}

\newcommand{\lglobal}[1][\frcstsystem]{\lex^{#1}}
\newcommand{\uglobal}[1][\frcstsystem]{\uex^{#1}}

\newcommand{\lglobalprob}[1][\frcstsystem]{\lpr^{#1}}
\newcommand{\uglobalprob}[1][\frcstsystem]{\upr^{#1}}

\newcommand{\lp}[1][]{\smash{\underline{p}}_{#1}}
\newcommand{\up}[1][]{\smash{\overline{p}}_{#1}}

\newcommand{\pinterval}[1][]{\sqgroup{\lp[#1],\up[#1]}}

\newcommand{\lptoo}[1][]{\smash{\underline{q}}_{#1}}
\newcommand{\uptoo}[1][]{\smash{\overline{q}}_{#1}}

\newcommand{\qinterval}[1][]{\sqgroup{\lptoo[#1],\uptoo[#1]}}

\newcommand{\lowr}[1][]{\smash{\underline{r}}_{#1}}
\newcommand{\uppr}[1][]{\smash{\overline{r}}_{#1}}
\newcommand{\rprime}[1][]{\smash{r'_{#1}}}

\newcommand{\init}{\square}
\newcommand{\pth}{\omega}
\newcommand{\pthat}[1]{\pth_{#1}}
\newcommand{\pthto}[1]{\pth_{1:#1}}
\newcommand{\altpth}{\varpi}

\newcommand{\altpthto}[1]{\altpth_{1:#1}}
\newcommand{\wangpth}{\hat{\omega}}
\newcommand{\wangpthat}[1]{\wangpth_{#1}}
\newcommand{\wangpthto}[1]{\wangpth_{1:#1}}

\newcommand{\sit}{s}
\newcommand{\sitto}[1]{\sit_{1:#1}}
\newcommand{\sitat}[1]{\sit_{#1}}
\newcommand{\altsit}{t}

\newcommand{\altsitat}[1]{\altsit_{#1}}

\newcommand{\ualtsit}{\overline{\altsit}}
\newcommand{\ualtsitto}[1]{\ualtsit_{1:#1}}
\newcommand{\ualtsitat}[1]{\ualtsit_{#1}}
\newcommand{\precedes}{\sqsubseteq}
\newcommand{\sprecedes}{\sqsubset}

\newcommand{\sfollows}{\sqsupset}

\newcommand{\xval}[1][]{x_{#1}}
\newcommand{\xvaltolong}[1][n]{\xval[1],\dots,\xval[#1]}
\newcommand{\xvalto}[1][n]{\xval[1:#1]}
\newcommand{\zval}[1][]{z_{#1}}
\newcommand{\zvaltolong}[1][n]{\zval[1],\dots,\zval[#1]}

\newcommand{\randomoutcome}[1][]{X_{#1}}
\newcommand{\randomoutcometolong}[1][n]{\randomoutcome[1],\dots,\randomoutcome[#1]}
\newcommand{\randomoutcometo}[1][n]{\randomoutcome[1:#1]}

\newcommand{\comp}{computable}
\newcommand{\Comp}{Computable}
\newcommand{\scomp}{semicomputable}
\newcommand{\lscomp}{lower semicomputable}
\newcommand{\uscomp}{upper semicomputable}
\newcommand{\compy}{computability}
\newcommand{\Compy}{Computability}
\newcommand{\scompy}{semicomputability}
\newcommand{\lscompy}{lower semicomputability}

\newcommand{\ML}{Martin-L\"of}

\newcommand{\cset}[3][]{\set[#1]{#2\colon#3}}
\newcommand{\ind}[1]{\mathbb{I}_{#1}}
\newcommand{\indsing}[1]{\ind{\set{#1}}}
\newcommand{\then}{\Rightarrow}

\newcommand{\ifandonlyif}{\Leftrightarrow}

\newcommand{\adddelta}{\Delta}

\newcommand{\parw}[1]{{\scriptscriptstyle(}#1{\scriptscriptstyle)}}

\newtheorem{theorem}{Theorem}
\newtheorem{proposition}[theorem]{Proposition}
\newtheorem{lemma}[theorem]{Lemma}
\newtheorem{corollary}[theorem]{Corollary}
\theoremstyle{definition}
\newtheorem{definition}{Definition}
\newtheorem*{game}{Game}
\theoremstyle{remark}

\begin{document}
\title{Randomness is inherently imprecise}
\author{Gert de Cooman}
\address{Ghent University, Foundations Lab for imprecise probabilities, Technologiepark--Zwijnaarde 125, 9052 Zwijnaarde, Belgium}
\email{gert.decooman@ugent.be}
\author{Jasper De Bock}
\address{Ghent University, Foundations Lab for imprecise probabilities, Technologiepark--Zwijnaarde 125, 9052 Zwijnaarde, Belgium}
\email{jasper.debock@ugent.be}

\begin{abstract}
We use the martingale-theoretic approach of game-theoretic probability to incorporate imprecision into the study of randomness.
In particular, we define several notions of randomness associated with interval, rather than precise, forecasting systems, and study their properties.
The richer mathematical structure that thus arises lets us, amongst other things, better understand and place existing results for the precise limit.
When we focus on constant interval forecasts, we find that every sequence of binary outcomes has an associated filter of intervals it is random for.
It may happen that none of these intervals is precise---a single real number---which justifies the title of this paper.
We illustrate this by showing that randomness associated with non-stationary precise forecasting systems can be captured by a constant interval forecast, which must then be less precise: a gain in model simplicity is thus paid for by a loss in precision.
But imprecise randomness can't always be explained away as a result of oversimplification: we show that there are sequences that are random for a constant interval forecast, but never random for any {\comp} (more) precise forecasting system.
We also show that the set of sequences that are random for a non-vacuous interval forecasting system is meagre, as it is for precise forecasting systems.
\end{abstract}


\maketitle

\section{Introduction}\label{sec:introduction}
This paper documents the first steps in our attempt to incorporate imprecision into the study of algorithmic randomness.
What this means is that we want to allow for, give a precise mathematical meaning to, and study the mathematical consequences of, associating randomness with interval rather than precise probabilities and expectations.
We will see that this is a non-trivial problem, argue that it leads to surprising conclusions about the nature of randomness, and discover that it opens up interesting and hitherto uncharted territory for mathematical and even philosophical investigation.
We believe that our work provides (the beginnings of) a satisfactory answer to questions raised by a number of researchers \cite{walley1982a,fierens2009:frequentist,fierens2009:chaotic,gorban2016:statisticalstability} about frequentist and `objective' aspects of interval, or imprecise, probabilities.

To explain what it is we're after, consider an infinite sequence~\(\pth=(\zvaltolong[n],\dots)\), whose components \(\zval[k]\) are either zero or one, and are typically considered as successive \emph{outcomes} of some experiment.
When do we call such a sequence \emph{random}?
There are many notions of randomness, and many of them have a number of equivalent definitions \cite{ambosspies2000,bienvenu2009:randomness}.
We will focus here essentially on {\ML} randomness, {\comp} randomness, and Schnorr randomness.

The randomness of a sequence~\(\pth\) is typically associated with a probability measure on the sample space of all such infinite sequences, or---which is essentially equivalent due to Ionescu Tulcea's extension theorem \cite[Theorem~II.9.2]{billingsley1995}---with a so-called \emph{forecasting system}~\(\frcstsystem\) that associates with each finite sequence of outcomes~\((\xvaltolong[n])\) the (conditional) expectation~\(\frcstsystem(\xvaltolong[n])=\ex(\randomoutcome[n+1]\vert\xvaltolong[n])\) for the next, as yet unknown, outcome~\(\randomoutcome[n+1]\).\footnote{We will follow the convention of denoting (as yet) unknown things---\emph{variables}---with a capital letter.}\textsuperscript{,}\footnote{The expectation~\(\ex(X)\) of a variable~\(X\) that may only assume the values~\(0\) and \(1\) is actually the probability \(\pr(X=1)\) that it assumes the value~\(1\): \(\ex(X)=0\cdot\pr(X=0)+1\cdot\pr(X=1)\). This observation already explains why, further on, we will assume that this expectation~\(\ex(X)\) lies in the unit interval~\(\frcsts\). For reasons that will become clear later, we prefer to use the language of expectations in this paper.}
This \(\frcstsystem(\xvaltolong[n])\) is the (precise) \emph{forecast} for the value of~\(\randomoutcome[n+1]\) after observing the values~\(\xvaltolong[n]\) of the respective variables~\(\randomoutcometolong[n]\).
The sequence~\(\pth\) is then typically called `random' when it passes some countable number of randomness tests, where the collection of such randomness tests depends of the forecasting system~\(\frcstsystem\).

An alternative and essentially equivalent approach to defining randomness, going back to Ville \cite{ville1939}, sees each forecast~\(\frcstsystem(\xvaltolong[n])\) as a fair price for---and therefore a commitment to bet on---the as yet unknown next outcome~\(\randomoutcome[n+1]\) after observing the first \(n\) outcomes~\(\xvaltolong[n]\).
The sequence~\(\pth\) is then `random' when there is no `allowable' strategy for getting infinitely rich by exploiting the bets made available by the forecasting system~\(\frcstsystem\) along the sequence, without borrowing.
Betting strategies that are made available by the forecasting system~\(\frcstsystem\) are called supermartingales.
Which supermartingales are considered `allowable' differs in various approaches \cite{schnorr1971,downey2010,levin1973:random:sequence,bienvenu2009:randomness,ambosspies2000}, but typically involves some (semi){\compy} requirement---we discuss relevant aspects of {\compy} in Section~\ref{sec:computability}.
Technically speaking, randomness then requires that all allowable non-negative supermartingales (that start with unit value) should remain bounded on~\(\pth\).

It is this last, martingale-theoretic, approach that seems to lend itself most easily to allowing for interval rather than precise forecasts, and therefore to allowing for `imprecision' in the definition of randomness.
As we explain in Sections~\ref{sec:single:forecast} and~\ref{sec:forecasting:systems}, an interval, or `imprecise', forecasting system~\(\frcstsystem\) associates with each finite sequence of outcomes~\((\xvaltolong[n])\) a (conditional) expectation \emph{interval} \(\frcstsystem(\xvaltolong[n])\) for the next, as yet unknown, outcome~\(\randomoutcome[n+1]\).
The lower bound of this \emph{interval forecast} represents a supremum acceptable buying price, and its upper bound an infimum acceptable selling price, for the next outcome~\(\randomoutcome[n+1]\).
This idea rests firmly on the common ground between Walley's \cite{walley1991} theory of coherent lower previsions and Shafer and Vovk's \cite{shafer2001,shafer2019:book} game-theoretic approach to probability that we have helped establishing in recent years, through our research on imprecise stochastic processes \cite{cooman2007d,cooman2015:markovergodic}; see also Refs.~\cite{troffaes2013:lp,augustin2013:itip} for more details on so-called `imprecise probabilities'.
These theoretical developments allow us here to associate supermartingales with an interval forecasting system, and therefore in Section~\ref{sec:randomness} to extend a number of existing notions of randomness to allow for interval, rather than precise, forecasts: we include in particular {\ML} randomness and {\comp} randomness \cite{schnorr1971,downey2010,ambosspies2000,bienvenu2009:randomness}.
In Section~\ref{sec:schnorr:randomness}, we also extend Schnorr randomness \cite{schnorr1971,downey2010,ambosspies2000,bienvenu2009:randomness} to allow for interval forecasts.
We then show in Section~\ref{sec:consistency} that our approach allows us to extend to interval forecasting some of Dawid's \cite{dawid1982:well:calibrated:bayesian} well-known work on calibration, and to establish a number of interesting `limiting frequencies' or {\comp} stochasticity results.

We believe the discussion becomes especially interesting in Section~\ref{sec:constantintervalforecasts}, where we start restricting our attention to \emph{constant}, or \emph{stationary}, interval forecasts.
We see this as an extension of the more classical accounts of randomness, which typically consider a forecasting system with constant forecast~\(\nicefrac{1}{2}\)---corresponding to flipping a fair coin.
As we have by now come to expect from our experience with so-called imprecise probability models, when we allow for interval forecasts, a mathematical structure appears that is much more interesting than the rather simpler case of precise forecasts would lead us to suspect.
In the precise case, a given sequence may not be random for any stationary forecast, but as we will see, in the case of interval forecasting there typically is a filter of intervals that a given sequence is random for.
Furthermore, as we show in Section~\ref{sec:non-stationarity} by means of explicit examples, this filter may not have a smallest element, and even when it does, this smallest element may be a non-vanishing interval: this is the first cornerstone for our argument that \emph{randomness is inherently imprecise}.

The examples in Section~\ref{sec:non-stationarity} all involve sequences that are random for some {\comp} non-stationary precise forecast, but can't be random for a stationary forecast unless it becomes interval-valued, or imprecise.
This might lead to the suspicion that this imprecision is perhaps only an artefact, which results from looking at non-stationary phenomena through an imperfect stationary lens.
We show in Section~\ref{sec:inherently} that this suspicion is unfounded: there are sequences that are random for a stationary interval forecast, but that aren't random for \emph{any {\comp} (more) precise} forecast, be it stationary or not.
This further corroborates our claim that \emph{randomness is, indeed, inherently imprecise}.

Finally, in Section~\ref{sec:meagreness}, we argue that `imprecise' randomness is an interesting extension of the existing notions of `precise' randomness, because it is equally rare: just as for precise stationary forecasts, the set of all sequences that are random for a \emph{non-vacuous} stationary interval forecast is \emph{meagre}.
This, we will argue, indicates that the essential distinction lies not between precise and imprecise forecasts (or randomness), but between non-vacuous and vacuous ones, and provides further evidence for the essentially `imprecise' nature of the randomness notion.

We conclude with a short discussion of the significance of our findings, and of possible avenues for further research.
In order to maintain focus, we have decided to move all technical proofs of auxiliary results about computability and growth functions to an appendix.
We have also, as much as possible, tried to make sure that our more complicated and technical proofs in the main text are preceded by informal arguments, in order to help the reader build some intuition about why and how they work.

\section{A single interval forecast}\label{sec:single:forecast}
The dynamics of making a single forecast can be made very clear, after the fashion first introduced by Shafer and Vovk \cite{shafer2001,shafer2019:book}, by considering a simple game, with three players, namely Forecaster, Sceptic and Reality.
The game involves an initially unknown outcome in the set~\(\outcomes\), which we will denote by~\(\randomoutcome\).
To stress that it is unknown, we call it a \emph{variable}, and use upper-case notation.

\begin{game}[Single forecast of an outcome~\(\randomoutcome\)]
In a first step, the first player, \emph{Forecaster}, specifies an interval bound~\(I=\pinterval\subseteq\frcsts\) for the expectation of an as yet unknown outcome~\(\randomoutcome\) in~\(\outcomes\)---or equivalently, for the probability that \(\randomoutcome=1\).
We interpret this so-called \emph{interval forecast}~\(I\) as a commitment, on the part of Forecaster, to adopt \(\lp\) as his \emph{supremum acceptable buying price} and \(\up\) as his \emph{infimum acceptable selling price} for the gamble (with reward function) \(\randomoutcome\).
This is taken to mean that the second player, \emph{Sceptic}, can now in a second step take Forecaster up on any (combination) of the following commitments, whose uncertain pay-offs are expressed in units of a linear utility:
\begin{enumerate}[label=\upshape(\roman*),leftmargin=*,noitemsep,topsep=0pt]
\item for all real~\(q\leq\lp\) and all real~\(\alpha\geq0\), Forecaster is committed to accepting the gamble~\(\alpha[\randomoutcome-q]\), leading to a (possibly negative) uncertain reward~\(-\alpha[\randomoutcome-q]\) for Sceptic;\footnote{Because we allow \(q\leq\lp\) rather than \(q<\lp\), we actually see \(\lp\) as a \emph{maximum} acceptable buying price, rather than a supremum one. We do this because it doesn't affect the conclusions, as we show in the Appendix, but does simplify the mathematics and the discussion somewhat. Similarly for~\(r\geq\up\).}
\item for all real~\(r\geq\up\) and all real~\(\beta\geq0\), Forecaster is committed to accepting the gamble~\(\beta[r-\randomoutcome]\), leading to a (possibly negative) uncertain reward~\(-\beta[r-\randomoutcome]\) for Sceptic.
\end{enumerate}
Finally, in a third step, the third player, \emph{Reality}, determines the value~\(x\) of~\(\randomoutcome\) in~\(\outcomes\), and the corresponding rewards~\(-\alpha[x-q]\) or~\(-\beta[r-x]\) are paid by Forecaster to Sceptic.\qed
\end{game}

Elements~\(x\) of~\(\outcomes\) are called \emph{outcomes}, and elements \(p\) of the real unit interval \(\frcsts\) will serve as (precise) \emph{forecasts}.
We denote by~\(\imprecisefrcsts\) the set of non-empty closed subintervals of the real unit interval \(\frcsts\).
Any element~\(I\) of~\(\imprecisefrcsts\) will serve as an \emph{interval forecast}.
It has a smallest element~\(\min I\) and a greatest element~\(\max I\), so \(I=\sqgroup{\min I,\max I}\).
We will use the generic notation~\(I\) for such an interval forecast, and \(\lp\coloneqq\min I\) and \(\up\coloneqq\max I\) for its lower and upper bounds, respectively.
An interval forecast~\(I=\sqgroup{\lp,\up}\) is of course \emph{precise} when \(\lp=\up\eqqcolon p\), and we will then make no distinction between the singleton interval forecast~\(I=\set{p}\in\imprecisefrcsts\) and the corresponding precise forecast~\(p\in\frcsts\).

After Forecaster announces an interval forecast~\(I\), what Sceptic can do is essentially to try and increase her capital by taking a gamble on the unknown outcome~\(\randomoutcome\).
Any such \emph{gamble} can be considered as a map \(f\colon\outcomes\to\reals\), and can therefore be represented as a point or vector \((f(1),f(0))\) in the two-dimensional vector space \(\reals^2\); see also Figure~\ref{fig:capitalincrease} below.
\(f(\randomoutcome)\) is then the (possibly negative) increase in Sceptic's capital after the game has been played, as a function of the outcome variable \(\randomoutcome\).
Of course, not every gamble~\(f(\randomoutcome)\) on the unknown outcome~\(\randomoutcome\) will be available to Sceptic: which gambles she can take is determined by Forecaster's interval forecast~\(I\).
As we indicated above, in their most general form, they're given by~\(f(\randomoutcome)=-\alpha[\randomoutcome-q]-\beta[r-\randomoutcome]\), where \(\alpha\) and \(\beta\) are non-negative real numbers, \(q\leq\lp\) and \(r\geq\up\).
We see that the gambles that are available to Sceptic constitute a closed convex cone~\(\availables_I\) in~\(\reals^2\), see also Figure~\ref{fig:capitalincrease}:
\begin{equation*}
\availables_I\coloneqq\cset[\big]{-\alpha[\randomoutcome-q]-\beta[r-\randomoutcome]}
{\text{\(q\leq\lp\), \(\up\leq r\) and \(\alpha,\beta\in\nonnegreals\)}},
\end{equation*}
where we use \(\nonnegreals\) to denote the set of non-negative real numbers.

Let us associate with any precise forecast \(p\in\frcsts\) the \emph{expectation} (functional)~\(\ex_p\), defined by
\begin{equation}\label{eq:local:linear}
\ex_p(f)
\coloneqq pf(1)+(1-p)f(0)
\text{ for any gamble~\(f\colon\outcomes\to\reals\).}
\end{equation}
If we also consider the so-called \emph{lower expectation} (functional)~\(\lex_I\) associated with an interval forecast~\(I\in\imprecisefrcsts\), defined by
\begin{multline}\label{eq:local:lower}
\lex_I(f)
\coloneqq\min_{p\in I}\ex_p(f)
=\min_{p\in I}\sqgroup[\big]{pf(1)+(1-p)f(0)}
=
\begin{cases}
\ex_{\lp}(f)&\text{if \(f(1)\geq f(0)\)}\\
\ex_{\up}(f)&\text{if \(f(1)\leq f(0)\)}
\end{cases}\\
\text{ for any gamble~\(f\colon\outcomes\to\reals\)},
\end{multline}
and similarly, the \emph{upper expectation} (functional)~\(\uex_I\), defined by
\begin{multline}\label{eq:local:upper}
\uex_I(f)
\coloneqq\max_{p\in I}\ex_p(f)
=\begin{cases}
\ex_{\up}(f)&\text{if \(f(1)\geq f(0)\)}\\
\ex_{\lp}(f)&\text{if \(f(1)\leq f(0)\)}
\end{cases}
=-\lex_I(-f)\\
\text{ for any gamble~\(f\colon\outcomes\to\reals\)},
\end{multline}
then it is not difficult to see\footnote{Use the characterisation~\(f(\randomoutcome)=-\alpha[\randomoutcome-p]-\beta[q-\randomoutcome]\) of the available gambles derived above, and the properties of the upper expectation~\(\uex_I\) listed in Proposition~\ref{prop:properties:of:expectations}.} that \emph{the closed convex cone~\(\availables_I\) of all gambles~\(f(\randomoutcome)\) that are available to Sceptic after Forecaster announces his interval forecast~\(I\) is completely determined by the condition~\(\uex_I(f)\leq0\)}, as depicted by the blue regions in Figure~\ref{fig:capitalincrease}.
In fact, the condition ~\(\uex_I(f)\leq0\) is equivalent to~\((\forall p\in I)\ex_p(f)\leq0\), so the available gambles belong to the intersection of all half-planes determined by~\(\ex_p(f)\leq0\) for all~\(p\in I\).

The functionals \(\lex_I\) and \(\uex_I\) are easily shown to have the following so-called \emph{coherence properties}, typical for the more general lower and upper expectation operators defined on arbitrary gamble spaces \cite{walley1991,troffaes2013:lp}:

\begin{proposition}\label{prop:properties:of:expectations}
Consider any forecast interval \(I\in\imprecisefrcsts\).
Then for all gambles \(f,g\) on~\(\outcomes\), all~\(\mu\in\reals\) and all non-negative \(\lambda\in\reals\):
\begin{enumerate}[label=\upshape C{\arabic*}.,ref=\upshape C{\arabic*},leftmargin=*,noitemsep,topsep=0pt]
\item\label{axiom:coherence:bounds} \(\min f\leq\lex_I(f)\leq\uex_I(f)\leq\max f\);\upshape\hfill[bounds]
\item\label{axiom:coherence:homogeneity} \(\lex_I(\lambda f)=\lambda\lex_I(f)\) and \(\uex_I(\lambda f)=\lambda\uex_I(f)\);\upshape\hfill[non-negative homogeneity]
\item\label{axiom:coherence:subadditivity} \(\lex_I(f+g)\geq\lex_I(f)+\lex_I(g)\) and \(\uex_I(f+g)\leq\uex_I(f)+\uex_I(g)\);\upshape\hfill[super/subadditivity]
\item\label{axiom:coherence:constantadditivity} \(\lex_I(f+\mu)=\lex_I(f)+\mu\) and \(\uex_I(f+\mu)=\uex_I(f)+\mu\);\upshape\hfill[constant additivity]
\item\label{axiom:coherence:monotonicity} if \(f\leq g\) then \(\lex_I(f)\leq\lex_I(g)\) and \(\uex_I(f)\leq\uex_I(g)\).\upshape\hfill[monotonicity]
\end{enumerate}
\end{proposition}

\begin{figure}[ht]
\centering
\begin{tikzpicture}[scale=1.25]\footnotesize
\coordinate (xaxis) at (1.5,0);
\coordinate (yaxis) at (0,1.5);
\coordinate (origin) at (0,0);
\coordinate (bottomleft) at (-1.5,-1.5);
\coordinate (aboveleft) at (-1.5,0.5);
\coordinate (belowright) at (0.5,-1.5);
\coordinate (gamble) at (-1,-0.5);
\coordinate (label) at (0,-1.7);
\draw[blue,thick,name path=marginal line] (aboveleft) -- node[midway,above,rotate=-18.43] {\footnotesize \(\ex_{\lp}(f)=0\)} (origin);
\draw[help lines,dashed] (origin) -- (1.5,-0.5);
\draw[help lines,dashed] (-0.5,1.5) -- (origin);
\draw[blue,thick] (origin) -- node[midway,above,rotate=-71.57] {\footnotesize \(\ex_{\up}(f)=0\)} (belowright);
\fill[nearly transparent,blue] (origin) -- (belowright) -- (bottomleft) -- (aboveleft) -- cycle;
\draw[step=.5cm,gray,very thin] (-1.4,-1.4) grid (1.4,1.4);
\draw[->] (-1.5,0) -- (xaxis) node[below] {\footnotesize\(f(1)\)};
\draw[->] (0,-1.5) -- (yaxis) node[above] {\footnotesize\(f(0)\)};
\draw[red,very thin] (bottomleft) -- (origin) -- coordinate[midway] (diagonal) (1.5,1.5) ;
\node[red,rotate=45,above=5pt] at (diagonal) {\footnotesize \(f(1)\leq f(0)\)};
\node[red,rotate=45,below=5pt] at (diagonal) {\footnotesize \(f(1)\geq f(0)\)};
\node[below] at (label) {(a)};
\end{tikzpicture}
\qquad\qquad
\begin{tikzpicture}[scale=1.25]\footnotesize
\coordinate (xaxis) at (1.5,0);
\coordinate (yaxis) at (0,1.5);
\coordinate (origin) at (0,0);
\coordinate (bottomleft) at (-1.5,-1.5);
\coordinate (bottomright) at (1.5,-1.5);
\coordinate (aboveleft) at (-1.5,0.5);
\coordinate (belowright) at (1.5,-0.5);
\coordinate (gamble) at (-1,-0.5);
\coordinate (label) at (0,-1.7);
\draw[step=.5cm,gray,very thin] (-1.4,-1.4) grid (1.4,1.4);
\draw[->] (-1.5,0) -- (xaxis) node[below] {\footnotesize\(f(1)\)};
\draw[->] (0,-1.5) -- (yaxis) node[above] {\footnotesize\(f(0)\)};
\draw[red,very thin] (bottomleft) -- (origin) -- coordinate[midway] (diagonal) (1.5,1.5) ;
\node[red,rotate=45,above=5pt] at (diagonal) {\footnotesize \(f(1)\leq f(0)\)};
\node[red,rotate=45,below=5pt] at (diagonal) {\footnotesize \(f(1)\geq f(0)\)};
\fill[nearly transparent,blue] (aboveleft) -- (belowright) -- (bottomright) -- (bottomleft) -- cycle;
\draw[fill=white,blue,thick,name path=marginal line] (aboveleft) -- node[midway,above,rotate=-18.43] {\footnotesize \(\ex_{p}(f)=0\)} (origin) -- (belowright);
\node[below] at (label) {(b)};
\end{tikzpicture}
\caption{Gambles~\(f\) available to Sceptic when (a) Forecaster announces \(I\in\imprecisefrcsts\) with \(\lp<\up\); and when (b) Forecaster announces \(I\in\imprecisefrcsts\) with \(\lp=\up\eqqcolon p\).}
\label{fig:capitalincrease}
\end{figure}

\section{Interval forecasting systems and imprecise probability trees}\label{sec:forecasting:systems}
We now consider a sequence of repeated versions of the forecasting game in the previous section.
At each successive stage~\(k\in\naturals\), Forecaster presents an interval forecast~\(I_k=\pinterval[k]\) for the unknown outcome variable \(\randomoutcome[k]\).
This effectively allows Sceptic to choose any gamble~\(f_k(\randomoutcome[k])\) such that \(\uex_{I_k}(f_k)\leq0\).
Finally, Reality then chooses a value \(\xval[k]\) for~\(\randomoutcome[k]\), resulting in a gain in capital~\(f_k(\xval[k])\) for Sceptic.
This gain~\(f_k(\xval[k])\) can, of course, be negative, resulting in an actual decrease in Sceptic's capital.

Here and in what follows, \(\naturals\) is the set of all natural numbers, without zero.
We will also use the notation \(\naturalswithzero\coloneqq\naturals\cup\set{0}\) and the notation~\(\integers\) for the set of all integer numbers.

\subsection{The event tree and its forecasting systems}
We call~\((\xval[1],\xval[2],\dots,\xval[n],\dots)\) an outcome sequence, and we collect all possible outcome sequences in the set~\(\pths\coloneqq\outcomes^\naturals\).
We collect the finite outcome sequences~\(\xvalto[n]\coloneqq(\xvaltolong[n])\) in the set~\(\sits\coloneqq\outcomes^*=\bigcup_{n\in\naturalswithzero}\outcomes^n\).
The finite outcome sequences~\(\sit\) in~\(\sits\) and infinite outcome sequences~\(\pth\) in~\(\pths\) constitute the nodes---also called \emph{situations}---and \emph{paths} in an event tree with unbounded horizon, part of which is depicted below.
The empty sequence~\(\xvalto[0]\eqqcolon\init\) is also called the \emph{initial} situation.
From now on, we will systematically use the `situations' and `paths' terminology.
Keep in mind that any path~\(\pth\in\pths\) is an infinite outcome sequence, and can therefore also be identified with---the binary expansion of---a real number in the unit interval~\(\frcsts\).
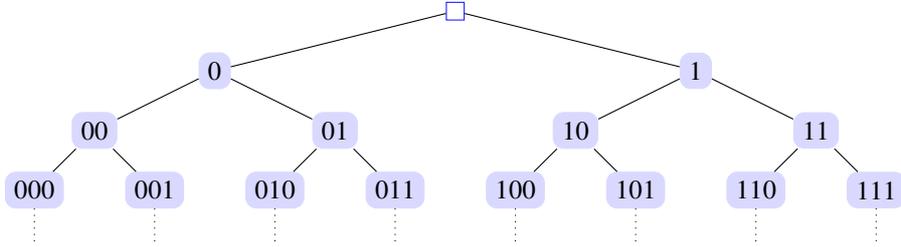
\begin{center}
\begin{tikzpicture}
\tikzstyle{level 1}=[sibling distance=18em]
\tikzstyle{level 2}=[sibling distance=9em]
\tikzstyle{level 3}=[sibling distance=4.5em]
\tikzstyle{level 4}=[level distance=2em]
\node[root] (root) {} [grow=down,level distance=5ex]
child {node[nonterminal] (a) {\(0\)}
child {node[nonterminal] (aa) {\(00\)}
child {node[nonterminal] (aaa) {\(000\)}
child[black,dotted]}
child {node[nonterminal] (aab) {\(001\)}
child[black,dotted]}
}
child {node[nonterminal] (ab) {\(01\)}
child {node[nonterminal] (aba) {\(010\)}
child[black,dotted]}
child {node[nonterminal] (abb) {\(011\)}
child[black,dotted]}
}
}
child {node[nonterminal] (b) {\(1\)}
child {node[nonterminal] (ba) {\(10\)}
child {node[nonterminal] (baa) {\(100\)}
child[black,dotted]}
child {node[nonterminal] (bab) {\(101\)}
child[black,dotted]}
}
child {node[nonterminal] (bb) {\(11\)}
child {node[nonterminal] (bba) {\(110\)}
child[black,dotted]}
child {node[nonterminal] (bbb) {\(111\)}
child[black,dotted]}
}
};
\end{tikzpicture}
\end{center}

In the repeated game described above, Forecaster will only provide interval forecasts~\(I_k\) after observing the actual sequence~\((\xvaltolong[k-1])\) that Reality has chosen, and the corresponding sequence of gambles~\((f_1,\dots,f_{k-1})\) that Sceptic has chosen.
This is the essence of so-called prequential forecasting \cite{dawid1982:well:calibrated:bayesian,dawid1984,dawid1999}.
But for the purposes of the present discussion, it will be advantageous to consider an alternative, and in some aspects more involved, setting where a forecast~\(I_\sit\) is specified in each of the possible situations~\(\sit\) in the event tree~\(\sits\); see the figure below:
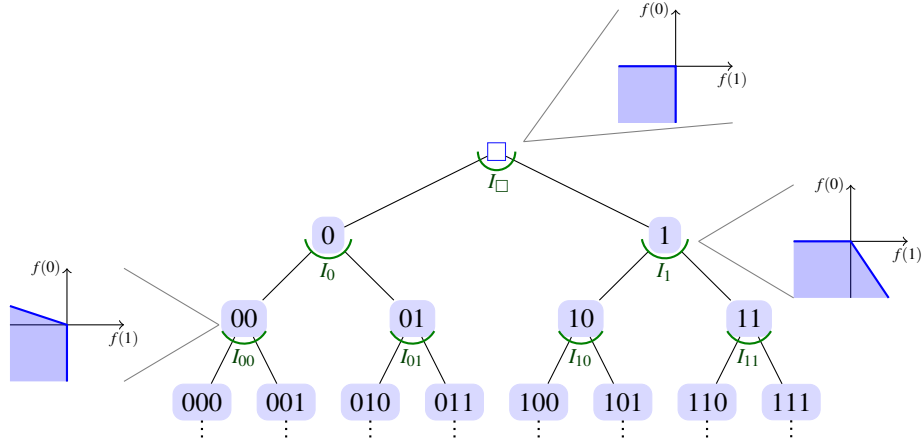
\begin{center}
\begin{tikzpicture}[scale=.7]
\tikzstyle{level 1}=[sibling distance=18em]
\tikzstyle{level 2}=[sibling distance=9em]
\tikzstyle{level 3}=[sibling distance=4.5em]
\tikzstyle{level 4}=[level distance=2em]
\node[root] (root) {} [grow=down,level distance=10ex]
child {node[nonterminal] (a) {\(0\)}
child {node[nonterminal] (aa) {\(00\)}
child {node[nonterminal] (aaa) {\(000\)}
child[thick,black,dotted]}
child {node[nonterminal] (aab) {\(001\)}
child[thick,black,dotted]}
}
child {node[nonterminal] (ab) {\(01\)}
child {node[nonterminal] (aba) {\(010\)}
child[thick,black,dotted]}
child {node[nonterminal] (abb) {\(011\)}
child[thick,black,dotted]}
}
}
child {node[nonterminal] (b) {\(1\)}
child {node[nonterminal] (ba) {\(10\)}
child {node[nonterminal] (baa) {\(100\)}
child[thick,black,dotted]}
child {node[nonterminal] (bab) {\(101\)}
child[thick,black,dotted]}
}
child {node[nonterminal] (bb) {\(11\)}
child {node[nonterminal] (bba) {\(110\)}
child[thick,black,dotted]}
child {node[nonterminal] (bbb) {\(111\)}
child[thick,black,dotted]}
}
};
\draw[local,thick] (root) +(180:1em) arc (180:360:1em);
\draw[local,thick] (b) +(190:1.25em) arc (190:350:1.25em);
\draw[local,thick] (a) +(190:1.25em) arc (190:350:1.25em);
\draw[local,thick] (bb) +(210:1.35em) arc (210:330:1.35em);
\draw[local,thick] (ba) +(210:1.35em) arc (210:330:1.35em);
\draw[local,thick] (ab) +(210:1.35em) arc (210:330:1.35em);
\draw[local,thick] (aa) +(210:1.35em) arc (210:330:1.35em);
\path (root) +(275:1.7em) node[local] {\footnotesize\(I_\init\)};
\path (a) +(270:2em) node[local] {\footnotesize\(I_{0}\)};
\path (b) +(270:2em) node[local] {\footnotesize\(I_{1}\)};
\path (aa) +(270:2.1em) node[local] {\footnotesize\(I_{00}\)};
\path (bb) +(270:2.1em) node[local] {\footnotesize\(I_{11}\)};
\path (ba) +(270:2.1em) node[local] {\footnotesize\(I_{10}\)};
\path (ab) +(270:2.1em) node[local] {\footnotesize\(I_{01}\)};
\node[right = 1mm of root,anchor=south west] {%
\begin{tikzpicture}[scale=.5,anchor=south west]\tiny
\coordinate (xaxis) at (1.5,0);
\coordinate (yaxis) at (0,1.5);
\coordinate (origin) at (0,0);
\coordinate (bottomleft) at (-1.5,-1.5);
\coordinate (bottomright) at (1.5,-1.5);
\coordinate (topright) at (1.5,1.5);
\coordinate (topleft) at (-1.5,1.5);
\coordinate (aboveleft) at (-1.5,0);
\coordinate (belowright) at (0,-1.5);
\coordinate (focus) at (-4,-2);
\draw[->] (-1.5,0) -- (xaxis) node[below] {\(f(1)\)};
\draw[->] (0,-1.5) -- (yaxis) node[left] {\(f(0)\)};
\fill[nearly transparent,blue] (aboveleft) -- (bottomleft) -- (belowright) -- (origin) -- cycle;
\draw[blue,thick,name path=marginal line] (aboveleft) -- (origin) -- (belowright);
\draw[gray,thin] (focus) -- (topleft);
\draw[gray,thin] (focus) -- (bottomright);
\end{tikzpicture}%
};
\node[right = 1mm of b] {%
\begin{tikzpicture}[scale=.5]\tiny
\coordinate (xaxis) at (1.5,0);
\coordinate (yaxis) at (0,1.5);
\coordinate (origin) at (0,0);
\coordinate (bottomleft) at (-1.5,-1.5);
\coordinate (bottomright) at (1.5,-1.5);
\coordinate (topright) at (1.5,1.5);
\coordinate (topleft) at (-1.5,1.5);
\coordinate (aboveleft) at (-1.5,0);
\coordinate (belowright) at (1,-1.5);
\coordinate (focus) at (-4,0);
\draw[->] (-1.5,0) -- (xaxis) node[below] {\(f(1)\)};
\draw[->] (0,-1.5) -- (yaxis) node[left] {\(f(0)\)};
\fill[nearly transparent,blue] (aboveleft) -- (origin) -- (belowright) -- (bottomleft) -- cycle;
\draw[blue,thick,name path=marginal line] (aboveleft) -- (origin) -- (belowright);
\draw[gray,thin] (focus) -- (topleft);
\draw[gray,thin] (focus) -- (bottomleft);
\end{tikzpicture}%
};
\node[left = -1mm of aa] {%
\begin{tikzpicture}[scale=.5]\tiny
\coordinate (xaxis) at (1.5,0);
\coordinate (yaxis) at (0,1.5);
\coordinate (origin) at (0,0);
\coordinate (bottomleft) at (-1.5,-1.5);
\coordinate (bottomright) at (1.5,-1.5);
\coordinate (topright) at (1.5,1.5);
\coordinate (aboveleft) at (-1.5,0.5);
\coordinate (belowright) at (0,-1.5);
\coordinate (focus) at (4,0);
\draw[->] (-1.5,0) -- (xaxis) node[below] {\(f(1)\)};
\draw[->] (0,-1.5) -- (yaxis) node[left] {\(f(0)\)};
\fill[nearly transparent,blue] (aboveleft) -- (origin) -- (belowright) -- (bottomleft) -- cycle;
\draw[blue,thick,name path=marginal line] (aboveleft) -- (origin) -- (belowright);
\draw[gray,thin] (focus) -- (topright);
\draw[gray,thin] (focus) -- (bottomright);
\end{tikzpicture}%
};
\end{tikzpicture}
\end{center}
We can use this idea to extend the notion of a forecasting system in Refs.~\cite{dawid1985:cbep,vovk2010:randomness} from precise to interval forecasts.

\begin{definition}[Forecasting system]
A \emph{forecasting system} is a map \(\frcstsystem\colon\sits\to\imprecisefrcsts\), that associates an interval forecast~\(\frcstsystem(\sit)\in\imprecisefrcsts\) with every situation~\(\sit\) in the event tree~\(\sits\).
With any forecasting system~\(\frcstsystem\) we associate two real processes~\(\lfrcstsystem\) and \(\ufrcstsystem\),\footnote{For a more concrete definition of a `process', we refer to the discussion in Section~\ref{sec:trees:and:processes}.} defined by~\(\lfrcstsystem(\sit)\coloneqq\min\frcstsystem(\sit)\) and \(\ufrcstsystem(\sit)\coloneqq\max\frcstsystem(\sit)\) for all~\(\sit\in\sits\).
A forecasting system~\(\frcstsystem\) is called \emph{precise} if \(\lfrcstsystem=\ufrcstsystem\).
\end{definition}
\noindent
Specifying a forecasting system~\(\frcstsystem\) requires that Forecaster should imagine in advance all the moves that Reality (and Sceptic) could make, and that he should devise in advance what forecast~\(\frcstsystem(\sit)\) to give in each imaginable situation~\(\sit\in\sits\).

We will use the notation \(\frcstsystem\subseteq\frcstsystem^*\) to mean that the forecasting system~\(\frcstsystem^*\) is \emph{at least as conservative} as \(\frcstsystem\), meaning that \(\frcstsystem(\sit)\subseteq\frcstsystem^*(\sit)\) for all~\(\sit\in\sits\).

\subsection{Imprecise probability trees and supermartingales}\label{sec:trees:and:processes}
Since in each situation~\(\sit\) the interval forecast~\(I_\sit=\frcstsystem(\sit)\) corresponds to a so-called \emph{local} upper expectation~\(\uex_{I_\sit}\), we can use the argumentation in our earlier papers \cite{cooman2007d,cooman2015:ergodic,cooman2015:markovergodic} on imprecise stochastic processes to help \(\frcstsystem\) turn the event tree into an \emph{imprecise probability tree}, with an associated \emph{global} upper expectation on paths, and a corresponding notion of `almost surely'.

In what follows, we recall in some detail how to do this.
However, we will limit ourselves to discussing only those aspects that are essential for a proper understanding of our treatment of randomness further on; for a much more extensive discussion, we refer to our earlier papers \cite{cooman2007d,cooman2015:ergodic,cooman2015:markovergodic}, based on the seminal work by Shafer and Vovk \cite{shafer2001,shafer2019:book,shafer2012:zero-one,vovk2014:itip}.

We will denote by~\(\frcstsystems\) the set~\(\imprecisefrcsts^{\sits}\) of all forecasting systems, or equivalently, all imprecise probability trees.

For any path~\(\pth\in\pths\), the initial sequence that consists of its first \(n\) elements is a situation in~\(\outcomes^n\) that is denoted by~\(\pthto{n}\).
Its \(n\)-th element belongs to~\(\outcomes\) and is denoted by~\(\pth_n\).
As a convention, we let its \(0\)-th element be the \emph{initial} situation~\(\pthto{0}=\pth_0=\init\).

For any situation~\(\sit\in\sits\) and any path~\(\pth\in\pths\), we say that \(\pth\) \emph{goes through}~\(\sit\) if there is some~\(n\in\naturalswithzero\) such that \(\pthto{n}=\sit\).
We denote by~\(\exact{\sit}\) the so-called \emph{cylinder set} of all paths~\(\pth\in\pths\) that go through~\(\sit\).

We write that \(\sit\precedes\altsit\), and say that the situation~\(\sit\) \emph{precedes} the situation~\(\altsit\), when every path that goes through~\(\altsit\) also goes through~\(\sit\)---so \(\sit\) is a precursor of~\(\altsit\).
An equivalent condition is of course that \(\exact{\altsit}\subseteq\exact{\sit}\).
We say that the situation~\(\sit\) \emph{strictly precedes} the situation~\(\altsit\), and write \(\sit\sprecedes\altsit\), when \(\sit\precedes\altsit\) and \(\sit\neq\altsit\), or equivalently, when \(\exact{\altsit}\subset\exact{\sit}\).

For any situation~\(\sit=(x_1,\dots,x_n)\in\sits\), we call~\(n=\dist{\sit}\) its depth in the tree.
Of course, \(\dist{\sit}\geq\dist{\init}=0\).
We will use a similar notational convention for situations as for paths: we let \(\sitat{k}\coloneqq\xval[k]\) and \(\sitto{k}\coloneqq(\xvaltolong[k])\) for all~\(k\in\set{1,\dots,n}\), and \(\sitto{0}=\sitat{0}\coloneqq\init\).
Also, for any~\(\xval\in\outcomes\), we denote by~\(\sit x\) the situation~\((\xvaltolong[n],x)\).

A \emph{process}~\(\process\) is a map defined on~\(\sits\).
A \emph{real process} is a real-valued process: it associates a real number~\(\process(\sit)\in\reals\) with every situation~\(\sit\in\sits\).
With any real process~\(\process\), we can always associate a process~\(\adddelta\process\), called the \emph{process difference}.
For every situation~\(\sit\in\sits\), \(\adddelta\process(\sit)\) is the gamble on~\(\outcomes\) defined by
\begin{equation*}
\adddelta\process(\sit)(\xval)
\coloneqq\process(\sit x)
-\process(\sit)
\text{ for all~\(\xval\in\outcomes\)}.
\end{equation*}
The \emph{initial value} of a process~\(\process\) is its value~\(\process(\init)\) in the initial situation~\(\init\).
Any real process is completely determined by its initial value and its process difference, because
\begin{equation*}
\process(\xvaltolong[n])=\process(\init)+\sum_{k=0}^{n-1}\adddelta\process(\xvaltolong[k])(\xval[k+1])
\text{ for all \((\xvaltolong[n])\in\sits\)}.
\end{equation*}
We call a real process \emph{non-negative} if it is non-negative in all situations.
Similarly, a \emph{positive} real process is (strictly) positive in all situations.
We call \emph{test process} any non-negative real process~\(\process\) with unit initial value~\(\process(\init)=1\).

We now look at at number of special real processes.
In the imprecise probability tree associated with a \emph{given} forecasting system~\(\frcstsystem\), a \emph{supermartingale}~\(\supermartin\) for~\(\frcstsystem\) is a real process such that
\begin{equation}\label{eq:supermartingale}
\uex_{\frcstsystem(\sit)}(\adddelta\supermartin(\sit))\leq0,
\text{ or equivalently, }
\uex_{\frcstsystem(\sit)}(\supermartin(\sit\cdot))\leq\supermartin(s),
\text{ for all~\(\sit\in\sits\)}.
\end{equation}
In other words, all supermartingale differences have non-positive upper expectation: supermartingales are real processes that Forecaster expects to decrease.
A real process~\(\submartin\) is a \emph{submartingale} for~\(\frcstsystem\) if \(-\submartin\) is a supermartingale, which means that \(\lex_{\frcstsystem(\sit)}(\adddelta\submartin(\sit))\geq0\) for all~\(\sit\in\sits\): all submartingale differences have non-negative lower expectation, so submartingales are real processes that Forecaster expects to increase.
We denote the set of all supermartingales for a given forecasting system~\(\frcstsystem\) by~\(\supermartins[\frcstsystem]\)---whether a real process is a supermartingale depends of course on the forecasts in the situations.
Similarly, the set~\(\submartins[\frcstsystem]\coloneqq-\supermartins[\frcstsystem]\) is the set of all submartingales for~\(\frcstsystem\), and \(\martins\coloneqq\submartins[\frcstsystem]\cap\supermartins[\frcstsystem]\) is the set of all \emph{martingales} for~\(\frcstsystem\)---real processes that are at the same time super- and submartingales, and therefore real processes that Forecaster expects to remain constant.
\par
It ought to be clear from the discussion in Section~\ref{sec:single:forecast} that the supermartingales for~\(\frcstsystem\) are effectively all the possible \emph{capital processes}~\(\supermartin\) for a Sceptic who starts with an initial capital~\(\supermartin(\init)\), and in each possible subsequent situation~\(\sit\) selects a gamble~\(f_\sit=\adddelta\supermartin(\sit)\) that is available there because of Forecaster's specification of the interval forecast~\(I_\sit=\frcstsystem(\sit)\): \(\uex_{I_\sit}(f_\sit)\leq0\).
If Reality chooses the successive outcomes~\(\xvaltolong[n]\), then Sceptic will end up in the corresponding situation~\(\sit=(\xvaltolong[n])\) with a capital
\begin{equation*}
\supermartin(\xvaltolong[n])
=\supermartin(\init)+\sum_{k=0}^{n-1}\adddelta\supermartin(\xvaltolong[k])(\xval[k+1])
=\supermartin(\init)+\sum_{k=0}^{n-1}f_{(\xvaltolong[k])}(\xval[k+1]).
\end{equation*}
\par
We call \emph{test supermartingale} for~\(\frcstsystem\) any test process that is also a supermartingale for~\(\frcstsystem\), or in other words, any non-negative supermartingale~\(\supermartin\) for~\(\frcstsystem\) with initial value~\(\supermartin(\init)=1\).
It corresponds to Sceptic starting with unit capital and never borrowing.
We collect all test supermartingales for~\(\frcstsystem\) in the set~\(\testsupermartins[\frcstsystem]\).
\par
We will also need to pay attention to a particular way of constructing test supermartingales.
We define a \emph{gamble process} as a map~\(\multprocess\) from~\(\sits\) to gambles on~\(\outcomes\).
If these gambles~\(\multprocess(\sit)\) are all \emph{non-negative}, then we call this~\(\multprocess\) a \emph{multiplier process}.
Given such a multiplier process~\(\multprocess\), we can construct the test process~\(\mint\) by the recursion equation
\begin{equation*}
\mint(\init)\coloneqq1
\text{ and }
\mint(\sit x)\coloneqq\mint(\sit)\multprocess(\sit)(x)
\text{ for all~\(\sit\in\sits\) and \(x\in\outcomes\)},
\end{equation*}
or equivalently by letting \(\mint(\xvaltolong[n])\coloneqq\prod_{k=0}^{n-1}\multprocess(\xvaltolong[k])(x_{k+1})\) for all~\(n\in\naturalswithzero\) and \((\xvaltolong[n])\in\sits\).
We call~\(\mint\) the test process \emph{generated by} the multiplier process~\(\multprocess\).

Any multiplier process~\(\multprocess\) that satisfies the additional condition that \(\uex_{\frcstsystem(\sit)}(\multprocess(\sit))\leq1\) for all~\(\sit\in\sits\), is called a \emph{supermartingale multiplier} for the forecasting system~\(\frcstsystem\).
It is easy to see that the test process~\(\mint\) generated by~\(\multprocess\) is then a test supermartingale for~\(\frcstsystem\): it suffices to check that
\begin{multline}\label{eq:supermartingale:multiplier:differences}
\adddelta\mint(\sit)
=\mint(\sit)[\multprocess(\sit)-1]
\text{ and therefore }
\uex_{\frcstsystem(\sit)}(\adddelta\mint(\sit))
=\mint(\sit)\sqgroup[\big]{\uex_{\frcstsystem(\sit)}(\multprocess(\sit))-1}\\
\text{ for all~\(\sit\in\sits\)},
\end{multline}
due to the coherence properties~\ref{axiom:coherence:homogeneity} and~\ref{axiom:coherence:constantadditivity} of upper expectation operators.

\subsection{Upper expectations and null events}
In the context of (imprecise) probability trees, we call \emph{variable} any map defined on the so-called \emph{sample space}---the set~\(\pths\) of all paths.
When this variable is real-valued and bounded, we call it a \emph{gamble} on~\(\pths\), or also a \emph{global gamble}.
An \emph{event}~\(A\) in this context is a subset of~\(\pths\), and its \emph{indicator}~\(\ind{A}\) is the gamble on~\(\pths\) that assumes the value~\(1\) on~\(A\) and \(0\) elsewhere.

The sub- and supermartingales for a forecasting system~\(\frcstsystem\) can be used to associate so-called \emph{global} lower and upper expectation operators---defined on global gambles---with the forecasting system~\(\frcstsystem\):
\begin{align}
\lglobal(g)
\coloneqq&\sup\cset[\big]{\submartin(\init)}
{\submartin\in\submartins[\frcstsystem]
\text{ and }
\limsup\submartin(\pth)\leq g(\pth)
\text{ for all~\(\pth\in\pths\)}}
\label{eq:tree:lower:expectation}\\
\uglobal(g)
\coloneqq&\inf\cset[\big]{\supermartin(\init)}
{\supermartin\in\supermartins[\frcstsystem]
\text{ and }
\liminf\supermartin(\pth)\geq g(\pth)
\text{ for all~\(\pth\in\pths\)}}
\label{eq:tree:upper:expectation}
\end{align}
for all gambles~\(g\) on~\(\pths\).
In these expressions, we have used the notations
\begin{equation*}
\liminf\supermartin(\pth)\coloneqq\liminf_{n\to\infty}\supermartin(\pthto{n})
\text{ and }
\limsup\supermartin(\pth)\coloneqq\limsup_{n\to\infty}\supermartin(\pthto{n})
\text{ for all~\(\pth\in\pths\)}.
\end{equation*}
It is clear that lower and upper expectations are related to each other through the following \emph{conjugacy relationship}:
\begin{equation}\label{eq:conjugacy}
\lglobal(g)=-\uglobal(-g)
\text{ for all gambles~\(g\) on~\(\pths\)}.
\end{equation}
These lower and upper expectations satisfy coherence properties that are completely similar to---direct counterparts of---those in Proposition~\ref{prop:properties:of:expectations}, and we list these properties again below.
Their proofs are by now fairly well-known \cite{shafer2001,shafer2019:book,tjoens2019:global}, but for the sake of completeness, we repeat them in the Appendix.

\begin{proposition}\label{prop:properties:of:global:expectations}
Consider any forecasting system \(\frcstsystem\in\frcstsystems\).
Then for all gambles \(f,g\) on~\(\pths\), all~\(\mu\in\reals\) and all non-negative \(\lambda\in\reals\):
\begin{enumerate}[label=\upshape E{\arabic*}.,ref=\upshape E{\arabic*},leftmargin=*,noitemsep,topsep=0pt]
\item\label{axiom:lower:upper:bounds} \(\inf f\leq\lglobal(f)\leq\uglobal(f)\leq\sup f\);
\item\label{axiom:lower:upper:homogeneity} \(\lglobal(\lambda f)=\lambda\lglobal(f)\) and \(\uglobal(\lambda f)=\lambda\uglobal(f)\);
\item\label{axiom:lower:upper:subadditivity} \(\lglobal(f)+\lglobal(g)\leq\lglobal(f+g)\leq\lglobal(f)+\uglobal(g)\leq\uglobal(f+g)\leq\uglobal(f)+\uglobal(g)\);
\item\label{axiom:lower:upper:constantadditivity} \(\lglobal(f+\mu)=\lglobal(f)+\mu\) and \(\uglobal(f+\mu)=\uglobal(f)+\mu\);
\item\label{axiom:lower:upper:monotonicity} if \(f\leq g\) then \(\lglobal(f)\leq\lglobal(g)\) and \(\uglobal(f)\leq\uglobal(g)\).
\end{enumerate}
\end{proposition}

For extensive discussion about why the expressions~\eqref{eq:tree:lower:expectation} and~\eqref{eq:tree:upper:expectation} are interesting and useful, we refer to Refs.~\cite{cooman2007d,cooman2015:markovergodic,shafer2001,shafer2019:book,tjoens2019:continuity:arxiv,tjoens2019:global,TJOENS202130,TJoens2021Equivalence}.
For our present purposes, it may suffice to mention that for precise forecasts, they lead to models that coincide with the ones found in measure-theoretic probability theory; see Refs.~\cite[Chapter~8]{shafer2001} and~\cite[Chapter~9]{shafer2019:book}, as well as Ref.~\cite{TJOENS202130}.
In particular, when all~\(I_\sit\) equal~\(\set{\nicefrac{1}{2}}\), these models coincide on all measurable global gambles with the usual uniform (Lebesgue) expectations.
More generally, for an imprecise forecast~\(\frcstsystem\in\frcstsystems\), the lower and upper expectation~\(\lglobal\) and~\(\uglobal\) provide tight lower and upper bounds on the measure-theoretic expectation of every precise forecasting system \(\frcstsystem'\) that is compatible with \(\frcstsystem\), in the sense that \(\frcstsystem'\subseteq\frcstsystem\)~\cite{TJoens2021Equivalence}.

For an event~\(A\subseteq\pths\), the corresponding lower and upper probabilities are defined by \(\lglobalprob(A)\coloneqq\lglobal(\ind{A})\) and \(\uglobalprob(A)\coloneqq\uglobal(\ind{A})\).
The following \emph{conjugacy relationship for events} follows at once from the property~\ref{axiom:lower:upper:constantadditivity} for global lower and upper expectations:
\begin{equation*}
\lglobalprob(A)=1-\uglobalprob(A^c)
\text{ for all~\(A\subseteq\pths\)},
\end{equation*}
where \(A^c\coloneqq\pths\setminus A\) is the complement of~\(A\).

We call an event \(A\subseteq\pths\) \emph{null} for a forecasting system~\(\frcstsystem\) if \(\uglobalprob(A)=0\), or equivalently, if \(\lglobalprob(A^c)=1\).
As usual, any property that holds, except perhaps on a null event, is said to hold \emph{almost surely} for the forecasting system~\(\frcstsystem\).
We will then also say that \emph{almost all paths have that property in the imprecise probability tree corresponding to~\(\frcstsystem\)}.

\section{Basic {\compy} results}\label{sec:computability}
We now give a brief survey of a number of basic notions and results from {\compy} theory, and a few derived results, that are relevant to the developments in this paper.
For a much more extensive discussion, we refer, for instance, to Refs.~\cite{pourel1989,li1993}.

\subsection{Basic definitions}
A \emph{recursive} map~\(\psi\colon\naturalswithzero\to\naturalswithzero\) is a map that can be computed by a Turing machine.
By the Church--Turing (hypo)thesis, this is equivalent to the existence of an algorithm that, upon input of a number \(n\in\naturalswithzero\), outputs the number \(\psi(n)\in\naturalswithzero\).
All notions of {\compy} that we will need are based on this notion, and we will use the equivalent condition consistently.
It is clear that in this definition, we can replace any of the \(\naturalswithzero\) with any other countable set that is linked with~\(\naturalswithzero\) through a recursive bijection whose inverse is also recursive.

We start with the definition of a {\comp} real number.
We call a sequence of rational numbers~\(r_n\) \emph{recursive} if there are three recursive maps \(a,b,\varsigma\) from~\(\naturalswithzero\) to~\(\naturalswithzero\) such that
\begin{equation*}
b(n)>0
\text{ and }
r_n=(-1)^{\varsigma(n)}\frac{a(n)}{b(n)}
\text{ for all~\(n\in\naturalswithzero\)},
\end{equation*}
and we say that it \emph{converges effectively} to a real number~\(x\) if there is some recursive map~\(e\colon\naturalswithzero\to\naturalswithzero\) such that
\begin{equation*}
n\geq e(N)\then\abs{r_n-x}\leq2^{-N}
\text{ for all~\(n,N\in\naturalswithzero\)}.
\end{equation*}
A real number is then called \emph{{\comp}} if there is some recursive sequence of rational numbers that converges effectively to it.
Of course, every rational number is a {\comp} real.

We also need a notion of {\comp} real processes, or in other words, {\comp} real-valued maps \(\process\colon\sits\to\reals\) defined on the set~\(\sits\) of all situations.
Because there is an obvious recursive bijection between \(\naturalswithzero\) and \(\sits\), whose inverse is also recursive, we can identify real processes and real sequences, and simply import, {\itshape mutatis mutandis}, the definitions for {\comp} real sequences common in the literature \cite[Chapter~0, Definition~5]{pourel1989}.
We call a net of rational numbers~\(r_{\sit,n}\) \emph{recursive} if there are three recursive maps \(a,b,\varsigma\) from~\(\sits\times\naturalswithzero\) to~\(\naturalswithzero\) such that
\begin{equation*}
b(\sit,n)>0
\text{ and }
r_{\sit,n}=(-1)^{\varsigma(\sit,n)}\frac{a(\sit,n)}{b(\sit,n)}
\text{ for all~\(\sit\in\sits\) and \(n\in\naturalswithzero\)}.
\end{equation*}
We call a real process~\(\process\colon\sits\to\reals\) \emph{{\comp}} if there is a recursive net of rational numbers~\(r_{\sit,n}\) and a recursive map~\(e\colon\sits\times\naturalswithzero\to\naturalswithzero\) such that
\begin{equation*}
n\geq e(\sit,N)\then\abs{r_{\sit,n}-\process(\sit)}\leq2^{-N}
\text{ for all~\(\sit\in\sits\) and \(n,N\in\naturalswithzero\)}.
\end{equation*}
Again, there is no problem with the notions `recursive net of rational numbers' or `recursive map' in this definition, because we can identify~\(\sits\times\naturalswithzero\) with~\(\naturalswithzero\) through a recursive bijection whose inverse is also recursive.

Obviously, it follows from this definition that in particular \(\process(\altsit)\) is a {\comp} real number for any~\(\altsit\in\sits\): fix \(\sit=\altsit\) and consider the sequence~\(r_{\altsit,n}\), which converges effectively to the real number~\(\process(\altsit)\) as \(n\to\infty\).
Also, a constant real process is {\comp} if and only if its constant real value is.

We also need to mention {\scomp} real processes; see for instance \cite{schnorr1971,li1993} for more details.
A real process~\(\process\) is \emph{\lscomp} if it can be approximated from below by a recursive net of rational numbers, meaning that there is some recursive net of rational numbers~\(r_{\sit,n}\) such that
\begin{enumerate}[label=\upshape(\roman*),leftmargin=*,noitemsep,topsep=0pt]
\item \(r_{\sit,n+1}\geq r_{\sit,n}\) for all~\(\sit\in\sits\) and \(n\in\naturalswithzero\);
\item \(\process(\sit)=\lim_{n\to\infty}r_{\sit,n}\) for all~\(\sit\in\sits\).
\end{enumerate}
We say that \(\process\) is \emph{\uscomp} if \(-\process\) is {\lscomp}.
A real number~\(x\) is {\lscomp} if the real process with constant value~\(x\) is, or equivalently, if there is some recursive sequence of rational numbers~\(r_n\) such that \(r_n\nearrow x\).

In the (semi){\compy} definitions above, as well as in the results that follow in this section, we can replace the countable set~\(\sits\) with any countable set that can be identified with~\(\sits\) through a recursive bijection whose inverse is also recursive.
Further on in this paper, we will for instance have occasion to replace \(\sits\) with~\(\naturals\), \(\naturalswithzero\) and \(\sits\times\outcomes\).

\subsection{Basic results from the literature}
We recall the following standard results; see for instance Ref.~\cite[Chapter~0]{pourel1989}.
The following propositions apply {\itshape mutatis mutandis} also to {\comp} real numbers and {\comp} real sequences in lieu of {\comp} real processes.
Even though they're fairly standard, we give their proofs in the Appendix for the sake of completeness, and to give the reader an idea of why they work.

The condition for {\compy} of a real process can be simplified as follows.

\begin{proposition}[{\protect\cite[Chapter~0, Definition~5a]{pourel1989}}]\label{prop:computable:simplified}
A real process~\(\process\) is {\comp} if and only if there is some recursive net of rational numbers~\(r_{\sit,n}\) such that \(\abs{r_{\sit,n}-\process(\sit)}\leq2^{-n}\) for all~\(\sit\in\sits\) and \(n\in\naturalswithzero\).
\end{proposition}

If \(\process\) and \(\processtoo\) are {\comp} real processes, then so are \(-\process\), \(\process+\processtoo\), \(\process\processtoo\), \(\process/\processtoo\) (provided that \(\processtoo(\sit)\neq0\) for all~\(\sit\in\sits\)), \(\max\set{\process,\processtoo}\), \(\min\set{\process,\processtoo}\), \(\exp(\process)\), \(\ln\process\) (provided that \(\process(\sit)>0\) for all~\(\sit\in\sits\)), and \(F^{\frac{1}{m}}\) for all~\(m\in\naturals\) (provided that \(\process(\sit)\geq0\) for all~\(\sit\in\sits\)); see for instance Ref.~\cite[Chapter~0, Section~2]{pourel1989}.

{\Compy} can be related to lower and upper {\scompy}.

\begin{proposition}\label{prop:computable:upper:lower}
A real process~\(\process\) is {\comp} if and only if it is both lower and upper {\scomp}.
\end{proposition}

The set of all (lower or upper semi){\comp} processes is countable; see for instance Ref.~\cite[Lemma 13]{vovk2010:randomness}.

\subsection{New material for the present context}
We conclude this section with a number of new definitions and results that are specifically tailored to the discussion further on.

The following definitions should be obvious.

A gamble~\(f\) on~\(\outcomes\) is called \emph{{\comp}} if both its values \(f(0)\) and \(f(1)\) are {\comp} real numbers.

An interval forecast~\(I=\pinterval\in\imprecisefrcsts\) is called \emph{{\comp}} if and only if both its lower bound~\(\lp\) and upper bound~\(\up\) are {\comp} real numbers.

A forecasting system~\(\frcstsystem\) is called \emph{{\comp}} if the associated real processes~\(\lfrcstsystem\) and \(\ufrcstsystem\) are {\comp}.

Finally, a process difference \(\adddelta\process\) is called (lower/upper semi)\emph{{\comp}} if the real processes~\(\adddelta\process(\cdot)(0)\) and \(\adddelta\process(\cdot)(1)\) are (lower/upper semi)\emph{{\comp}}; and similarly for a multiplier process~\(\multprocess\).\footnote{These definitions can also be seen as special cases of a more general (lower/upper semi)\emph{{\compy}} condition, where the set~\(\sits\) is replaced by the set~\(\sits\times\outcomes\).}

We also list a number of useful propositions that are less immediate, and perhaps require explicit proofs.
We have gathered these proofs in the Appendix.

\begin{proposition}\label{prop:IcompIffGammaComp}
For any~\(I=\pinterval\in\imprecisefrcsts\), the so-called stationary forecasting system~\(\constantfrcstsystem[I]\), defined by \(\constantfrcstsystem[I](\sit)\coloneqq I\) for all~\(\sit\in\sits\), is {\comp} if and only if the interval \(I\) is {\comp}, and therefore if and only if~\(\lp\) and~\(\up\) are.
\end{proposition}

\begin{proposition}\label{prop:computable:from:delta}
Consider any real process~\(\process\), and its process difference \(\adddelta\process\).
Then the following statements hold:
\begin{enumerate}[label=\upshape(\roman*),leftmargin=*,noitemsep,topsep=0pt]
\item if \(\process(\init)\) and \(\adddelta\process\) are {\lscomp} then so is \(\process\);
\item if \(\process(\init)\) and \(\adddelta\process\) are {\uscomp} then so is \(\process\);
\item \(\process\) is {\comp} if and only if \(\process(\init)\) and \(\adddelta\process\) are.
\end{enumerate}
\end{proposition}

\begin{proposition}\label{prop:computable:from:multiplier}
Consider any multiplier process~\(\multprocess\), then the following implications hold:
\begin{enumerate}[label=\upshape(\roman*),leftmargin=*,noitemsep,topsep=0pt]
\item if \(\multprocess\) is {\lscomp}, then so is \(\mint\);
\item if \(\multprocess\) is {\uscomp}, then so is \(\mint\);
\item if \(\multprocess\) is {\comp}, then so are \(\mint\) and \(\adddelta\mint\).
\end{enumerate}
\end{proposition}

\begin{proposition}\label{prop:computablemultiplier:from:process}
Consider a multiplier process~\(\multprocess\), and the associated real process~\(\mint\).
If \(\mint\) is positive and {\comp}, then so is \(\multprocess\).
As a consequence, any positive {\comp} real process~\(\process\) has a positive {\comp} multiplier process~\(\multprocess\), such that \(\process=\process(\init)\mint\).
\end{proposition}

\section{Random sequences in an imprecise probability tree}\label{sec:randomness}
With all the scaffolding now in place, we're finally ready to associate various notions of randomness with a forecasting system~\(\frcstsystem\)---or in other words, with an imprecise probability tree.
We want to be able to introduce and study several versions of randomness, each connected with a particular class of test supermartingales---capital processes for Sceptic when she starts with unit capital and never borrows.

\subsection{Allowable test processes and test supermartingales}
In what follows, we will denote by~\(\allowables\) any \emph{countable} set of test processes that includes the countable set of all {\comp} positive test processes, which we denote by~\(\poscallowables\).
Examples of such sets~\(\allowables\) are:
\begin{center}
\begin{tabular}{r|l}
\(\poscallowables\) & all {\comp} positive test processes\\[.5ex]
\(\callowables\) & all {\comp} test processes\\[.5ex]
\(\mlallowables\) & all {\lscomp} test processes\\[.5ex]
\(\multmlallowables\) & all test processes generated by {\lscomp} multiplier processes.\\
\end{tabular}
\end{center}
We will call such test processes in~\(\allowables\) \emph{allowable}.
Observe that,
\begin{equation}\label{eq:allowables:order}
\poscallowables\subseteq\callowables\text{ and }\poscallowables\subseteq\multmlallowables\subseteq\mlallowables,
\end{equation}
where the second chain of inclusions follows from Propositions~\ref{prop:computable:upper:lower}, \ref{prop:computable:from:multiplier} and~\ref{prop:computablemultiplier:from:process}.

The test supermartingales for~\(\frcstsystem\) that belong to this set \(\allowables\) will also be called \emph{allowable test supermartingales}, and collected in the set \(\allowabletests[\frcstsystem]\coloneqq\allowables\cap\testsupermartins[\frcstsystem]\).
In particular,
\begin{center}
\begin{tabular}{r|l}
\(\poscallowabletests[\frcstsystem]\coloneqq\poscallowables\cap\testsupermartins[\frcstsystem]\) & all {\comp} positive test supermartingales for~\(\frcstsystem\)\\[.5ex]
\(\callowabletests[\frcstsystem]\coloneqq\callowables\cap\testsupermartins[\frcstsystem]\) & all {\comp} test supermartingales for~\(\frcstsystem\)\\[.5ex]
\(\mlallowabletests[\frcstsystem]\coloneqq\mlallowables\cap\testsupermartins[\frcstsystem]\) & all {\lscomp} test supermartingales for~\(\frcstsystem\)\\[.5ex]
\(\multmlallowabletests[\frcstsystem]\coloneqq\multmlallowables\cap\testsupermartins[\frcstsystem]\) & all ({\lscomp}) test supermartingales for~\(\frcstsystem\)\\
& generated by {\lscomp} supermartingale multipliers.\\
\end{tabular}
\end{center}

\subsection{Randomness}
In the rest of this section (and paper), and unless explicitly stated to the contrary, \(\allowables\) is an arbitrary but fixed set of allowable test processes.
We remind the reader once again that all such sets~\(\allowables\) and the corresponding sets~\(\allowabletests[\frcstsystem]\) are countable.

\begin{definition}[Randomness]\label{def:randomness}
Consider any forecasting system~\(\frcstsystem\colon\sits\to\imprecisefrcsts\) and path~\(\pth\in\pths\).
We call \(\pth\) \emph{\(\allowables\)-random for~\(\frcstsystem\)} if all (allowable) test supermartingales~\(\test\) in~\(\allowabletests[\frcstsystem]\) remain bounded above on~\(\pth\), meaning that there is some~\(B_\test\in\reals\) such that \(\test(\pthto{n})\leq B_\test\) for all~\(n\in\naturals\), or equivalently, that \(\sup_{n\in\naturals}\test(\pthto{n})<\infty\).
We then also say that the forecasting system~\(\frcstsystem\) \emph{makes \(\pth\) \(\allowables\)-random}.
\end{definition}
\noindent
In other words, \(\allowables\)-randomness of a path means that there is no allowable strategy that starts with unit capital and avoids borrowing, and allows Sceptic to increase her capital without bounds by exploiting the bets on the outcomes along the path that are made available to her by Forecaster's specification of the forecasting system~\(\frcstsystem\).

When the forecasting system~\(\frcstsystem\) is precise and {\comp}, and \(\allowables\) is the set~\(\mlallowables\) of all {\lscomp} test processes, our definition reduces to that of \emph{{\ML} randomness} on the Schnorr--Levin (martingale-theoretic) account \cite{vovk2010:randomness,ambosspies2000,bienvenu2009:randomness,schnorr1971,schnorr1973}, because \(\mlallowabletests[\frcstsystem]\) is the set of all {\lscomp} test supermartingales for~\(\frcstsystem\).
We will therefore continue to call \(\mlallowables\)-randomness \emph{{\ML} randomness}, also when the forecasting system~\(\frcstsystem\) is no longer precise or {\comp}.

Similarly, when the forecasting system~\(\frcstsystem\) is precise and {\comp}, and \(\allowables\) is the set~\(\callowables\) of all {\comp} test processes, our definition reduces to that of \emph{{\comp} randomness} \cite{ambosspies2000,bienvenu2009:randomness}, because \(\callowabletests[\frcstsystem]\) is the set of all {\comp} test supermartingales for~\(\frcstsystem\).
We will therefore continue to call \(\callowables\)-randomness \emph{{\comp} randomness}, also when the forecasting system~\(\frcstsystem\) is no longer precise or {\comp}.

We denote by
\begin{equation*}
\random[\pth]{\allowables}
\coloneqq\cset{\frcstsystem\in\frcstsystems}
{\text{\(\pth\) is \(\allowables\)-random for~\(\frcstsystem\)}}
\end{equation*}
the set of all forecasting systems for which the path~\(\pth\) is \(\allowables\)-random.
We will also use the special notations \(\poscrandom[\pth]\), \(\crandom[\pth]\), \(\multmlrandom[\pth]\) and \(\mlrandom[\pth]\) in the cases that \(\allowables\) is equal to~\(\poscallowables\), \(\callowables\), \(\multmlallowables\) and \(\mlallowables\), respectively.

As a special, not unimportant but fairly trivial case, the ({\comp}) \emph{vacuous} forecasting system~\(\vacfrcstsystem\) assigns the vacuous forecast~\(\vacfrcstsystem(\sit)\coloneqq\frcsts\) to all situations~\(\sit\in\sits\).
Recall that we have introduced the notation~\(\frcstsystem\subseteq\frcstsystem^*\) to mean that \(\frcstsystem^*\) is at least as conservative as~\(\frcstsystem\), so \(\frcstsystem(\sit)\subseteq\frcstsystem^*(\sit)\) for all~\(\sit\in\sits\).
Then clearly \(\frcstsystem\subseteq\vacfrcstsystem\) for all~\(\frcstsystem\in\frcstsystems\), so \(\vacfrcstsystem\) is the most conservative forecasting system, with local models \(\uex_{\vacfrcstsystem(\sit)}=\max\) for all~\(\sit\in\sits\).
It corresponds to Forecaster making no actual commitments, and the closed convex cone~\(\availables_{\frcsts}\) of gambles~\(f\leq0\) that are then available to Sceptic at each successive stage is depicted in Figure~\ref{fig:vacuous:capital:increase}.

\begin{figure}[ht]
\centering
\begin{tikzpicture}[scale=1.25]\footnotesize
\coordinate (xaxis) at (1.5,0);
\coordinate (yaxis) at (0,1.5);
\coordinate (origin) at (0,0);
\coordinate (bottomleft) at (-1.5,-1.5);
\coordinate (aboveleft) at (-1.5,0.0);
\coordinate (belowright) at (0.0,-1.5);
\coordinate (gamble) at (-1,-0.5);
\coordinate (label) at (0,-1.7);
\draw[step=.5cm,gray,very thin] (-1.4,-1.4) grid (1.4,1.4);
\draw[->] (-1.5,0) -- (xaxis) node[below] {\footnotesize\(f(1)\)};
\draw[->] (0,-1.5) -- (yaxis) node[above] {\footnotesize\(f(0)\)};
\draw[blue,thick,name path=marginal line] (aboveleft) -- node[midway,above,rotate=0] {\footnotesize \(\ex_{0}(f)=0\)} (origin);
\draw[blue,thick] (origin) -- node[midway,above,rotate=-90] {\footnotesize \(\ex_{1}(f)=0\)} (belowright);
\fill[nearly transparent,blue] (origin) -- (belowright) -- (bottomleft) -- (aboveleft) -- cycle;
\draw[red,very thin] (bottomleft) -- (origin) -- coordinate[midway] (diagonal) (1.5,1.5) ;
\node[red,rotate=45,above=5pt] at (diagonal) {\footnotesize \(f(1)\leq f(0)\)};
\node[red,rotate=45,below=5pt] at (diagonal) {\footnotesize \(f(1)\geq f(0)\)};
\end{tikzpicture}
\caption{Gambles~\(f\) available to Sceptic when Forecaster announces the vacuous forecast~\(I=\frcsts\) with~\(\lp=0\) and~\(\up=1\).}
\label{fig:vacuous:capital:increase}
\end{figure}

The following proposition uses this vacuous forecasting system to conclude that no \(\random[\pth]{\allowables}\) is empty.

\begin{proposition}\label{prop:vacuous}
All paths are \(\allowables\)-random for the vacuous forecasting system, so \(\vacfrcstsystem\in\random[\pth]{\allowables}\) for all~\(\pth\in\pths\).
\end{proposition}

\begin{proof}
In the imprecise probability tree associated with the vacuous forecasting system~\(\vacfrcstsystem\), a real process \(\supermartin\) is a supermartingale if and only if it is non-increasing: \(\adddelta\supermartin\leq0\).
All test supermartingales for~\(\vacfrcstsystem\) are therefore bounded above by~\(1\) on any path~\(\pth\in\pths\).
\end{proof}

The more conservative, or imprecise, the forecasting system, the less stringent is the corresponding randomness notion.

\begin{proposition}\label{prop:nestedfrcstsystems}
Let \(\pth\) be \(\allowables\)-random for a forecasting system~\(\frcstsystem\).
Then \(\pth\) is also \(\allowables\)-random for any forecasting system~\(\frcstsystem^*\) such that \(\frcstsystem\subseteq\frcstsystem^*\).
\end{proposition}

\begin{proof}
Since \(\frcstsystem\subseteq\frcstsystem^*\) implies that \(\allowabletests[\frcstsystem^*]\subseteq\allowabletests[\frcstsystem]\), this follows trivially from Definition~\ref{def:randomness}.
\end{proof}

The larger the set~\(\allowables\) of allowable test processes, the more stringent is the corresponding randomness notion, and the `fewer' \(\allowables\)-random paths there are.
More precisely:

\begin{proposition}\label{prop:nestedallowables}
Consider two sets~\(\allowables,\allowables'\) of allowable test processes such that \(\smash{\allowables'\subseteq\allowables}\).
If \(\pth\) is \(\allowables\)-random for a forecasting system~\(\frcstsystem\), then \(\pth\) is also \(\allowables'\)-random for~\(\frcstsystem\), and therefore \(\random[\pth]{\allowables}\subseteq\random[\pth]{\allowables'}\).
\end{proposition}

\begin{proof}
Since \(\allowables'\subseteq\allowables\) implies that \(\allowableteststoo[\frcstsystem]\subseteq\allowabletests[\frcstsystem]\), this follows trivially from Definition~\ref{def:randomness}.
\end{proof}
\noindent As a fairly direct consequence of Equation~\eqref{eq:allowables:order}, we can now infer from Proposition~\ref{prop:nestedallowables} that
\begin{equation}\label{eq:randomness:order}
\mlrandom[\pth]\subseteq\multmlrandom[\pth]\subseteq\poscrandom[\pth]=\crandom[\pth],
\end{equation}
where only the equality needs more explanation.

\begin{proof}[Proof of the equality in Equation~\eqref{eq:randomness:order}]
Since \(\poscallowables\subseteq\callowables\), Proposition~\ref{prop:nestedallowables} already guarantees that \(\crandom[\pth]\subseteq\poscrandom[\pth]\).
To prove the converse inclusion, assume that the forecasting system~\(\frcstsystem\) makes~\(\pth\) \(\poscallowables\)-random.
Consider any {\comp} test supermartingale~\(\test\) for~\(\frcstsystem\), so \(\test\in\callowabletests\), and assume {\itshape ex absurdo} that \(\test\) is unbounded on~\(\pth\), then so is the {\comp} positive test supermartingale \((1+\test)/2\in\poscallowabletests\), a contradiction.
\end{proof}

Because \(\multmlallowables\)-randomness is weaker than {\ML} randomness, but has a similar flavour, we will also call it \emph{weak {\ML} randomness}.

\subsection{Real-valued versus extended real-valued supermartingales}
Before moving on, we want to comment on a particular aspect of our randomness definition that differs slightly from other approaches, such as for instance described in Refs.~\cite{schnorr1971,vovk2010:randomness}, which allow the test supermartingales in their randomness definition to be \emph{extended} real-valued; we restrict ourselves to real-valued test supermartingales in the present approach.
Let us explain what are the differences between these two approaches, and indicate briefly why we prefer ours.

A process~\(\supermartin\) assuming values in~\(\reals\cup\set{\infty}\) that satisfies the \emph{extended supermartingale inequality}
\begin{equation}\label{eq:extended:supermartingale}
\uex_{\frcstsystem(\sit)}(\supermartin(\sit\cdot))\leq\supermartin(\sit)
\text{ for all~\(\sit\in\sits\)},
\end{equation}
is called an \emph{extended supermartingale} for~\(\frcstsystem\).
The upper expectation \(\uex_{\frcstsystem(\sit)}\) in Equation~\eqref{eq:extended:supermartingale} is defined on maps \(f\colon\outcomes\to\reals\cup\set{\infty}\) by generalising Equation~\eqref{eq:local:upper}: for any~\(I\in\imprecisefrcsts\),
\begin{equation}\label{eq:local:upper:extended}
\uex_{I}(f)
\coloneqq\sup_{p\in I}\ex_p(f)
=\sup_{p\in I}[pf(1)+(1-p)f(0)],
\end{equation}
taking into account the conventions that \(0\cdot\infty=0\), \(x+\infty=\infty\) for all real~\(x\), and \(\infty+\infty=\infty\).
This implies that if \(f(1)=\infty\), then also \(\uex_I(f)=\infty\), unless \(I=\set{0}\), in which case we have that \(\uex_I(f)=f(0)\).
Similarly, if \(f(0)=\infty\), then also \(\uex_I(f)=\infty\), unless \(I=\set{1}\), in which case \(\uex_I(f)=f(1)\).
This argumentation, in combination with~\ref{axiom:coherence:bounds}, tells us that
\begin{equation}\label{eq:local:upper:extended:bounds}
\uex_{I}(f)\geq\min\set{f(0),f(1)}
\text{ for all \(f\colon\outcomes\to\reals\cup\set{\infty}\) and \(I\in\imprecisefrcsts\)}.
\end{equation}
The extended supermartingale inequality~\eqref{eq:extended:supermartingale} imposes no requirements on \(\supermartin(\sit\cdot)\) whenever \(\supermartin(\sit)=\infty\).
At the same time, it tells us that an extended supermartingale with \(\supermartin(\sit)\in\reals\) can't jump to~\(\infty\) in~\(\sit1\) unless \(\frcstsystem(\sit)=0\),\footnote{Recall that we make no distinction between a singleton forecast and its single element.} and similarly, it can't jump to~\(\infty\) in~\(\sit0\) unless \(\frcstsystem(\sit)=1\): infinite jumps upwards in going from one situation to the next are only allowed when the transition between these situations has upper probability zero, and can therefore only occur in situations~\(\sit\) whose forecast~\(\frcstsystem(\sit)\) is \emph{degenerate}, meaning that \(\frcstsystem(\sit)=0\) or \(\frcstsystem(\sit)=1\).

If, contrary to our approach, randomness of a path~\(\pth\) means that all allowable \emph{extended} test supermartingales must remain bounded, this means that, in addition to the requirements on real supermartingales present in our condition, such extended test supermartingales must not be allowed to jump to~\(\infty\) anywhere on the path~\(\pth\).
Now, an extended test supermartingale~\(\test\) starts with the initial value~\(\test(\init)=1\) in~\(\init\), so if it is to assume the value~\(\infty\) somewhere, there must be at least one situation~\(\sit\in\sits\) and outcome~\(x\in\outcomes\) such that \(\test\) makes an infinite jump in going from~\(\test(\sit)\in\reals\) to~\(\test(\sit x)=\infty\).
As explained above, the extended supermartingale condition~\eqref{eq:extended:supermartingale} tells us that this can only happen when
\begin{equation*}
\frcstsystem(\sit)=
\begin{cases}
0&\text{if \(x=1\)}\\
1&\text{if \(x=0\)},
\end{cases}
\end{equation*}
meaning that the (upper) probability of the transition from~\(\sit\) to~\(\sit x\) is zero.
In other words, \emph{there can only be a difference between the two types of randomness definitions when there are degenerate transition probabilities in the (imprecise probability) tree.}\footnote{Of course, this argumentation tells us that both types of randomness definition coincide for the `fair-coin' forecasting system typically considered in the literature.}
And in such cases, when there is for instance a transition in some path~\(\pth\) that has upper probability zero, \(\pth\) can't be random according to the definition with extended test supermartingales.
So, in principle, whether a path is random on the `extended' definition can in that case depend on a single outcome, which is something we find unfortunate.
Whether such a path~\(\pth\) will be random according to our definition, will depend on the forecasting system and the transition behaviour on~\(\pth\).

\section{Schnorr randomness in an imprecise probability tree}\label{sec:schnorr:randomness}
Next, we concentrate on extending the notion of Schnorr randomness to our present context.
We begin with a definition borrowed from Schnorr's seminal work \cite{schnorr1971,schnorr1973}.
\begin{definition}[Growth function]
We call a map \(\ordering\colon\naturalswithzero\to\naturalswithzero\) a \emph{growth function} if
\begin{enumerate}[label=\upshape(\roman*),leftmargin=*,noitemsep,topsep=0pt]
\item it is recursive;
\item it is non-decreasing: \(\group{\forall n_1,n_2\in\naturalswithzero}\group{n_1\leq n_2\then\ordering(n_1)\leq\ordering(n_2)}\);
\item it is unbounded.
\end{enumerate}
We say that a real-valued map~\(\mu\colon\naturalswithzero\to\reals\) is \emph{computably unbounded} if there is some growth function~\(\ordering\) such that \(\limsup_{n\to\infty}\sqgroup{\mu(n)-\ordering(n)}>0\), or equivalently,
\begin{equation}\label{eq:computably:unbounded}
\inf_{m\in\naturalswithzero}\sup_{n\geq m}\sqgroup{\mu(n)-\ordering(n)}>0.
\end{equation}
\end{definition}

In what follows, it will often simplify our proofs to work with a more general notion of growth function.

\begin{definition}[Real growth function]
A map~\(\realordering\colon\naturalswithzero\to\nonnegreals\) is a \emph{real growth function} if
\begin{enumerate}[label=\upshape(\roman*),leftmargin=*,noitemsep,topsep=0pt]
\item it is {\comp};
\item it is non-decreasing: \(\group{\forall n_1,n_2\in\naturalswithzero}\group{n_1\leq n_2\then\realordering(n_1)\leq\realordering(n_2)}\);
\item it is unbounded.
\end{enumerate}
\end{definition}
\noindent It turns out that working with this more general definition does not really change what is important, namely {\comp} unboundedness.
Indeed, we can use it to give a number of equivalent characterisations of this notion that will prove useful further on.
The rather technical proof of these alternative characterisations is deferred to the Appendix.

\begin{proposition}\label{prop:computably:unbounded:various}
Consider a real-valued map~\(\mu\colon\naturalswithzero\to\reals\), then the following statements are equivalent:
\begin{enumerate}[label=\upshape(\roman*),leftmargin=*,noitemsep,topsep=0pt]
\item\label{it:computably:unbounded} there is a growth function~\(\ordering\) such that \(\limsup_{n\to\infty}\sqgroup{\mu(n)-\ordering(n)}>0\);
\item\label{it:computably:unbounded:real} there is a real growth function~\(\realordering\) such that \(\limsup_{n\to\infty}\sqgroup{\mu(n)-\realordering(n)}\geq0\);
\item\label{it:computably:unbounded:fraction} there is a real growth function~\(\realordering\) such that \(\limsup_{n\to\infty}\nicefrac{\mu(n)}{\realordering(n)}>0\).\footnote{This expression makes sense because, for large enough~\(n\), \(\tau(n)>0\).}
\end{enumerate}
All of these statements characterise the {\comp} unboundedness of~\(\mu\).
\end{proposition}

We will also need the following simple results.
Their proofs are fairly obvious, but we have included them in the Appendix for the sake of completeness.

\begin{proposition}\label{prop:computably:unbounded:implies:unbounded}
If a real-valued map~\(\mu\colon\naturalswithzero\to\reals\) is computably unbounded, it is also unbounded above.
\end{proposition}

\begin{proposition}\label{prop:computably:unbounded:product}
If the point-wise product \(\mu_1\mu_2\) of real maps~\(\mu_1\colon\naturalswithzero\to\reals\) and~\(\mu_2\colon\naturalswithzero\to\reals\) is computably unbounded, then at least one of the factors~\(\mu_1\) or~\(\mu_2\) is computably unbounded too.
\end{proposition}

We can now extend Schnorr's original definition of randomness \cite{schnorr1971,schnorr1973} in a probability tree with a constant precise forecast~\(I=\set{\nicefrac12}\) to imprecise probability trees associated with an arbitrary---and not necessarily precise nor {\comp}---forecasting system.

\begin{definition}[Schnorr randomness]\label{def:schnorrrandomness}
Consider any forecasting system~\(\frcstsystem\colon\sits\to\imprecisefrcsts\).
We call a path~\(\pth\in\pths\) \emph{Schnorr random for~\(\frcstsystem\)} if no {\comp} test supermartingale~\(\test\in\callowabletests[\frcstsystem]\) for~\(\frcstsystem\) is computably unbounded on~\(\pth\), or in other words, if \(\limsup_{n\to\infty}\sqgroup{\test(\pthto{n})-\ordering(n)}\leq0\) for all {\comp} test supermartingales~\(\test\in\callowabletests[\frcstsystem]\) for~\(\frcstsystem\) and all growth functions~\(\ordering\).
We then also say that the forecasting system~\(\frcstsystem\) \emph{makes \(\pth\) Schnorr random}.
\end{definition}
\noindent In this definition, we can of course also use the alternative characterisations of {\comp} unboundedness listed in Proposition~\ref{prop:computably:unbounded:various}.
Furthermore, without loss of generality, we can focus on {\comp} \emph{positive} test supermartingales.

\begin{proposition}\label{prop:schnorrwithpositiveT}
Consider any forecasting system~\(\frcstsystem\colon\sits\to\imprecisefrcsts\).
Then a path~\(\pth\in\pths\) is Schnorr random for~\(\frcstsystem\) if and only if no {\comp} \emph{positive} test supermartingale~\(\test\in\callowabletests[\frcstsystem]\) for~\(\frcstsystem\) is computably unbounded on~\(\pth\).
\end{proposition}

\begin{proof}
It clearly suffices to prove the `if' part.
To this end, we consider any path~\(\pth\in\pths\) that isn't Schnorr random for~\(\frcstsystem\) and prove that there is some {\comp} positive test supermartingale for~\(\frcstsystem\) that is computably unbounded on~\(\pth\).
Since~\(\pth\in\pths\) isn't Schnorr random for~\(\frcstsystem\), there is a {\comp} test supermartingale~\(\test\) for~\(\frcstsystem\) that is computably unbounded on \(\pth\), meaning that there is some growth function \(\ordering\) such that \(\limsup_{n\to\infty}\sqgroup{\test(\pthto{n})-\ordering(n)}>0\).
For the real growth function \(\realordering\coloneqq(1+\ordering)/2\) and the {\comp} positive test supermartingale \(\test'\coloneqq(1+\test)/2\in\callowabletests\), it then follows that \(\limsup_{n\to\infty}\sqgroup{\test'(\pthto{n})-\realordering(n)}>0\), so \(\test'\) is computably unbounded on~\(\pth\) by Proposition~\ref{prop:computably:unbounded:various}\ref{it:computably:unbounded:real}.
\end{proof}

We denote by
\begin{equation*}
\schnorrrandom[\pth]
\coloneqq\cset{\frcstsystem\in\frcstsystems}
{\text{\(\pth\) is Schnorr random for~\(\frcstsystem\)}}
\end{equation*}
the set of all forecasting systems that make the path~\(\pth\) Schnorr random.

The following results are now fairly immediate.
The first one shows that Schnorr randomness is the weakest form of randomness that we're considering here.

\begin{proposition}\label{prop:random:implies:schnorrandom}
Consider any set~\(\allowables\) of allowable test processes, any forecasting system~\(\frcstsystem\colon\sits\to\imprecisefrcsts\) and any path~\(\pth\in\pths\).
Then if \(\pth\) is \(\allowables\)-random for~\(\frcstsystem\), it is also Schnorr random for~\(\frcstsystem\), and therefore \(\random[\pth]{\allowables}\subseteq\schnorrrandom\).
\end{proposition}

\begin{proof}
It follows from Proposition~\ref{prop:nestedallowables} and \(\poscallowables\subseteq\allowables\) that it suffices to give a proof for the case that~\(\allowables=\poscallowables\).
Assume that \(\pth\) isn't Schnorr random for~\(\frcstsystem\).
Then it follows from Proposition~\ref{prop:schnorrwithpositiveT} that there is some positive {\comp} test supermartingale~\(\test\in\callowabletests[\frcstsystem]\) for~\(\frcstsystem\) that is computably unbounded on~\(\pth\).
But then Proposition~\ref{prop:computably:unbounded:implies:unbounded} implies that \(\test\) is also unbounded on~\(\pth\), and therefore \(\pth\) isn't \(\poscallowables\)-random for~\(\frcstsystem\).
\end{proof}
\noindent Together with Equation~\eqref{eq:randomness:order}, this result tells us that
\begin{equation}\label{eq:randomness:order:with:schnorr}
\mlrandom[\pth]\subseteq\multmlrandom[\pth]\subseteq\poscrandom[\pth]=\crandom[\pth]\subseteq\schnorrrandom[\pth].
\end{equation}

\begin{proposition}\label{prop:vacuous:schnorrrandom}
All paths are Schnorr random for the vacuous forecasting system, so \(\vacfrcstsystem\in\schnorrrandom\) for all~\(\pth\in\pths\).
\end{proposition}

\begin{proof}
This is an immediate consequence of Propositions~\ref{prop:vacuous} and~\ref{prop:random:implies:schnorrandom}.
\end{proof}

\begin{proposition}\label{prop:nestedfrcstsystems:schnorrrandom}
Let \(\pth\) be Schnorr random for a forecasting system~\(\frcstsystem\).
Then \(\pth\) is also Schnorr random for any forecasting system~\(\frcstsystem^*\) such that \(\frcstsystem\subseteq\frcstsystem^*\).
\end{proposition}

\begin{proof}
Since \(\frcstsystem\subseteq\frcstsystem^*\) implies that \(\callowabletests[\frcstsystem^*]\subseteq\callowabletests[\frcstsystem]\), this follows trivially from Definition~\ref{def:schnorrrandomness}.
\end{proof}

\section{Consistency results}\label{sec:consistency}
We now turn to a number of important consistency results for the various randomness notions we have introduced.
In the rest of this section, unless explicitly mentioned to the contrary, \(\allowables\) is an arbitrary but fixed set of allowable test processes.

\subsection{All paths are almost surely random.}
We first show that any Forecaster who specifies a forecasting system is consistent in the sense that he believes himself to be well-calibrated: in the imprecise probability tree generated by his own forecasts, almost all paths will be random, so he is `almost sure' that Sceptic will not be able to become infinitely rich by exploiting his---Forecaster's---forecasts.

\begin{theorem}[The well-calibrated imprecise Bayesian; strong version]\label{thm:consistency}
Consider any forecasting system~\(\frcstsystem\colon\sits\to\imprecisefrcsts\).
Then almost all paths are \(\allowables\)-random for~\(\frcstsystem\) in the imprecise probability tree that corresponds to~\(\frcstsystem\).
As a consequence, almost all paths are Schnorr random for~\(\frcstsystem\) in the imprecise probability tree corresponding to~\(\frcstsystem\).
\end{theorem}

\begin{proof}
We first prove the result for~\(\allowables\)-randomness.
Consider the event
\begin{equation*}
A\coloneqq\cset{\pth\in\pths}{\pth\text{ is \(\allowables\)-random for~\(\frcstsystem\)}},
\end{equation*}
then we have to prove that \(\lglobalprob(A)=1\), or equivalently, that \(\uglobalprob(A^c)=0\): the non-random paths belong to a null set.
It follows from the assumptions that, for every \(\pth\) in~\(A^c\), there is some allowable test supermartingale~\(\test_\pth\in\allowabletests[\frcstsystem]\) that becomes unbounded on~\(\pth\).
Let \((\test_k)_{k\in\naturals}\) be any enumeration of the countable set \(\allowabletests[\frcstsystem]=\testallowables[\frcstsystem]\) of allowable test supermartingales, and consider any collection of positive real weights \((w_k)_{k\in\naturals}\) such that \(\sum_{k\in\naturals}w_k=1\).
We use these to construct the non-negative extended real process \(\process\coloneqq\sum_{k\in\naturals}w_k\test_k\),  with \(\process(\init)=1\).

We now construct, for any real~\(\alpha>1\), a test supermartingale \(\test^{\parw{\alpha}}\) for~\(\frcstsystem\).
Let
\begin{equation*}
\test^{\parw{\alpha}}(\sit)
\coloneqq
\begin{cases}
\alpha&\text{if \(\process(\altsit)\geq\alpha\) for some precursor~\(\altsit\precedes\sit\) of~\(\sit\)}\\
\process(\sit)&\text{if \(\process(\altsit)<\alpha\) for all precursors~\(\altsit\precedes\sit\) of~\(\sit\)}
\end{cases}
\quad\text{for all~\(\sit\in\sits\)},
\end{equation*}
It is a matter of direct verification to show that \(\test^{\parw{\alpha}}\) is indeed a test supermartingale for~\(\frcstsystem\).
It is clear that, for any~\(\pth\in A^c\), some~\(\test_k\) becomes unbounded on~\(\pth\), and therefore~\(\process\) will eventually exceed~\(\alpha\) on~\(\pth\).
Hence, \(\lim_{n\to\infty}\test^{(\alpha)}(\pthto{n})=\alpha\) for all \(\pth\in A^c\).
This implies that \(\liminf\test^{\parw{\alpha}}(\omega)\geq\alpha\ind{A^c}(\pth)\) for all \(\pth\in\pths\), and therefore
\begin{equation*}
0
\leq\uglobalprob(A^c)
=\uglobal(\ind{A^c})
=\frac1{\alpha}\uglobal(\alpha\ind{A^c})
\leq\frac1{\alpha}\test^{\parw{\alpha}}(\init)
=\frac1{\alpha}.
\end{equation*}
Here, the first inequality follows from the property~\ref{axiom:lower:upper:bounds} of the upper expectation~\(\uglobal\), the second equality from the property~\ref{axiom:lower:upper:homogeneity}, the second inequality from Equation~\eqref{eq:tree:upper:expectation}, and the last equality from the fact that \(\test^{\parw{\alpha}}\) is a test supermartingale.
Since this statement holds for all real~\(\alpha>0\), this implies that, indeed, \(\uglobalprob(A^c)=0\).

To prove the result for Schnorr randomness, it now suffices to recall Proposition~\ref{prop:random:implies:schnorrandom} and the monotonicity of the upper expectation~\(\uglobal\) [property~\ref{axiom:lower:upper:monotonicity}].
\end{proof}
\noindent
This result is quite powerful, and it guarantees in particular that there always are random paths, for any forecasting system.

\begin{corollary}\label{cor:consistency}
For any forecasting system~\(\frcstsystem\), there is at least one path that is \(\allowables\)-random, and therefore also Schnorr random, for~\(\frcstsystem\).
\end{corollary}

\begin{proof}
We give the proof for~\(\allowables\)-randomness; the result for Schnorr randomness will then follow from Proposition~\ref{prop:random:implies:schnorrandom}.
In the proof of Theorem~\ref{thm:consistency}, we considered the set~\(A\) of all paths that are \(\allowables\)-random for~\(\frcstsystem\), and proved that its complement~\(A^c\) is null for~\(\frcstsystem\), or in other words, that \(\uglobalprob(A^c)=0\).
If, {\itshape ex absurdo}, \(A\) were empty, this would imply that~\(A^c=\pths\) and therefore that~\(\uglobalprob(\pths)=0\).
But it follows from property~\ref{axiom:lower:upper:bounds} that, actually, \(\uglobalprob(\pths)=1\), leading to a contradiction.
\end{proof}

In fact, since Theorem~\ref{thm:consistency} tells us that the set of all random paths for a forecasting system has lower probability one, there are many such random paths in a `measure-theoretic' sense.
But we will see in Section~\ref{sec:meagreness} that, in a specific topological sense, random paths are few, as they typically constitute only a meagre set.
This is a known result for precise randomness, that was, as far as we can judge, first formulated as such in the context of a much more encompassing discussion on the nature of randomness by Muchnik, Semenov and Uspensky \cite{muchnik1998:metaphysics:of:randomness}.
It also appeared in a related form in the wake of discussions \cite{schervisch1985:discussion:of:dawid,oakes1985:selfcalibratingpriors,dawid1985:selfcalibratingpriors:comment,schervisch1985:selfcalibratingpriors:comment,belot2013:failure:of:calibration} of Philip Dawid's papers on calibration \cite{dawid1982:well:calibrated:bayesian,dawid1985:cbep}, and was foreshadowed by some of Terrence Fine's results \cite{fine1970}.

\subsection{The well-calibrated imprecise Bayesian}
We now turn to a weaker consistency result that deals with limits (inferior and superior) of relative frequencies.
We will see that it generalises to interval forecasts the arguments and conclusions in an earlier paper on calibration by Philip Dawid \cite{dawid1982:well:calibrated:bayesian}.

We start with any real process~\(\process\colon\sits\to\reals\).
We consider any so-called \emph{selection process}~\(\selection\colon\sits\to\outcomes\), and use it to define the real process~\(\average{\process}\colon\sits\to\reals\) as follows:
\begin{equation*}
\average{\process}(\sit)
\coloneqq
\begin{cases}
0
&\text{if \(\sum_{k=0}^{\dist{\sit}-1}\selection(\sitto{k})=0\)}\\
\dfrac{\sum_{k=0}^{\dist{\sit}-1}\selection(\sitto{k})
\sqgroup{\adddelta\process(\sitto{k})(\sitat{k+1})}}
{\sum_{k=0}^{\dist{\sit}-1}\selection(\sitto{k})}
&\text{if \(\sum_{k=0}^{\dist{\sit}-1}\selection(\sitto{k})>0\)}
\end{cases}
\quad\text{for all~\(\sit\in\sits\).}
\end{equation*}
In words, \(\average{\process}(\sit)\) is the arithmetic average of the process differences~\(\adddelta\process(\sitto{k})\) along the path segment~\(\sit\), where only the actually selected precursor situations~\(\sitto{k}\) with \(\selection(\sitto{k})=1\) are taken into account.

As a particular example that will be useful further on, fix any gamble~\(h\) on~\(\outcomes\), and consider the real process~\(\submartin^\frcstsystem_h\) defined by
\begin{equation*}
\submartin^\frcstsystem_h(\sit)
\coloneqq\smashoperator{\sum_{k=1}^{\dist{\sit}}}\sqgroup[\big]{h(\sitat{k})-\lex_{\frcstsystem(\sitto{k-1})}(h)}
\text{ for all~\(\sit\in\sits\)}.
\end{equation*}
On the one hand, we find for the corresponding process difference~\(\adddelta\submartin^\frcstsystem_h\) that
\begin{equation}\label{eq:lln:local:increments}
\adddelta\submartin^\frcstsystem_h(\sit)(\xval)
=\submartin^\frcstsystem_h(\sit\xval)-\submartin^\frcstsystem_h(\sit)
=h(\xval)-\lex_{\frcstsystem(\sit)}(h)
\text{ for all~\(\xval\in\outcomes\)},
\end{equation}
so \(\adddelta\submartin^\frcstsystem_h(\sit)=h-\lex_{\frcstsystem(\sit)}(h)\), and therefore we find on the other hand for its lower expectations in the imprecise probability tree that
\begin{equation}\label{eq:lln:local:increments:expectation}
\lex_{\frcstsystem(\sit)}(\adddelta\submartin^\frcstsystem_h(\sit))
=\lex_{\frcstsystem(\sit)}(h)-\lex_{\frcstsystem(\sit)}(h)
=0,
\end{equation}
using the coherence property \ref{axiom:coherence:constantadditivity} for the first equality.
We conclude that \(\submartin^\frcstsystem_h\) is a submartingale for~\(\frcstsystem\).
Its process differences~\(\adddelta\submartin^\frcstsystem_h(\sit)\) are furthermore uniformly bounded, for instance by the variation (semi)norm~\(\varnorm{h}\) of~\(h\):
\begin{equation}\label{eq:uniform:bound:on:increments}
\abs{\adddelta\submartin^\frcstsystem_h(\sit)}\leq\max h-\min h\eqqcolon\varnorm{h}
\text{ for all \(\sit\in\sits\)},
\end{equation}
where the inequality follows from Equation~\eqref{eq:lln:local:increments} and the coherence property \ref{axiom:coherence:bounds}.
Observe by the way that in this particular case, for all \(\sit\in\sits\):
\begin{equation}\label{eq:average:over:selected:increments}
\average{\submartin^\frcstsystem_h}(\sit)
=
\begin{cases}
0
&\text{if \(\sum_{k=0}^{\dist{\sit}-1}\selection(\sitto{k})=0\)}\\
\dfrac{\sum_{k=0}^{\dist{\sit}-1}\selection(\sitto{k})
\sqgroup[\big]{h(\sitat{k+1})-\lex_{\frcstsystem(\sitto{k})}(h)}}
{\sum_{k=0}^{\dist{\sit}-1}\selection(\sitto{k})}
&\text{if \(\sum_{k=0}^{\dist{\sit}-1}\selection(\sitto{k})>0\)}.
\end{cases}
\end{equation}

We can now apply our law of large numbers for uniformly bounded submartingale differences \cite[Theorem~7]{cooman2015:markovergodic} to get to the following result, which generalises Philip Dawid's well-known consistency result for Bayesian Forecasters~\cite[General Calibration Theorem]{dawid1982:well:calibrated:bayesian}, to deal with interval forecasts.
In its formulation, we use the following formalised version of notation that we introduced earlier: for any~\(n\in\naturals\), we consider the variables---maps defined on the sample space~\(\pths\)---\(\randomoutcometo[n]\) and \(\randomoutcome[n]\), defined by
\begin{equation*}
\randomoutcometo[n]\colon\pths\to\outcomes^n\colon\pth\mapsto\randomoutcometo[n](\pth)\coloneqq\pthto{n}
\text{ and }
\randomoutcome[n]\colon\pths\to\outcomes\colon\pth\mapsto\randomoutcome[n](\pth)\coloneqq\pthat{n},
\end{equation*}
where we also let, by convention, \(\randomoutcometo[0](\pth)=\randomoutcome[0](\pth)\coloneqq\init\) for all~\(\pth\in\pths\).

\begin{theorem}[The well-calibrated imprecise Bayesian]\label{thm:well:calibrated}
Let \(\frcstsystem\colon\sits\to\imprecisefrcsts\) be any forecasting system, let \(\selection\colon\sits\to\outcomes\) be any selection process, and let \(h\) be any gamble on~\(\outcomes\).
If \(\lim_{n\to\infty}\sum_{k=0}^{n-1}\selection(\randomoutcometo[k])=\infty\) then also \(\liminf_{n\to\infty}\average{\submartin^\frcstsystem_h}(\randomoutcometo[n])\geq0\) almost surely for the forecasting system~\(\frcstsystem\).
\end{theorem}

\begin{proof}
For any submartingale~\(\submartin\) for~\(\frcstsystem\) whose process differences are uniformly bounded, Theorem~7 in Ref.~\cite{cooman2015:markovergodic} states that, strictly almost surely, \(\lim_{n\to\infty}\sum_{k=0}^{n-1}\selection(\randomoutcometo[k])=\infty\) implies that \(\liminf_{n\to\infty}\average{\submartin}(\randomoutcometo[n])\geq0\), where `strictly almost surely' means that there is some test supermartingale that converges to~\(\infty\) on all paths where the statement isn't true.
Furthermore, Proposition~4 in Ref.~\cite{cooman2015:markovergodic} states that any event that holds strictly almost surely, also holds almost surely.
The result therefore follows because, as we have seen in the main text above, \(\submartin^\frcstsystem_h\) is a submartingale for~\(\frcstsystem\) whose process differences are uniformly bounded.
\end{proof}

One important step in the proof of this result---or actually, in the proof of Theorem~7 in Ref.~\cite{cooman2015:markovergodic} on which our proof above relies---is, stripped to its bare essentials, based on a surprisingly elegant and effective idea that goes back to Shafer and Vovk~\cite[Lemma~3.3]{shafer2001}.
We repeat it here, suitably adapted to the present context, in the lemma below, because it will next help us prove a related and equally important result---Theorem~\ref{thm:well:calibrated:general} further on---that will turn out to be crucial for establishing a number of claims in this paper: that randomness is inherently imprecise in Section~\ref{sec:non-stationarity}, and that random paths are few, topologically speaking, in Section~\ref{sec:meagreness}.

\begin{lemma}\label{lem:martinwlln}
Let \(\frcstsystem\colon\sits\to\imprecisefrcsts\) be any forecasting system, and consider any real~\(B>0\) and any~\(0<\xi<\frac{1}{B}\).
Let \(\submartin\) be any submartingale for~\(\frcstsystem\) such that \(\vert\adddelta\submartin\vert\leq B\).
Let \(\selection\colon\sits\to\outcomes\) be any selection process.
Then the real process~\(\process_\submartin\), defined by
\begin{equation}\label{eq:auxiliary:supermartingale}
\process_\submartin(\sit)
\coloneqq\smashoperator{\prod_{k=0}^{\dist{\sit}-1}}
\sqgroup[\big]{1-\xi\selection(\sitto{k})\adddelta\submartin(\sitto{k})(\sitat{k+1})}
\text{ for all~\(\sit\in\sits\)},
\end{equation}
is a positive test supermartingale for~\(\frcstsystem\).
Moreover, if we consider \(0<\epsilon<B\) and \(\xi\coloneqq\frac{\epsilon}{2B^2}\), so \(0<\xi<\frac{1}{2B}\), then
\begin{equation*}
\average{\submartin}(\sit)\leq-\epsilon
\then
\process_\submartin(\sit)
\geq\exp\group[\bigg]{\frac{\epsilon^2}{4B^2}\smashoperator{\sum_{k=0}^{\dist{\sit}-1}}\selection(\sitto{k})},
\text{ for all~\(\sit\in\sits\)}.
\end{equation*}
Finally, if \(\xi\) and \(\adddelta\submartin\) are {\comp} and \(\selection\) is recursive, then \(\process_\submartin\) is {\comp} as well.
\end{lemma}

\begin{proof}
Let
\begin{equation}\label{eq:multprocess:submartingale}
\multprocess_\submartin\coloneqq1-\xi\selection\adddelta\submartin.
\end{equation}
We first show that \(\multprocess_\submartin\) is a positive supermartingale multiplier for~\(\frcstsystem\).
To this end, consider any~\(\sit\in\sits\).
Then on the one hand, it follows from~\(0<\xi B<1\) and \(\vert\adddelta\submartin\vert\leq B\) that \(\multprocess_\submartin(\sit)=1-\xi\selection(\sit)\adddelta\submartin(\sit)\geq 1-\xi B>0\).
On the other hand, since \(\xi>0\) and \(\selection(\sit)\in\outcomes\), we infer from the coherence [\ref{axiom:coherence:constantadditivity} and~\ref{axiom:coherence:homogeneity}] and conjugacy properties of lower and upper expectations that
\begin{align*}
\uex_{\frcstsystem(\sit)}(\multprocess_\submartin(\sit))
&=\uex_{\frcstsystem(\sit)}\group[\big]{1-\xi\selection(\sit)\adddelta\submartin(\sit)}\\
&=1+\uex_{\frcstsystem(\sit)}\group[\big]{-\xi\selection(\sit)\adddelta\submartin(\sit)}
=1-\xi\selection(\sit)\lex_{\frcstsystem(\sit)}\group[\big]{\adddelta\submartin(\sit)}\leq1,
\end{align*}
where the inequality follows from~\(\lex_{\frcstsystem(\sit)}\group{\adddelta\submartin(\sit)}\geq0\), because we assumed that \(\submartin\) is a submartingale for~\(\frcstsystem\).
So, indeed, \(\multprocess_\submartin\) is a positive supermartingale multiplier for~\(\frcstsystem\).

Comparing Equations~\eqref{eq:auxiliary:supermartingale} and~\eqref{eq:multprocess:submartingale}, we see that \(\process_\submartin=\mint[{\multprocess_\submartin}]\), or in other words that \(\process_\submartin\) is generated by the multiplier process~\(\multprocess_\submartin\). 
Hence, \(\process_\submartin(\init)=1\) and, since \(\multprocess_\submartin\) is a positive supermartingale multiplier, \(\process_\submartin=\mint[{\multprocess_\submartin}]\) is a positive supermartingale for~\(\frcstsystem\).
So, indeed, \(\process_\submartin\) is a positive test supermartingale for~\(\frcstsystem\).

For the second statement, consider any~\(0<\epsilon<B\) and let \(\smash{\xi\coloneqq\frac{\epsilon}{2B^2}}\).
This implies that \(0<\xi<\frac{1}{2B}\), so we can already conclude that the first statement of the lemma holds for this particular choice of~\(\xi\).
For all \(\sit\in\sits\) and all real \(K\), since \(\process_\submartin\) and \(1-\xi\selection\adddelta\submartin=\multprocess_\submartin\) are positive, we infer from Equation~\eqref{eq:auxiliary:supermartingale} that
\begin{equation}\label{eq:K}
\process_\submartin(\sit)\geq\exp(K)
\ifandonlyif
\smashoperator{\sum_{k=0}^{\dist{\sit}-1}}\ln\sqgroup[\big]{1-\xi\selection(\sitto{k})\adddelta\submartin(\sitto{k})(\sitat{k+1})}\geq K.
\end{equation}
Since \(\vert\adddelta\submartin\vert\leq B\) and \(0<\epsilon<B\), we find that
\begin{equation}\label{eq:lower:bound:on:the:terms}
-\xi\selection(\sitto{k})\adddelta\submartin(\sitto{k})
\geq-\xi B
=-\frac{\epsilon}{2B}
>-\frac{1}{2}
\text{ for~\(0\leq k\leq\dist{\sit}-1\)}.
\end{equation}
We now restrict our attention to those~\(\sit\in\sits\) for which \(\average{\submartin}(\sit)\leq-\epsilon\).
Since \(\ln(1+x)\ge x-x^2\) for~\(x>-\frac{1}{2}\), we infer from Equation~\eqref{eq:lower:bound:on:the:terms} that
\begin{multline*}
\smashoperator{\sum_{k=0}^{\dist{\sit}-1}}
\ln\sqgroup[\big]{1-\xi\selection(\sitto{k})\adddelta\submartin(\sitto{k})(\sitat{k+1})}\\
\begin{aligned}
&\geq\smashoperator{\sum_{k=0}^{\dist{\sit}-1}}
\sqgroup[\big]{-\xi\selection(\sitto{k})\adddelta\submartin(\sitto{k})(\sitat{k+1})
-\xi^2\selection(\sitto{k})^2(\adddelta\submartin(\sitto{k})(\sitat{k+1}))^2}\\
&=-\xi\average{\submartin}(\sit)\smashoperator{\sum_{k=0}^{\dist{\sit}-1}}\selection(\sitto{k})
-\xi^2\smashoperator{\sum_{k=0}^{\dist{\sit}-1}}
\selection(\sitto{k})(\adddelta\submartin(\sitto{k})(\sitat{k+1}))^2\\
&\geq\xi\epsilon\smashoperator{\sum_{k=0}^{\dist{\sit}-1}}\selection(\sitto{k})
-\xi^2B^2\smashoperator{\sum_{k=0}^{\dist{\sit}-1}}\selection(\sitto{k})\\
&=\xi(\epsilon-\xi B^2)\smashoperator{\sum_{k=0}^{\dist{\sit}-1}}\selection(\sitto{k})
=\frac{\epsilon^2}{4B^2}\smashoperator{\sum_{k=0}^{\dist{\sit}-1}}\selection(\sitto{k}),
\end{aligned}
\end{multline*}
where the first equality holds because \(\selection^2=\selection\).
Choosing \(K\coloneqq\frac{\epsilon^2}{4B^2}\sum_{k=0}^{\dist{\sit}-1}\selection(\sitto{k})\) in Equation~\eqref{eq:K} now completes the proof of the second statement.

We now prove the last statement, dealing with the {\compy} of~\(\process_\submartin\).
Since \(\adddelta\submartin\) and \(\xi\) are assumed to be {\comp} and \(\selection\) is assumed to be recursive, we infer from Equation~\eqref{eq:multprocess:submartingale} that the multiplier process~\(\multprocess_\submartin\) is {\comp} too.
If we now invoke Proposition~\ref{prop:computable:from:multiplier}, we find that \(\process_\submartin=\mint[\multprocess_\submartin]\) is therefore {\comp} as well.
\end{proof}

Theorems~\ref{thm:consistency} and~\ref{thm:well:calibrated} provide statements that hold `almost surely for a forecasting system~\(\frcstsystem\)': any path is \emph{almost surely} random, and the limsup average gain for Sceptic along any path---where the average is taken over any recursive selection of situations---for betting on a fixed gamble with rates provided by Forecaster is \emph{almost surely} non-positive; the corresponding liminf average gain for Forecaster is \emph{almost surely} non-negative.
The following theorem connects these two properties, and at the same time gets rid of their `\emph{almost} sure' flavour: if we concentrate on a \emph{specific} path that is random, then the limsup average gain for Sceptic along \emph{that} path---where the average is again taken over any recursive selection of situations---for betting on a fixed gamble with rates provided by Forecaster is \emph{surely} non-positive.
Interestingly, and in contrast with Theorems~\ref{thm:consistency} and~\ref{thm:well:calibrated}, we need the forecasting system to be {\comp} for our argumentation to work.
We're convinced that this {\compy} requirement can be weakened considerably (but not dropped altogether), but we refrain from going in that direction here, because doing so would come at the cost of an even more abstract formulation, and because the version we state below suffices for our present purposes.

\begin{theorem}[Relative frequencies for selection processes]\label{thm:well:calibrated:general}
Consider a {\comp} forecasting system \(\frcstsystem\colon\sits\to\imprecisefrcsts\) and a path~\(\pth\in\pths\) that is \(\allowables\)-random for~\(\frcstsystem\).
Let \((I_1,\dots,I_n,\dots)\) be the corresponding sequence of interval forecasts \(I_n=\pinterval[n]\coloneqq\frcstsystem(\pthto{n-1})\) for the path~\(\pth\).
If \(\selection\colon\sits\to\outcomes\) is a recursive selection process such that \(\lim_{n\to\infty}\sum_{k=0}^n\selection(\pthto{k})=\infty\), then
\begin{equation*}
\liminf_{n\to\infty}
\dfrac{\sum_{k=0}^{n-1}\selection(\pthto{k})\sqgroup[\big]{h(\pth_{k+1})-\lex_{I_{k+1}}(h)}}
{\sum_{k=0}^{n-1}\selection(\pthto{k})}
\geq0
\text{ for any gamble~\(h\) on~\(\outcomes\)}.
\end{equation*}
\end{theorem}
\noindent
In short, the proof by contradiction proceeds in two steps.
First, we argue that if the inequality isn't satisfied for some gamble~\(h\), then we can always find another rational-valued gamble~\(h'\) close to it for which the inequality also fails.
In a second step, we show that we can use the gamble~\(h'\) to construct a positive test supermartingale~\(\process_\submartin\) for~\(\frcstsystem\) of the type considered in Lemma~\ref{lem:martinwlln}, and that this~\(\process_\submartin\) is unbounded on~\(\pth\).
Since this~\(\process_\submartin\) is computable because~\(\frcstsystem\) and~\(h'\) are, this contradicts the \(\poscallowables\)-randomness, and therefore also the \(\allowables\)-randomness, of~\(\pth\).  

\begin{proof}[Proof of Theorem~\ref{thm:well:calibrated:general}]
By Proposition~\ref{prop:nestedallowables}, it suffices to prove the result for the special case that \(\allowables\) is the set~\(\poscallowables\) of all {\comp} positive test processes.
Assume {\itshape ex absurdo} that the inequality isn't satisfied for some gamble~\(h\) on~\(\outcomes\).
Then there is some rational~\(0<\epsilon<1\) such that
\begin{equation*}
\liminf_{n\to\infty}
\dfrac{\sum_{k=0}^{n-1}\selection(\pthto{k})\sqgroup[\big]{h(\pth_{k+1})-\lex_{I_{k+1}}(h)}}
{\sum_{k=0}^{n-1}\selection(\pthto{k})}
<-2\epsilon.
\end{equation*}
Let \(h'\) be any rational-valued gamble on~\(\outcomes\) such that \(h\leq h'\leq h+\epsilon\).
Then for all \(k\in\naturalswithzero\), we find that \(h+\epsilon-\lex_{I_{k+1}}(h)\geq h'-\lex_{I_{k+1}}(h)\geq h'-\lex_{I_{k+1}}(h')\), using coherence property~\ref{axiom:coherence:monotonicity} for the last inequality.
It therefore follows that
\begin{align*}
\liminf_{n\to\infty}\average{\submartin^\frcstsystem_{h'}}(\pthto{n})
&=\liminf_{n\to\infty}
\dfrac{\sum_{k=0}^{n-1}\selection(\pthto{k})\sqgroup[\big]{h'(\pth_{k+1})-\lex_{I_{k+1}}(h')}}
{\sum_{k=0}^{n-1}\selection(\pthto{k})}\\
&\leq\liminf_{n\to\infty}
\dfrac{\sum_{k=0}^{n-1}\selection(\pthto{k})\sqgroup[\big]{h(\pth_{k+1})-\lex_{I_{k+1}}(h)}}
{\sum_{k=0}^{n-1}\selection(\pthto{k})}+\epsilon
<-\epsilon,
\end{align*}
where we used Equation~\eqref{eq:average:over:selected:increments} for the equality.

Let \(B\coloneqq\max\set{1,\varnorm{h'}}>0\).
Then on the one hand, we infer from Equation~\eqref{eq:uniform:bound:on:increments} that \(B\) is a uniform real bound on~\(\adddelta\submartin^\frcstsystem_{h'}\), meaning that \(\smash{\abs{\adddelta\submartin^\frcstsystem_{h'}(\sit)}\leq B}\) for all situations~\(\sit\in\sits\).
On the other hand, we also have that \(0<\epsilon<1\leq B\).
Now, consider the positive test supermartingale~\(\process_\submartin\) for~\(\frcstsystem\) introduced in Lemma~\ref{lem:martinwlln}, with in particular our present choice for~\(B\), \(\submartin\coloneqq\submartin^\frcstsystem_{h'}\) and \(\smash{\xi\coloneqq\frac{\epsilon}{2B^2}<\frac{1}{2B}}\).

We start by showing that \(\process_\submartin\) is unbounded on~\(\pth\).
For any~\(m\in\naturalswithzero\), since we know that \(\liminf_{n\to\infty}\average{\submartin^\frcstsystem_{h'}}(\pthto{n})<-\epsilon\), there is some~\(n_m\geq m\) such that \(\average{\submartin^\frcstsystem_{h'}}(\pthto{n_m})<-\epsilon\) and therefore also, because of Lemma~\ref{lem:martinwlln},
\begin{equation}\label{eq:slln:intermediate}
\process_\submartin(\pthto{n_m})
\geq\exp\group[\bigg]{\frac{\epsilon^2}{4B^2}\smashoperator{\sum_{k=0}^{n_m-1}}\selection(\pthto{k})}
\geq\exp\group[\bigg]{\frac{\epsilon^2}{4B^2}\smashoperator{\sum_{k=0}^{m-1}}\selection(\pthto{k})}.
\end{equation}
Now, consider any real~\(R>0\).
Since \(\lim_{n\to\infty}\sum_{k=0}^n\selection(\pthto{k})=\infty\), there is some~\(m_R\in\naturalswithzero\) such that \(\smash{\exp\big(\frac{\epsilon^2}{4B^2}\sum_{k=0}^{m_R-1}\selection(\pthto{k})\big)>R}\).
Due to Equation~\eqref{eq:slln:intermediate}, this implies that \(\smash{\process_\submartin(\pthto{r_R})}>R\), with \(r_R\coloneqq n_{m_R}\).
So, in conclusion, we find that for any~\(R>0\), there is some~\(r_R\in\naturalswithzero\) such that \(\smash{\process_\submartin(\pthto{r_R})}>R\).
This tells us that the positive test supermartingale~\(\process_\submartin\) is indeed unbounded on \(\pth\).

If we can now show that \(\process_\submartin\) is also {\comp}, this will contradict the assumed \(\poscallowables\)-randomness of \(\pth\) for~\(\frcstsystem\).
Recall from the argumentation above that \(\process_\submartin\) is the positive test supermartingale constructed in Lemma~\ref{lem:martinwlln}, for the particular choices \(\submartin=\submartin^\frcstsystem_{h'}\), \(B=\max\set{1,\varnorm{h'}}\) and \(\xi=\frac{\epsilon}{2B^2}\).
Since the gamble~\(h'\) is rational, so is the real number~\(\varnorm{h'}=\abs{h'(1)-h'(0)}\).
Hence, since \(\epsilon\) is rational, the real numbers~\(B\) and~\(\xi\) are rational and therefore definitely {\comp}.
Since \(h'\) is rational, the inequalities \(h'(1)\geq h'(0)\) and \(h'(1)\leq h'(0)\) are decidable.
Therefore, and because the computability of \(\frcstsystem\) means that \(\lfrcstsystem\) and \(\ufrcstsystem\) are {\comp}, Equations~\eqref{eq:local:lower} and~\eqref{eq:local:linear} imply that the real process \(\lex_{\frcstsystem(\sit)}(h')\), \(\sit\in\sits\), is {\comp}.
For any~\(x\in\outcomes\), since we know from Equation~\eqref{eq:lln:local:increments} that \(\adddelta\submartin^\frcstsystem_{h'}(\sit)(\xval)=h'(x)-\lex_{\frcstsystem(\sit)}(h')\), the rationality of \(h'\) therefore implies that \(\adddelta\submartin^\frcstsystem_{h'}(\sit)(\xval)\), \(\sit\in\sits\), is {\comp} as well.
By definition, this means that \(\adddelta\submartin^\frcstsystem_{h'}\) is {\comp}.
So we have found that \(\adddelta\submartin^\frcstsystem_{h'}\) and \(\xi\) are {\comp}.
Since in addition~\(\selection\) is assumed to be recursive, we infer from Lemma~\ref{lem:martinwlln} that the positive test supermartingale~\(\process_\submartin\) is indeed {\comp}.
\end{proof}

If we take a closer look at our argument in this proof, we see that it allows us to derive the desired result for~\(\allowables\)-randomness, but that it does not work for Schnorr random paths~\(\pth\).
Indeed, it follows from the assumptions that the {\comp} test supermartingale \(\process_\supermartin\) from Lemma~\ref{lem:martinwlln} that does the heavy lifting in the proof, is unbounded on~\(\pth\), but not necessarily computably so.
The argument shows that a \emph{sufficient condition} for the {\comp} unboundedness of~\(\process_\supermartin\) on~\(\pth\) is that the map
\begin{equation}\label{eq:number:of:selected:outcomes}
\zeta\colon\naturalswithzero\to\naturalswithzero\colon n\mapsto\zeta(n)\coloneqq\sum_{k=0}^{n-1}\selection(\pthto{k}),
\end{equation}
which gives the number~\(\zeta(n)\) of selected outcomes along the segments \(\pthto{n}\) of the Schnorr random path~\(\pth\), should be recursive.
But, even though the selection process~\(\selection\) is assumed to be recursive, the corresponding~\(\zeta\) in general won't be, simply because the random path~\(\pth\) typically isn't recursive.

This analysis points to a fairly direct way of salvaging the result for Schnorr-random paths as well: when we make sure that the selection process~\(\selection\) depends not on the situations~\(\sit\) themselves, but only on their depth \(\dist{\sit}\) in the event tree, so not on the history of the outcomes but only on the time that has passed, then \(\zeta\) will be recursive as soon as~\(\selection\) is.
This brings us to a new formulation, where we replace the selection process~\(\selection\colon\sits\to\outcomes\) by the simpler notion of a \emph{selection function} \(\selectionfunction\colon\naturals\to\outcomes\).
At any `time point' \(k\in\naturals\), if \(\selectionfunction(k)=1\), then the outcome~\(\pthat{k}\) is selected along the path~\(\pth\), and if \(\selectionfunction(k)=0\), it isn't.

\begin{theorem}[Relative frequencies for selection functions]\label{thm:well:calibrated:general:schnorr}
Consider a {\comp} forecasting system \(\frcstsystem\colon\sits\to\imprecisefrcsts\) and a path~\(\pth\in\pths\) that is \(\allowables\)-random for~\(\frcstsystem\).
Let \((I_1,\dots,I_n,\dots)\) be the corresponding sequence of interval forecasts~\(I_n=\pinterval[n]\coloneqq\frcstsystem(\pthto{n-1})\) for the path~\(\pth\).
If \(\selectionfunction\) is a recursive selection function such that \(\lim_{n\to\infty}\sum_{k=1}^n\selectionfunction(k)=\infty\), then
\begin{equation*}
\liminf_{n\to\infty}
\dfrac{\sum_{k=1}^{n}\selectionfunction(k)\sqgroup[\big]{h(\pth_{k})-\lex_{I_{k}}(h)}}
{\sum_{k=1}^{n}\selectionfunction(k)}
\geq0
\text{ for any gamble~\(h\) on~\(\outcomes\)}.
\end{equation*}
The same conclusion continues to hold when \(\pth\) is Schnorr random for~\(\frcstsystem\).
\end{theorem}

\begin{proof}
By Proposition~\ref{prop:random:implies:schnorrandom}, it clearly suffices to prove the result for Schnorr randomness.
Consider the selection process~\(\selection\colon\sits\to\outcomes\) defined by \(\selection(\sit)\coloneqq\selectionfunction(\abs{\sit}+1)\) for all~\(\sit\in\sits\).
Since \(\selectionfunction\) is recursive and \(\lim_{n\to\infty}\sum_{k=1}^n\selectionfunction(k)=\infty\), it also follows that \(S\) is recursive and that \(\lim_{n\to\infty}\sum_{k=0}^n\selection(\pthto{k})=\infty\).

Assume \emph{ex absurdo} that the inequality isn't satisfied.
This implies that there is some rational \(0<\epsilon<1\) such that
\begin{align*}
-2\epsilon&>\liminf_{n\to\infty}
\dfrac{\sum_{k=1}^{n}\selectionfunction(k)\sqgroup[\big]{h(\pth_{k})-\lex_{I_{k}}(h)}}
{\sum_{k=1}^{n}\selectionfunction(k)}
=\liminf_{n\to\infty}
\dfrac{\sum_{k=0}^{n-1}\selection(\pthto{k})\sqgroup[\big]{h(\pth_{k+1})-\lex_{I_{k+1}}(h)}}
{\sum_{k=0}^{n-1}\selection(\pthto{k})}.
\end{align*}
As we have shown in the proof of Theorem~\ref{thm:well:calibrated:general}, this implies that there is a {\comp} positive test supermartingale~\(\process_\submartin\) for~\(\frcstsystem\) and a {\comp} real number~\(B>0\) such that, for all \(m\in\naturalswithzero\), there is some~\(n_m\geq m\) such that
\begin{equation*}
\process_\submartin(\pthto{n_m})
\geq\exp\group[\bigg]{\frac{\epsilon^2}{4B^2}\smashoperator{\sum_{k=0}^{n_m-1}}\selection(\pthto{k})}
=\exp\group[\bigg]{\frac{\epsilon^2}{4B^2}\smashoperator{\sum_{k=1}^{n_m}}\selectionfunction(k)}
=\tau_\selectionfunction(n_m),
\end{equation*}
where we have defined the map~\(\realordering_\selectionfunction\colon\naturalswithzero\to\nonnegreals\) by
\begin{equation*}
\realordering_\selectionfunction(n)
\coloneqq\exp\group[\bigg]{\frac{\epsilon^2}{4B^2}\smashoperator{\sum_{k=1}^{n}}\selectionfunction(k)}
\text{ for all~\(n\in\naturalswithzero\)}.
\end{equation*}
So for all~\(m\in\naturalswithzero\), there is some~\(n_m\geq m\) such that \(\process_\supermartin(\pthto{n_m})\geq\realordering_\selectionfunction(n_m)\), which implies that \(\sup_{n\geq m}\sqgroup{\process_\supermartin(\pthto{n})-\tau_\selectionfunction(n)}\geq0\). Hence,
\begin{equation*}
\limsup_{n\to\infty}\sqgroup{\process_\supermartin(\pthto{n})-\tau_\selectionfunction(n)}
=\inf_{m\in\naturalswithzero}\sup_{n\geq m}\sqgroup{\process_\supermartin(\pthto{n})-\tau_\selectionfunction(n)}
\geq0.
\end{equation*}
Since \(\epsilon\) is rational, \(B>0\) is {\comp} and \(\selectionfunction\) is recursive, we also know that \(\realordering_\selectionfunction\) is {\comp}.
Furthermore, \(\tau_\selectionfunction\) is non-decreasing because \(\sum_{k=1}^{n}\selectionfunction(k)\) is non-decreasing in~\(n\), and it is unbounded because \(\lim_{n\to\infty}\sum_{k=0}^n\selectionfunction(k)=\infty\) and \(\epsilon>0\).
We conclude that \(\realordering_\selectionfunction\) is a real growth function such that \(\limsup_{n\to\infty}\sqgroup{\process_\supermartin(\pthto{n})-\tau_\selectionfunction(n)}\geq0\).
Proposition~\ref{prop:computably:unbounded:various}\ref{it:computably:unbounded:real} then guarantees that the {\comp} test supermartingale~\(\process_\supermartin\) for~\(\frcstsystem\) is computably unbounded on~\(\pth\), which contradicts the assumed Schnorr randomness of~\(\pth\) for~\(\frcstsystem\).
\end{proof}

\section{Constant interval forecasts}\label{sec:constantintervalforecasts}
From now on, we turn to the special case where the interval forecasts~\(I\in\imprecisefrcsts\) are constant, and don't depend on the already observed outcomes.
This leads to a generalisation of the classical case~\(I=\set{\nicefrac{1}{2}}\) of the randomness associated with a fair coin.

In the rest of this section, unless explicitly stated to the contrary, \(\allowables\) is any arbitrary but fixed set of allowable test processes.
For any interval \(I\in\imprecisefrcsts\), we denote by~\(\constantfrcstsystem[I]\colon\sits\to\imprecisefrcsts\) the corresponding so-called \emph{stationary} forecasting system that assigns the same interval forecast~\(I\) to all situations:
\begin{equation*}
\constantfrcstsystem[I](\sit)
\coloneqq I
\text{ for all~\(\sit\in\sits\).}
\end{equation*}

In order to investigate the mathematical properties of imprecise randomness, it will be helpful to associate, with any path~\(\pth\), the collection of all interval forecasts for which the corresponding stationary forecasting system makes \(\pth\) \(\allowables\)-random:
\begin{equation*}
\constantrandom[\pth]{\allowables}
\coloneqq\cset{I\in\imprecisefrcsts}{\constantfrcstsystem[I]\in\random{\allowables}}
=\cset{I\in\imprecisefrcsts}{\constantfrcstsystem[I]\text{ makes \(\pth\) \(\allowables\)-random}},
\end{equation*}
and we use the special notations \(\constantposcrandom[\pth]\), \(\constantcrandom[\pth]\), \(\constantmultmlrandom[\pth]\) and \(\constantmlrandom[\pth]\) in the cases that \(\allowables\) is equal to~\(\poscallowables\), \(\callowables\), \(\multmlallowables\) and \(\mlallowables\), respectively.
Similarly,
\begin{equation*}
\constantschnorrrandom
\coloneqq\cset{I\in\imprecisefrcsts}{\constantfrcstsystem[I]\in\schnorrrandom}
=\cset{I\in\imprecisefrcsts}{\constantfrcstsystem[I]\text{ makes \(\pth\) Schnorr random}}.
\end{equation*}
Proposition~\ref{prop:random:implies:schnorrandom} and Equation~\eqref{eq:randomness:order:with:schnorr} imply that
\begin{equation}\label{eq:random:implies:schnorrandom}
\constantrandom{\allowables}\subseteq\constantschnorrrandom
\text{ and }
\constantmlrandom[\pth]
\subseteq\constantmultmlrandom[\pth]
\subseteq\constantcrandom[\pth]
=\constantposcrandom[\pth]
\subseteq\constantschnorrrandom[\pth].
\end{equation}
Most of our efforts in this section will be devoted to investigating the mathematical structure of these special sets of interval forecasts.

As immediate consequences of the results proved earlier in Sections~\ref{sec:randomness} and~\ref{sec:schnorr:randomness}, we find that all these sets of interval forecasts associated with a random path are non-empty and increasing.

\begin{proposition}[Non-emptiness]\label{prop:constantcalibrated:top}
For all~\(\pth\in\pths\), \(\frcsts\in\constantrandom{\allowables}\subseteq\constantschnorrrandom\), so any sequence of outcomes~\(\pth\) has at least one stationary forecast that makes it \(\allowables\)-random and therefore also Schnorr random: \(\constantrandom{\allowables}\neq\emptyset\) and \(\constantschnorrrandom\neq\emptyset\).
\end{proposition}

\begin{proof}
This is an immediate consequence of Proposition~\ref{prop:vacuous}, with \(\constantfrcstsystem[\frcsts]=\vacfrcstsystem\in\random{\allowables}\), and Equation~\eqref{eq:random:implies:schnorrandom}.
\end{proof}

\begin{proposition}[Increasingness]\label{prop:constantcalibrated:increasing}
For all~\(\pth\in\pths\) and any~\(I,J\in\imprecisefrcsts\):
\begin{enumerate}[label=\upshape(\roman*),leftmargin=*,noitemsep,topsep=0pt]
\item if \(I\in\constantrandom{\allowables}\) and \(I\subseteq J\), then \(J\in\constantrandom{\allowables}\);
\item if \(I\in\constantschnorrrandom\) and \(I\subseteq J\), then \(J\in\constantschnorrrandom\).
\end{enumerate}
\end{proposition}

\begin{proof}
This follows from Propositions~\ref{prop:nestedfrcstsystems} and~\ref{prop:nestedfrcstsystems:schnorrrandom}, because \(I\subseteq J\) implies \(\constantfrcstsystem[I]\subseteq\constantfrcstsystem[J]\).
\end{proof}

\begin{proposition}\label{prop:constantcalibrated:nestedallowables}
Consider two sets \(\allowables,\allowables'\) of allowable test processes such that \(\allowables'\subseteq\allowables\).
Then \(\constantrandom{\allowables}\subseteq\constantrandom{\allowables'}\subseteq\constantschnorrrandom\).
\end{proposition}

\begin{proof}
This follows immediately from Proposition~\ref{prop:nestedallowables} and Equation~\eqref{eq:random:implies:schnorrandom}.
\end{proof}

\subsection{{\Comp} stochasticity}\label{sec:computable:stochasticity}
Before we continue our study of the structure of the sets of interval forecasts associated with a given random path, it will be helpful to make a small detour, and to consider the behaviour of relative frequencies along random paths.
Interestingly, Theorem~\ref{thm:well:calibrated:general} implies the consistency property in Corollary~\ref{cor:well:calibrated:constant} below, which is a counterpart in our more general context of the notion of \emph{{\comp} stochasticity} or \emph{Church randomness} in the precise fair-coin case where \(I=\set{\nicefrac{1}{2}}\) \cite{ambosspies2000}.
However, quite remarkably, and seemingly in contrast with Theorem~\ref{thm:well:calibrated:general}, this corollary does not impose any {\compy} requirements on the interval forecast~\(I\).

{\Comp} stochasticity, or Church randomness, is a notion that goes back to Alonzo Church's account of randomness \cite{church1940:random:sequence}.
He required of a random path~\(\pth\) that for any recursive selection process~\(\selection\) such that \(\sum_{k=0}^n\selection(\pthto{k})\to\infty\),
\begin{equation*}
\lim_{n\to\infty}\frac{\sum_{k=0}^{n-1}\selection(\pthto{k})\pth_{k+1}}{\sum_{k=0}^{n-1}\selection(\pthto{k})}=\frac12.
\end{equation*}
In other words, the relative frequencies of the ones---the successes---in the outcomes that \(\selection\) selects along the random path~\(\pth\) should converge to the constant probability \(\nicefrac{1}{2}\) of a success.
It is well-known that all paths that are computably random---and therefore also all {\ML} random paths---for a stationary forecast~\(I=\set{\nicefrac{1}{2}}\) are also computably stochastic, or Church random; see for instance Refs.~\cite{ambosspies2000,wang1996:phdthesis}.

In our generalisation, we will see that our notions of randomness no longer necessarily imply such convergence, but we're still able to conclude that the limits inferior and superior of the relative frequencies of the successes in the selected outcomes of a random path must lie in the forecast interval.

\begin{corollary}[Church randomness]\label{cor:well:calibrated:constant}
Consider any path~\(\pth\in\pths\) and any constant interval forecast~\(I=\pinterval\in\constantrandom{\allowables}\) that makes \(\pth\) \(\allowables\)-random.
Then for any recursive selection process~\(\selection\colon\sits\to\outcomes\) such that \(\sum_{k=0}^n\selection(\pthto{k})\to\infty\):
\begin{equation*}
\lp
\leq\liminf_{n\to\infty}
\frac{\sum_{k=0}^{n-1}\selection(\pthto{k})\pth_{k+1}}
{\sum_{k=0}^{n-1}\selection(\pthto{k})}
\leq\limsup_{n\to\infty}
\frac{\sum_{k=0}^{n-1}\selection(\pthto{k})\pth_{k+1}}
{\sum_{k=0}^{n-1}\selection(\pthto{k})}
\leq\up.
\end{equation*}
\end{corollary}

\begin{proof}
First, assume that \(I\) is {\comp}.
It follows from Proposition~\ref{prop:IcompIffGammaComp} that \(\constantfrcstsystem[I]\) is {\comp} as well.
Furthermore, if we let \(\indsing{1}(x)\coloneqq x\) for all~\(x\in\outcomes\), then \(\indsing{1}\) and \(-\indsing{1}\) are clearly gambles on~\(\outcomes\).
The first and last inequality now follow from Theorem~\ref{thm:well:calibrated:general}, by successively choosing \(f\coloneqq\indsing{1}\) and \(f\coloneqq-\indsing{1}\), respectively, since \(\lex_I(\indsing{1})=\lp\) and \(\lex_I(-\indsing{1})=-\up\).
The second inequality is a standard property of limits inferior and superior.

If \(I\) isn't {\comp}, then for any~\(\epsilon>0\), since all rational numbers are {\comp}, there is some {\comp} \(J=\qinterval\in\imprecisefrcsts\) such that \(\lp-\epsilon\leq\lptoo\leq\lp\leq\up\leq\uptoo\leq\up+\epsilon\).
Since \(I\subseteq J\), it follows from Proposition~\ref{prop:constantcalibrated:increasing} that also \(J\in\constantrandom{\allowables}\).
Since, moreover, \(J\) is {\comp}, it follows from the first part of the proof that
\begin{equation*}
\lp-\epsilon\leq\lptoo
\leq\liminf_{n\to\infty}
\frac{\sum_{k=0}^{n-1}S(\pthto{k})\pth_{k+1}}{\sum_{k=0}^{n-1}S(\pthto{k})}
\leq\limsup_{n\to\infty}
\frac{\sum_{k=0}^{n-1}S(\pthto{k})\pth_{k+1}}{\sum_{k=0}^{n-1}S(\pthto{k})}
\leq\uptoo\leq\up+\epsilon.
\end{equation*}
Since \(\epsilon>0\) is arbitrary, this completes the proof.
\end{proof}

For paths that are (only) Schnorr random, we will discuss below that this result needn't hold, but we can prove a weaker result, whose proof is completely similar---and therefore omitted---but now based on Theorem~\ref{thm:well:calibrated:general:schnorr}.

\begin{corollary}[Weak Church randomness]\label{cor:well:calibrated:constant:schnorr}
Consider any path~\(\pth\in\pths\) and any constant interval forecast~\(I=\pinterval\in\constantrandom{\allowables}\) that makes \(\pth\) \(\allowables\)-random.
Then for any recursive selection function \(\selectionfunction\) such that \(\lim_{n\to\infty}\sum_{k=0}^n\selectionfunction(k)=\infty\):
\begin{equation*}
\lp
\leq\liminf_{n\to\infty}
\frac{\sum_{k=1}^{n}\selectionfunction(k)\pth_{k}}
{\sum_{k=1}^{n}\selectionfunction(k)}
\leq\limsup_{n\to\infty}
\frac{\sum_{k=1}^{n}\selectionfunction(k)\pth_{k}}
{\sum_{k=1}^{n}\selectionfunction(k)}
\leq\up.
\end{equation*}
The same conclusion continues to hold when \(I\) makes \(\pth\) Schnorr random.
\end{corollary}

That Corollary~\ref{cor:well:calibrated:constant} needn't hold for Schnorr randomness, is in accordance with the fact that, in the particular fair-coin case where \(I=\set{\nicefrac{1}{2}}\), Schnorr randomness is known not to imply {\comp} stochasticity either.
This was in fact shown by Yongge Wang \cite{wang1996:phdthesis}, who proved the existence of a Schnorr random path~\(\wangpth\) and a {\comp} test martingale~\(\wangmartin\) for~\(\constantfrcstsystem[\nicefrac12]\) such that
\begin{enumerate}[label=\upshape(\roman*),leftmargin=*,noitemsep,topsep=0pt]
\item\label{it:wang:unbounded} \(\wangmartin\) is unbounded---but not computably so---on~\(\wangpth\), also implying that \(\wangpth\) isn't computably random;
\item\label{it:wang:martingale} for all~\(\sit\in\sits\), either \(\group{\forall x\in\outcomes}\wangmartin(\sit x)=2x\wangmartin(\sit)\) or \(\group{\forall x\in\outcomes}\wangmartin(\sit x)=\wangmartin(\sit)\);
\end{enumerate}
and as a consequence also
\begin{enumerate}[label=\upshape(\roman*),leftmargin=*,noitemsep,topsep=0pt,resume]
\item\label{it:wang:martingale:on:path} if \(\wangmartin(\wangpthto{n})=2\wangpthat{n}\wangmartin(\wangpthto{n-1})\) then \(\wangpthat{n}=1\), for all~\(n\in\naturals\).
\end{enumerate}
One immediate conclusion we can draw from these conditions, is that \(\wangmartin\) remains positive on~\(\wangpth\), so \(\wangmartin(\wangpthto{n})>0\) for all \(n\in\naturalswithzero\), simply because \ref{it:wang:martingale} implies that if \(\wangmartin\) ever becomes zero, it remains zero, and can therefore then never become unbounded on~\(\pth\), contradicting~\ref{it:wang:unbounded}.
Another conclusion we can draw from~\ref{it:wang:martingale}, is that~\(\wangmartin\) assumes values in the set~\(\set{0}\cup\cset{2^m}{m\in\naturalswithzero}\).
This implies that, in situations~\(\sit\) such that \(\wangmartin(\sit)>0\), it is decidable which of the two multiplication rules applies in~\ref{it:wang:martingale}.
Hence, the selection process~\(\wangselection\), defined by
\begin{equation*}
\wangselection(\sit)
\coloneqq
\begin{cases}
1&\text{if \(\wangmartin(\sit1)=2\wangmartin(\sit)\) and \(\wangmartin(\sit)>0\)}\\
0&\text{if \(\wangmartin(\sit1)=\wangmartin(\sit)\)}
\end{cases}
\quad\text{for all~\(\sit\in\sits\)},
\end{equation*}
is recursive.
In combination with~\ref{it:wang:martingale:on:path}, this implies that
\begin{equation}\label{eq:wang:one}
\wangselection(\wangpthto{n-1})=1\then\wangpthat{n}=1
\text{ for all~\(n\in\naturals\)}.
\end{equation}
Since it follows from~\ref{it:wang:unbounded} that the first multiplication rule in~\ref{it:wang:martingale} must apply an infinite number of times on~\(\wangpth\), we infer from the conclusion~\eqref{eq:wang:one} that~\(\wangselection\) selects a subsequence of ones from~\(\wangpth\), so the corresponding sequence of relative frequencies on this recursively selected subsequence converges to~\(1\), thus violating {\comp} stochasticity: \emph{Schnorr randomness does not imply {\comp} stochasticity}.

It also follows from these considerations that \(\wangmartin\) either doubles or remains constant on~\(\wangpth\), and that it doubles precisely in those situations~\(\wangpthto{n}\) where~\(\wangselection(\wangpthto{n})=1\).
Hence, if we let \(\wangselectionsum(n)\coloneqq\sum_{k=0}^{n-1}\wangselection(\wangpthto{k})\) in accordance with Equation~\eqref{eq:number:of:selected:outcomes}, then
\begin{equation*}
\wangmartin(\wangpthto{n})=2^{\wangselectionsum(n)}
\text{ for all~\(n\in\naturalswithzero\)}.
\end{equation*}
The map~\(\wangselectionsum\) can't be recursive, because if it were, \(\wangmartin\) would be computably unbounded on~\(\wangpth\), contradicting the Schnorr randomness of~\(\wangpth\).
So, we see that the sufficient condition for `convergence' that we mentioned following the proof of Theorem~\ref{thm:well:calibrated:general}, namely the recursive character of~\(\zeta\) in Equation~\eqref{eq:number:of:selected:outcomes}, is perfectly at ease with Wang's example, as it isn't satisfied for this particular case~\(\selectionsum=\wangselectionsum\).

Observe, by the way, that the recursive character of~\(\zeta\) in Equation~\eqref{eq:number:of:selected:outcomes} is equivalent to the recursive character of the behaviour
\begin{equation*}
\selectionfunction_{\selection,\pth}\colon\naturals\to\outcomes\colon n\mapsto\selectionfunction(n)\coloneqq\selection(\pthto{n-1})
\end{equation*}
of the selection process~\(\selection\) on the random path~\(\pth\).
It should therefore not be surprising that recursive selection \emph{functions}, such as this \(\selectionfunction_{\selection,\pth}\), play such an important part in our Theorem~\ref{thm:well:calibrated:general:schnorr} and Corollary~\ref{cor:well:calibrated:constant:schnorr}.
This also means that if we were to strengthen the requirements on the selection processes~\(\selection\) in Theorem~\ref{thm:well:calibrated:general} and Corollary~\ref{cor:well:calibrated:constant} from `being recursive' to `being recursive and displaying recursive behaviour on the path under consideration', then the corresponding (weaker) {\comp} stochasticity result would still hold for all Schnorr random paths.
This is essentially what we do in Theorem~\ref{thm:well:calibrated:general:schnorr} and Corollary~\ref{cor:well:calibrated:constant:schnorr}.
Any criticism of Schnorr randomness along the lines of Wang's argument \cite{wang1996:phdthesis} will therefore have to include an argumentation for why such a strengthening of the requirements on the selection processes is unreasonable or undesirable, or alternatively, why selection processes rather than selection functions appear in the requirements.

\subsection{The structure of the interval forecasts that make a path random}
We return to our study of the mathematical structure behind constant interval forecasts.
Our digression about Church randomness around Corollary~\ref{cor:well:calibrated:constant} now displays its usefulness, because it allows us to prove the following consistency result: any collection of constant interval forecasts that make some path random must have a non-empty intersection.

\begin{proposition}\label{prop:constantcalibrated:nonempty:intersection}
For any~\(\pth\in\pths\), \(\constantrandom{\allowables}\) and \(\constantschnorrrandom\) have the intersection property: for any collection~\(\imprecisefrcsts'\subseteq\imprecisefrcsts\) of interval forecasts:
\begin{enumerate}[label=\upshape(\roman*),leftmargin=*,noitemsep,topsep=0pt]
\item if \(\imprecisefrcsts'\subseteq\constantrandom{\allowables}\), then \(\bigcap\imprecisefrcsts'\neq\emptyset\);
\item if \(\imprecisefrcsts'\subseteq\constantschnorrrandom\), then \(\bigcap\imprecisefrcsts'\neq\emptyset\).
\end{enumerate}
In fact,
\begin{equation}\label{eq:intersection:property}
\sqgroup[\bigg]{\liminf_{n\to\infty}\frac{1}{n}\smashoperator{\sum_{k=1}^n}\pth_k,
\limsup_{n\to\infty}\frac{1}{n}\smashoperator{\sum_{k=1}^n}\pth_k}
\subseteq\bigcap\constantrandom{\allowables}
\subseteq\bigcap\constantschnorrrandom.
\end{equation}
\end{proposition}

\begin{proof}
It clearly suffices to prove the inclusions in Equation~\eqref{eq:intersection:property}, and Proposition~\ref{eq:random:implies:schnorrandom} allows us to concentrate on the first inclusion.
So, for any~\(I=\pinterval\in\constantrandom{\allowables}\), it follows from Corollary~\ref{cor:well:calibrated:constant:schnorr}, with \(\selectionfunction(n)\coloneqq1\) for all~\(n\in\naturals\), that
\begin{equation*}
\lp
\leq\liminf_{n\to\infty}\frac{1}{n}\smashoperator{\sum_{k=1}^n}\pth_k
\leq\limsup_{n\to\infty}\frac{1}{n}\smashoperator{\sum_{k=1}^n}\pth_k
\leq\up.
\end{equation*}
Hence, indeed,
\begin{equation*}
\emptyset
\neq\sqgroup[\bigg]{\liminf_{n\to\infty}\frac{1}{n}\smashoperator{\sum_{k=1}^n}\pth_k,
\limsup_{n\to\infty}\frac{1}{n}\smashoperator{\sum_{k=1}^n}\pth_k}
\subseteq\bigcap_{I\in\constantrandom{\allowables}}I
=\bigcap\constantrandom{\allowables}.
\qedhere
\end{equation*}
\end{proof}

Whether the non-empty closed intervals~\(\bigcap\constantrandom{\allowables}\) and~\(\bigcap\constantschnorrrandom\) themselves also make the path~\(\pth\) \(\allowables\)-random, respectively Schnorr random, depends on the case at hand: we will come across an example in Section~\ref{sec:non-stationarity} where they do (Section~\ref{sec:example:p:q}), and another example where they don't (Section~\ref{sec:example:almost:half}).

We continue our discussion by introducing the following subsets of~\(\frcsts\), which respectively collect the left and right boundaries of the interval forecasts that make a given path~\(\pth\in\pths\) random:
\begin{align*}
\lowerconstantrandom{\allowables}
\coloneqq\cset{\min I}{I\in\constantrandom{\allowables}}
&\text{ and }
\upperconstantrandom{\allowables}
\coloneqq\cset{\max I}{I\in\constantrandom{\allowables}}\\
\lowerconstantschnorrrandom
\coloneqq\cset{\min I}{I\in\constantschnorrrandom}
&\text{ and }
\upperconstantschnorrrandom
\coloneqq\cset{\max I}{I\in\constantschnorrrandom}.
\end{align*}
Proposition~\ref{prop:constantcalibrated:increasing} guarantees that \(\lowerconstantrandom{\allowables}\) and \(\lowerconstantschnorrrandom\) are decreasing sets (down-sets), and that \(\upperconstantrandom{\allowables}\) and \(\upperconstantschnorrrandom\) are increasing (up-sets).
They are therefore all of them subintervals of~\(\frcsts\).
If we also let
\begin{align*}
\lowersmallestrandom{\allowables}
\coloneqq\sup\lowerconstantrandom{\allowables}
=\min\bigcap\constantrandom{\allowables}
&\text{ and }
\uppersmallestrandom{\allowables}
\coloneqq\inf\upperconstantrandom{\allowables}
=\max\bigcap\constantrandom{\allowables}\\
\lowersmallestschnorrrandom
\coloneqq\sup\lowerconstantschnorrrandom
=\min\bigcap\constantschnorrrandom
&\text{ and }
\uppersmallestschnorrrandom
\coloneqq\inf\upperconstantschnorrrandom
=\max\bigcap\constantschnorrrandom,
\end{align*}
then clearly
\begin{align*}
&\lowerconstantrandom{\allowables}=[0,\lowersmallestrandom{\allowables})
\text{ or }
\lowerconstantrandom{\allowables}=[0,\lowersmallestrandom{\allowables}]\\
&\upperconstantrandom{\allowables}=(\uppersmallestrandom{\allowables},1]
\text{ or }
\upperconstantrandom{\allowables}=[\uppersmallestrandom{\allowables},1],
\end{align*}
and similarly for the Schnorr variants.
Proposition~\ref{prop:constantcalibrated:nonempty:intersection} also implies the following consistency property:
\begin{equation*}
\lowersmallestrandom{\allowables}
\leq\uppersmallestrandom{\allowables}
\text{ and }
\lowersmallestschnorrrandom
\leq\uppersmallestschnorrrandom.
\end{equation*}
All of this is illustrated in Figure~\ref{fig:set:filter} for the special, but in no way atypical, case that \(\allowables=\callowables\).

\begin{figure}[ht]
\centering
\begin{tikzpicture}[xscale=5,yscale=2.5]\small
\useasboundingbox (-0.2,-0.4) rectangle (1.2,1.1);
\fill [lightgray] (0.4,0) rectangle (0.65,1);
\draw[densely dotted,semithick] (0,1) -- (0,0) (1,1) -- (1,0) ;
\draw[densely dotted,semithick] (0.4,1) -- (0.4,0) (.65,1) -- (.65,0);
\draw[semithick,->] (-0.1,0) -- (1.1,0) ;
\draw (0,.01) -- (0,-.01) node[below] {\(0\)} ;
\draw (1,.01) -- (1,-.01) node[below] {\(1\)} ;
\draw (.4,.01) -- (.4,-.01) node[below] {\(\lowersmallestrandom[\pth]{\mathrm{C}}\)} ;
\draw (.65,.01) -- (.65,-.01) node[below] {\(\uppersmallestrandom[\pth]{\mathrm{C}}\)} ;
\draw[thick,blue] (0,1) node[circle,inner sep=1pt,fill=green!50!black] {} -- (1,1) node[circle,inner sep=1pt,fill=red] {};
\draw[thick,blue] (0.39,0.9) node[circle,inner sep=1pt,fill=green!50!black] {} -- (1,0.9) node[circle,inner sep=1pt,fill=red] {};
\draw[thick,blue] (0.1,0.8) node[circle,inner sep=1pt,fill=green!50!black] {} -- (0.9,0.8) node[circle,inner sep=1pt,fill=red] {};
\draw[thick,blue] (0,0.7) node[circle,inner sep=1pt,fill=green!50!black] {} -- (0.66,0.7) node[circle,inner sep=1pt,fill=red] {};
\draw[thick,blue] (0.35,0.6) node[circle,inner sep=1pt,fill=green!50!black] {} -- (0.8,0.6) node[circle,inner sep=1pt,fill=red] {};
\draw[thick,blue] (0.39,0.5) node[circle,inner sep=1pt,fill=green!50!black] {} -- (0.66,0.5) node[circle,inner sep=1pt,fill=red] {};
\draw[thick,blue] (0.2,0.4) node[circle,inner sep=1pt,fill=green!50!black] {} -- (0.7,0.4) node[circle,inner sep=1pt,fill=red] {};
\draw[thick,blue] (0.39,0.3) node[circle,inner sep=1pt,fill=green!50!black] {} -- (0.7,0.3) node[circle,inner sep=1pt,fill=red] {};
\draw[thick,blue] (0.35,0.2) node[circle,inner sep=1pt,fill=green!50!black] {} -- (0.7,0.2) node[circle,inner sep=1pt,fill=red] {};
\draw[thick,blue] (0.35,0.1) node[circle,inner sep=1pt,fill=green!50!black] {} -- (0.66,0.1) node[circle,inner sep=1pt,fill=red] {};
\draw[thick,green!50!black] (0,-0.3) node[circle,inner sep=1.5pt,fill] {} -- node[midway,below] {\(\lowerconstantrandom[\pth]{\mathrm{C}}\)} (0.4,-0.3) node[circle,inner sep=1.5pt,fill] {} node[circle,inner sep=1pt,fill=white] {} node[circle,inner sep=.5pt,fill] {};
\draw[thick,red] (0.65,-0.3) node[circle,inner sep=1.5pt,fill] {} node[circle,inner sep=1pt,fill=white] {} node[circle,inner sep=.5pt,fill] {} -- node[midway,below] {\(\upperconstantrandom[\pth]{\mathrm{C}}\)} (1,-0.3) node[circle,inner sep=1.5pt,fill] {};
\end{tikzpicture}
\caption{Some interval forecasts in the set~\(\constantcrandom\) (in blue), and corresponding~\(\lowersmallestrandom[\pth]{\mathrm{C}}\) and~\(\uppersmallestrandom[\pth]{\mathrm{C}}\)}
\label{fig:set:filter}
\end{figure}

It is obvious that, for any~\(I\in\constantrandom{\allowables}\), we have that \(\min I\in\lowerconstantrandom[\pth]{\allowables}\) and \(\max I\in\upperconstantrandom[\pth]{\allowables}\), and similarly for the Schnorr randomness variants.
We're about to prove, as a result of Propositions~\ref{prop:constantcalibrated:intersectioninside:ML}--\ref{prop:constantcalibrated:intersectioninside:schnorr} below, that for weak {\ML} randomness, {\comp} randomness and Schnorr randomness, the converse is also true.
Therefore, in those cases, where \(\allowables\) is equal to \(\callowables\) or \(\multmlallowables\),
\begin{equation}\label{eq:intervals:both:sides}
\left\{
\begin{aligned}
I\in\constantrandom{\allowables}
&\ifandonlyif
\group[\big]{%
\min I\in\lowerconstantrandom{\allowables}
\text{ and }
\max I\in\upperconstantrandom{\allowables}
}\\
I\in\constantschnorrrandom
&\ifandonlyif
\group[\big]{%
\min I\in\lowerconstantschnorrrandom
\text{ and }
\max I\in\upperconstantschnorrrandom
}.
\end{aligned}
\right.
\end{equation}

\begin{proof}[Proof of Equation~\eqref{eq:intervals:both:sides}]
We first give a proof of the converse implication for~\(\allowables\)-random\-ness.
Consider any~\(I=\pinterval\in\imprecisefrcsts\) for which \(\lp\in\lowerconstantrandom{\allowables}\) and \(\up\in\upperconstantrandom{\allowables}\).
That \(\lp\in\lowerconstantrandom{\allowables}\) implies by Proposition~\ref{prop:constantcalibrated:increasing} that also \([\lp,1]\in\constantrandom{\allowables}\).
Similarly, \(\up\in\upperconstantrandom{\allowables}\) implies by Proposition~\ref{prop:constantcalibrated:increasing} that also \([0,\up]\in\constantrandom{\allowables}\).
Propositions~\ref{prop:constantcalibrated:intersectioninside:ML} and~\ref{prop:constantcalibrated:intersectioninside:C} then guarantee that, indeed, \(I=\pinterval=[\lp,1]\cap[0,\up]\in\constantrandom{\allowables}\).

The proof for Schnorr randomness is completely similar, but uses Proposition~\ref{prop:constantcalibrated:intersectioninside:schnorr} rather than Propositions~\ref{prop:constantcalibrated:intersectioninside:ML} and~\ref{prop:constantcalibrated:intersectioninside:C}.
\end{proof}

Propositions~\ref{prop:constantcalibrated:intersectioninside:ML}--\ref{prop:constantcalibrated:intersectioninside:schnorr} below can of course be extended straightforwardly to any finite number of interval forecasts, and they guarantee, together with Proposition~\ref{prop:constantcalibrated:increasing}, that \(\constantcrandom\), \(\constantmultmlrandom\) and \(\constantschnorrrandom\) are \emph{set filters}: increasing sets that are closed under finite intersections.

We have no proof for a corresponding result for {\ML} randomness: it is an open problem whether the set of constant interval forecasts~\(\constantmlrandom\) that make a path~\(\pth\) {\ML} random is closed under finite intersections, and therefore a set filter.

\begin{proposition}\label{prop:constantcalibrated:intersectioninside:ML}
For any~\(\pth\in\pths\), \(\constantmultmlrandom\) is closed under (finite) intersections: for any two interval forecasts~\(I\) and \(J\) in~\(\constantmultmlrandom\), we have that \(I\cap J\in\constantmultmlrandom\).
\end{proposition}
\noindent The idea behind the proof is that we show how to write any {\lscomp} supermartingale multiplier for~\(\constantfrcstsystem[I\cap J]\) as a product of two {\lscomp} supermartingale multipliers, one for~\(\constantfrcstsystem[I]\) and one for~\(\constantfrcstsystem[J]\).

\begin{proof}[Proof of Proposition~\ref{prop:constantcalibrated:intersectioninside:ML}]
Let \(K\coloneqq I\cap J\).
We will prove that \(K\in\constantmultmlrandom\).

Let \(I=\pinterval\) and \(J=\qinterval\).
Because of symmetry, we may assume without loss of generality that \(\lptoo\leq\lp\). 
Furthermore, due to Proposition~\ref{prop:constantcalibrated:nonempty:intersection}, we know that then \(\lp\leq\uptoo\).
If we have that \(I\subseteq J\), then \(I=I\cap J\) and therefore, since \(I\in\constantmultmlrandom\), the result holds trivially. Hence, we may assume without loss of generality that \(\lptoo\leq\lp\leq\uptoo<\up\), which implies that \(K=I\cap J=[\lp,\uptoo]\).

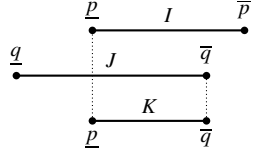
\begin{center}
\begin{tikzpicture}[xscale=5,yscale=2]\footnotesize
\useasboundingbox (0,0) rectangle (1,1);
\draw[thick] (0.4,0.8) node[circle,inner sep=1pt,fill] {} node [above=3pt] {\(\lp\)} -- node[midway,above] {\(I\)} (0.8,0.8) node[circle,inner sep=1pt,fill] {} node[above=2pt] {\(\up\)};
\draw[thick] (0.2,0.5) node[circle,inner sep=1pt,fill] {} node [above=3pt] {\(\lptoo\)} -- node[midway,above] {\(J\)} (0.7,0.5) node[circle,inner sep=1pt,fill] {} node[above=2pt] {\(\uptoo\)};
\draw[thick] (0.4,0.2) node[circle,inner sep=1pt,fill] {} node [below=4pt] {\(\lp\)} -- node[midway,above] {\(K\)} (0.7,0.2) node[circle,inner sep=1pt,fill] {} node[below=5pt] {\(\uptoo\)};
\draw[densely dotted] (0.4,0.2) -- (0.4,0.8) (0.7,0.2) -- (0.7,0.5);
\end{tikzpicture}
\end{center}

Consider any test supermartingale~\(\test\) in~\(\multmlallowabletests[{\constantfrcstsystem[K]}]\), then we must show that \(\test\) remains bounded on~\(\pth\).
We know that there is some {\lscomp} supermartingale multiplier \(\multprocess\) for~\(\constantfrcstsystem[K]\) such that \(\test=\mint\).

Now let \(\multprocess_I\) be the map from situations to gambles on~\(\outcomes\), defined by
\begin{equation*}
\multprocess_I(\sit)(z)
\coloneqq
\begin{cases}
\min\set{\multprocess(\sit)(1),1}
&\text{if \(z=1\)}\\
\max\set{\multprocess(\sit)(0),1}
&\text{if \(z=0\)}
\end{cases}
\quad\text{for all~\(\sit\in\sits\) and \(z\in\outcomes\).}
\end{equation*}
We now show that \(\multprocess_I\) is a supermartingale multiplier for~\(\constantfrcstsystem[I]\).
That it is  non-negative follows from the non-negativity of~\(\multprocess\).
It therefore remains to show that \(\uex_I(\multprocess_I(\sit))\leq1\) for all \(\sit\in\sits\).
To this end, we consider two cases: \(\multprocess(\sit)(0)\leq1\) and \(\multprocess(\sit)(0)>1\).
If \(\multprocess(\sit)(0)\leq1\), then \(\multprocess_I(\sit)\leq1\) and therefore also \(\uex_{I}(\multprocess_I(\sit))\leq1\), by~\ref{axiom:coherence:bounds}.
The case that \(\multprocess(\sit)(0)>1\) is a bit more involved.
For a start, since \(\multprocess(\sit)(0)>1\) implies that \(\multprocess_I(\sit)(1)<\multprocess_I(\sit)(0)\) [because then \(\multprocess_I(\sit)(0)=\multprocess(\sit)(0)>1\), and at the same time always \(\multprocess_I(\sit)(1)\leq1\)], we find that
\begin{equation*}
\uex_{I}(\multprocess_I(\sit))
=\ex_{\lp}(\multprocess_I(\sit))
=\uex_{K}(\multprocess_I(\sit)).
\end{equation*}
Furthermore, since we know that \(\multprocess\) is a supermartingale multiplier for~\(\constantfrcstsystem[K]\) and therefore~\(\uex_K(\multprocess(\sit))\leq1\), \(\multprocess(\sit)(0)>1\) implies that \(\multprocess(\sit)(1)\leq1\), again by~\ref{axiom:coherence:bounds}.
We therefore find that \(\multprocess(\sit)=\multprocess_I(\sit)\).
By combining these two findings, it follows that indeed here also
\begin{equation*}
\uex_{I}(\multprocess_I(\sit))
=\uex_{K}(\multprocess_I(\sit))
=\uex_{K}(\multprocess(\sit))
\leq1.
\end{equation*}

Since we now know that \(\multprocess_I\) is a supermartingale multiplier for~\(\constantfrcstsystem[I]\), we may conclude that \(\test_I\coloneqq\mint[\multprocess_I]\) is a test supermartingale for~\(\constantfrcstsystem[I]\).
Furthermore, since \(\multprocess\) is {\lscomp}, so is \(\multprocess_I\), because taking minima and maxima are continuous and monotone (non-decreasing) operations.
Hence, \(\test_I\) belongs to~\(\multmlallowabletests[{\constantfrcstsystem[I]}]\).
Therefore, and because \(I\in\constantmultmlrandom\), we can conclude that \(\test_I\) remains bounded on~\(\pth\).

Also, if we let \(\multprocess_J\) be a map from situations to gambles on~\(\outcomes\), defined by
\begin{equation*}
\multprocess_J(\sit)(z)
\coloneqq
\begin{cases}
\max\set{\multprocess(\sit)(1),1}
&\text{if \(z=1\)}\\
\min\set{\multprocess(\sit)(0),1}
&\text{if \(z=0\)}
\end{cases}
\quad\text{for all~\(\sit\in\sits\) and \(z\in\outcomes\)},
\end{equation*}
and consider \(\test_J\coloneqq\mint[\multprocess_J]\), a similar course of reasoning leads us to conclude that \(\test_J\in\multmlallowabletests[{\constantfrcstsystem[J]}]\).
Therefore, and because \(J\in\constantmultmlrandom\), can also conclude that \(\test_J\) remains bounded on~\(\pth\).

Next, we observe that \(\multprocess=\multprocess_I\multprocess_J\), and therefore also \(\test=\mint=\mint[{\multprocess_I}]\mint[{\multprocess_J}]=\test_I\test_J\).
And since both \(\test_I\) and \(\test_J\) remain bounded on~\(\pth\), so, therefore, does \(\test\).
\end{proof}

\begin{proposition}\label{prop:constantcalibrated:intersectioninside:C}
For any~\(\pth\in\pths\), \(\constantcrandom\) is closed under (finite) intersections: for any two interval forecasts~\(I\) and \(J\) in~\(\constantcrandom\), we have that \(I\cap J\in\constantcrandom\).
\end{proposition}

\begin{proof}
The proof is almost completely analogous to that of Proposition~\ref{prop:constantcalibrated:intersectioninside:ML}.
After replacing~\(\multmlallowables\) and~\(\constantmultmlrandom\) with~\(\poscallowables\) and~\(\constantcrandom=\constantposcrandom\) [the equality follows from Equation~\eqref{eq:random:implies:schnorrandom}], respectively, the only steps that require changes are those that are concerned with {\lscompy}.

First, since \(\test\) is here a test supermartingale for~\(\constantfrcstsystem[I\cap J]\) that belongs to~\(\poscallowables\) and is therefore positive and {\comp}, we infer from Proposition~\ref{prop:computablemultiplier:from:process} that there now is some supermartingale multiplier~\(\multprocess\) that is positive and {\comp}, rather than merely {\lscomp}, such that \(\test=\mint\).
Secondly, we now need to show that the supermartingale multipliers~\(\multprocess_I\) and~\(\multprocess_J\) are positive and {\comp}, rather than merely {\lscomp}.
But this is trivially implied by the positive and {\comp} character of~\(\multprocess\).
\end{proof}

\begin{proposition}\label{prop:constantcalibrated:intersectioninside:schnorr}
For any~\(\pth\in\pths\), \(\constantschnorrrandom\) is closed under (finite) intersections: for any two interval forecasts~\(I\) and \(J\) in~\(\constantschnorrrandom\), we have that \(I\cap J\in\constantschnorrrandom\).
\end{proposition}

\begin{proof}
The proof starts from the proof of Proposition~\ref{prop:constantcalibrated:intersectioninside:C}.
Taking into account Proposition~\ref{prop:schnorrwithpositiveT}, the only additional complication is that we now also have to prove that if \(\test_I\) and \(\test_J\) are not computably unbounded on~\(\pth\), then neither is \(\test=\test_I\test_J\).
But this is an immediate consequence of Proposition~\ref{prop:computably:unbounded:product}, with \(\mu_1(n)\coloneqq\test_I(\pthto{n})\) and \(\mu_2(n)\coloneqq\test_J(\pthto{n})\) for all \(n\in\naturalswithzero\).
\end{proof}

\subsection{A few examples at the extreme ends}
We finish the discussion in this section by giving a few immediate examples of possible sets of interval forecasts.

On the one hand, for any precise forecast \(p\in\frcsts\), there always are sequences \(\pth\) that are \(\allowables\)-random, and at least as many that are Schnorr random, for the \emph{precise} stationary forecasting system~\(\constantfrcstsystem[p]\); see Corollary~\ref{cor:consistency}.
These types of random sequences have received most attention in the literature, thus far.
For any such sequence, a constant interval forecast~\(I\) will make it random if and only if it contains the precise forecast~\(p\): \(\constantrandom{\allowables}=\cset{I\in\imprecisefrcsts}{p\in I}\).
Hence, \(\lowerconstantrandom{\allowables}=[0,p]\) and \(\upperconstantrandom{\allowables}=[p,0]\), and therefore also
\begin{equation*}
\lowersmallestrandom{\allowables}=\uppersmallestrandom{\allowables}=p;
\end{equation*}
and similarly for Schnorr randomness.

At the other extreme end, any \emph{recursive} path with infinitely many zeroes and ones will only be random for the vacuous interval forecast.

\begin{proposition}\label{prop:constantcalibrated:computable:sequence}
If a path~\(\pth\in\pths\) is recursive and has infinitely many zeroes and infinitely many ones, then \(\constantrandom{\allowables}=\constantschnorrrandom=\set{[0,1]}\), so \(\lowerconstantrandom{\allowables}=\lowerconstantschnorrrandom=\set{0}\), \(\upperconstantrandom{\allowables}=\upperconstantschnorrrandom=\set{1}\), \(\lowersmallestrandom{\allowables}=\lowersmallestschnorrrandom=0\) and \(\uppersmallestrandom{\allowables}=\uppersmallestschnorrrandom=1\).
\end{proposition}

\begin{proof}
Since \(\pth\) is recursive, the selection functions~\(\selectionfunction_0\) and~\(\selectionfunction_1\) defined by
\begin{equation*}
\selectionfunction_1(n)\coloneqq\pthat{n}
\text{ and }
\selectionfunction_0(n)\coloneqq1-\pthat{n}
\text{ for all \(n\in\naturals\)},
\end{equation*}
are also recursive.
Moreover, since \(\pth\) has infinitely many zeroes and ones, \(\sum_{k=1}^{n}\selectionfunction_0(k)\to\infty\) and \(\sum_{k=1}^{n}\selectionfunction_1(k)\to\infty\).
For any~\(I\in\constantschnorrrandom\), we then infer from Corollary~\ref{cor:well:calibrated:constant:schnorr} that
\begin{equation*}
\min I
\leq\liminf_{n\to\infty}\frac{\sum_{k=1}^{n}\selectionfunction_0(k)\pthat{k}}
{\sum_{k=1}^{n}\selectionfunction_0(k)}
=\liminf_{n\to\infty}\frac{\sum_{k=1}^{n}(1-\pthat{k})\pthat{k}}
{\sum_{k=1}^{n}\selectionfunction_0(k)}
=0,
\end{equation*}
since all \(\pthat{k}(1-\pthat{k})=0\), and similarly
\begin{equation*}
\max I
\geq\limsup_{n\to\infty}\frac{\sum_{k=1}^{n}\selectionfunction_1(k)\pthat{k}}
{\sum_{k=1}^{n}\selectionfunction_1(k)}
\geq\limsup_{n\to\infty}\frac{\sum_{k=1}^{n}\pthat{k}^2}
{\sum_{k=1}^{n}\pthat{k}}
=1,
\end{equation*}
since all \(\pthat{k}^2=\pthat{k}\).
Hence, \(I=[0,1]\), and therefore \(\constantschnorrrandom=\set{[0,1]}\).
The same argument works for any~\(I\in\constantrandom{\allowables}\), and leads to the conclusion that also \(\constantrandom{\allowables}=\set{[0,1]}\).
\end{proof}

We show by means of a number of concrete examples in the next section that, in between these extremes of total imprecision and maximal precision, there lies a---to the best of our knowledge---previously uncharted realm of sequences, with `similar' unpredictability to the ones traditionally called `random', for which the intervals~\(\lowerconstantrandom{\allowables}\) and \(\upperconstantrandom{\allowables}\) need not always be closed, and more importantly, for which \(0<\lowersmallestrandom{\allowables}<\uppersmallestrandom{\allowables}<1\)---and similarly for Schnorr randomness.
This will provide the first evidence for our claim that `randomness is inherently imprecise'.

\section{Imprecise randomness due to non-stationarity}\label{sec:non-stationarity}
Our work on imprecise Markov chains \cite{cooman2008,cooman2015:markovergodic,krak2017:ICTMC,debock2021:uai,tjoens2019} has taught us that in some cases, we can very efficiently compute tight bounds on expectations in non-stationary precise Markov chains, by replacing them with their stationary imprecise versions.
Similarly, in statistical modelling, when learning from data sampled from a distribution with a varying (non-stationary) parameter, it seems hard to estimate the time sequence of its values, but we may be more successful in learning about its (stationary) interval range.
Similar ideas were also considered earlier by Fierens {\itshape et al.} \cite{fierens2009:frequentist}, when they argued for a frequentist interpretation of imprecise probability models based on non-stationarity.

In this section, we explore this idea in the context of our study of imprecise randomness, and show in a number of interesting examples that randomness associated with non-stationary precise forecasting systems can be captured by a stationary forecasting system, which must then be less precise: we gain simplicity of representation by going from a non-stationary to a stationary one, but we \emph{must} then pay for it by losing precision.

\subsection{A simple example}\label{sec:example:p:q}
Let us begin with a simple example to get some idea of where we want to go to.
In what follows, \(\allowables\) is any set of allowable test processes.
We discuss \(\allowables\)-randomness here, but completely analogous arguments and conclusions are valid for Schnorr randomness.

Consider any~\(p\) and \(q\) in~\(\frcsts\) with \(p<q\), and \emph{any} path~\(\pth\) that is \(\allowables\)-random for the forecasting system~\(\frcstsystem_{p,q}\) that is defined by
\begin{equation*}
\frcstsystem_{p,q}(\sit)
\coloneqq
\begin{cases}
p &\text{if \(\dist{\sit}\) is odd}\\
q &\text{if \(\dist{\sit}\) is even}
\end{cases}
\quad\text{for all~\(\sit\in\sits\).}
\end{equation*}
We know from Corollary~\ref{cor:consistency} that there is at least one such path.

We now look for the \emph{stationary} forecasting systems that make this~\(\pth\) \(\allowables\)-random, and we intend to show that for all~\(I\in\imprecisefrcsts\):
\begin{equation}\label{eq:pq:intervals}
I\in\constantrandom{\allowables}
\ifandonlyif
\sqgroup{p,q}\subseteq I,
\end{equation}
which then also implies that
\begin{equation*}
\lowerconstantrandom[\pth]{\allowables}=\sqgroup{0,p}
\text{, }
\upperconstantrandom[\pth]{\allowables}=\sqgroup{q,1}
\text{, }
\lowersmallestrandom[\pth]{\allowables}=p
\text{ and }
\uppersmallestrandom[\pth]{\allowables}=q.
\end{equation*}

\begin{proof}[Proof of Equation~\eqref{eq:pq:intervals}]
The converse implication follows at once from Proposition~\ref{prop:nestedfrcstsystems} and the fact that for any~\(I\in\imprecisefrcsts\) such that \([p,q]\subseteq I\), the stationary forecasting system~\(\constantfrcstsystem[I]\) is more conservative than \(\frcstsystem_{p,q}\), in the sense that \(\frcstsystem_{p,q}\subseteq\constantfrcstsystem[I]\).

For the direct implication, assume that \(I\in\constantrandom{\allowables}\) and fix any~\(\epsilon>0\).
Since all rational numbers are {\comp}, there are {\comp} intervals \(\pinterval\in\imprecisefrcsts\) and \(\qinterval\in\imprecisefrcsts\) such that
\begin{equation*}
p\in\pinterval\subseteq[p-\epsilon,p+\epsilon]
\text{ and }
q\in\qinterval\subseteq[q-\epsilon,q+\epsilon].
\end{equation*}
Consider now the forecasting system~\(\frcstsystem_\epsilon\), defined by
\begin{equation*}
\frcstsystem_\epsilon(\sit)
\coloneqq
\begin{cases}
\pinterval &\text{if \(\dist{\sit}\) is odd}\\
\qinterval &\text{if \(\dist{\sit}\) is even}
\end{cases}
\quad\text{for all~\(\sit\in\sits\).}
\end{equation*}
Then \(\frcstsystem_\epsilon\) is clearly {\comp} and, since \(\frcstsystem_{p,q}\subseteq\frcstsystem_\epsilon\), we know from Proposition~\ref{prop:nestedfrcstsystems} that \(\pth\) is \(\allowables\)-random for~\(\frcstsystem_\epsilon\).
Therefore, we find that
\begin{equation*}
\min I
\leq\liminf_{n\to\infty}\frac{1}{n}\smashoperator{\sum_{k=1}^n}\pth_{2k}
\leq\limsup_{n\to\infty}\frac{1}{n}\smashoperator{\sum_{k=1}^n}\pth_{2k}
\leq\up
\leq p+\epsilon,
\end{equation*}
where the first and third inequality follow from Corollary~\ref{cor:well:calibrated:constant:schnorr} and Theorem~\ref{thm:well:calibrated:general:schnorr}, respectively, for appropriately chosen recursive selection functions, and for~\(h=\indsing{1}\).
Similarly, but now with \(h=-\indsing{1}\), we also find that
\begin{equation*}
\max I
\geq\limsup_{n\to\infty}\frac{1}{n}\smashoperator{\sum_{k=1}^n}\pth_{2k-1}
\geq\liminf_{n\to\infty}\frac{1}{n}\smashoperator{\sum_{k=1}^n}\pth_{2k-1}
\geq\lptoo
\geq q-\epsilon.
\end{equation*}
Since \(\epsilon>0\) is arbitrary, this allows us to conclude that \(\min I\leq p\) and \(\max I\geq q\), and, therefore, that \([p,q]\subseteq I\).
\end{proof}

\subsection{A more complicated example}\label{sec:example:almost:half}
Next, we turn to a more complicated example, where we look at sequences that are `nearly' random for the constant precise forecast~\(\nicefrac{1}{2}\), but not quite.
We begin by considering the following sequence~\(\set{p_n}_{n\in\naturalswithzero}\) of precise forecasts:
\begin{equation*}
p_{n}\coloneqq\frac{1}{2}+(-1)^{n}\delta_{n}
\text{ with }
\delta_n\coloneqq\sqrt{\frac{8}{n+33}}
\quad\text{for all~\(n\in\naturalswithzero\)}.
\end{equation*}
Since the sequence~\(\{\delta_n\}_{n\in\naturalswithzero}\) decreases towards its limit \(0\) and \(\delta_n\in(0,\nicefrac{1}{2})\) for all~\(n\in\naturalswithzero\), we see that~\(p_n\to\nicefrac{1}{2}\) and that \(p_n\in(0,1)\) for all \(n\in\naturalswithzero\).

In this example, we will focus our attention on an \emph{arbitrary} but fixed path~\(\pth\) that is \(\mlallowables\)-random for the {\comp} precise forecasting system~\(\frcstsystem_{\sim\nicefrac{1}{2}}\) defined by
\begin{equation*}
\frcstsystem_{\sim\nicefrac{1}{2}}(\sit)\coloneqq p_{\dist{\sit}}
\text{ for all~\(\sit\in\sits\).}
\end{equation*}
We know from Corollary~\ref{cor:consistency} that there is at least one such path.
We will show, in a number of successive steps, that for all~\(\allowables\) such that \(\poscallowables\subseteq\allowables\subseteq\mlallowables\):
\begin{equation*}
\constantrandom{\allowables}
=\constantschnorrrandom
=\cset[\Big]{\pinterval\in\imprecisefrcsts}{\lp<\nicefrac{1}{2}<\up},
\end{equation*}
and therefore
\begin{equation*}
\lowerconstantrandom{\allowables}=\lowerconstantschnorrrandom=[0,\nicefrac{1}{2})
\text{ and }
\upperconstantrandom{\allowables}=\upperconstantschnorrrandom=(\nicefrac{1}{2},1]
\end{equation*}
and
\begin{equation*}
\lowersmallestrandom{\allowables}=\lowersmallestschnorrrandom
=\uppersmallestrandom{\allowables}=\uppersmallestschnorrrandom=\nicefrac{1}{2}.
\end{equation*}

We first prove that
\begin{equation*}
\text{\(\pinterval\in\constantmlrandom\) for all \(\pinterval\in\imprecisefrcsts\) such that \(\lp<\nicefrac{1}{2}<\up\)}, 
\end{equation*}
and therefore also \(\pinterval\in\constantrandom{\allowables}\) and \(\pinterval\in\constantschnorrrandom\), by Proposition~\ref{prop:constantcalibrated:nestedallowables}.

\begin{proof}[{\protect Proof that \(\pinterval\in\constantmlrandom\) if \(\pinterval\in\imprecisefrcsts\) and \(\lp<\nicefrac{1}{2}<\up\)}]
We provide a proof by contradiction.
Assume {\itshape ex absurdo} that there is some~\(I\coloneqq\pinterval\in\imprecisefrcsts\) such that \(\lp<\nicefrac{1}{2}<\up\) and \(I\notin\constantmultmlrandom\).
This implies that there is some {\lscomp} test supermartingale~\(\supermartin_I\) for the stationary forecasting system~\(\constantfrcstsystem[I]\) that is unbounded on~\(\pth\).

Consider any~\(m\in\naturalswithzero\) such that \(p_n\in\pinterval=I\) for all \(n\geq m\); this is always possible because \(p_n\) converges to \(\nicefrac{1}{2}\) and \(\lp<\nicefrac{1}{2}<\up\).
Let~\(\alpha>0\) be any rational number such that \(\supermartin_I(\sit)\leq\alpha\) for all \(\sit\in\sits\) with \(\abs{\sit}=m+1\); there always is such an \(\alpha\) because the number of situations of length \(m+1\) is finite.
We now consider a new process~\(\supermartin\), defined by
\begin{equation*}
\supermartin(\sit)
\coloneqq
\begin{cases}
\frac{1}{\alpha}\supermartin_I(\sit)
&\text{if \(\abs{\sit}>m\)}\\
1
&\text{if \(\abs{\sit}\leq m\)}
\end{cases}
\quad\text{for all~\(\sit\in\sits\),}
\end{equation*}
which is {\lscomp} because \(\supermartin_I\) is {\lscomp} and because \(\alpha\) is rational.
This process is furthermore positive because \(\supermartin_I\) and \(\alpha\) are, and it has unit initial value~\(\supermartin(\init)=1\) by definition.
Hence, it is a test process.
To see that it is also a supermartingale for~\(\frcstsystem_{\sim\nicefrac{1}{2}}\), we verify the condition in Equation~\eqref{eq:supermartingale}.
Consider any~\(\sit\in\sits\).
We distinguish three cases: \(\abs{\sit}>m\), \(\abs{\sit}=m\) and \(\abs{\sit}<m\).
If \(\abs{\sit}>m\), then
\begin{equation*}
\ex_{p_{\abs{\sit}}}(\supermartin(\sit\cdot))\leq\uex_I(\supermartin(\sit\cdot))=\uex_I\Big(\frac{1}{\alpha}\supermartin_I(\sit\cdot)\Big)=\frac{1}{\alpha}\uex_I(\supermartin_I(\sit\cdot))\leq\frac{1}{\alpha}\supermartin_I(\sit)=\supermartin(\sit),
\end{equation*}
where the first inequality holds because \(p_{\abs{\sit}}\in I\), the second equality follows from coherence property~\ref{axiom:coherence:homogeneity}, and the second inequality holds because \(\supermartin_I\) is a supermartingale for~\(\constantfrcstsystem[I]\).
If \(\abs{\sit}=m\), then \(\supermartin_I(\sit\cdot)\leq\alpha\) and therefore
\begin{equation*}
\ex_{p_{\abs{\sit}}}(\supermartin(\sit\cdot))\leq
\uex_I(\supermartin(\sit\cdot))=\uex_I\Big(\frac{1}{\alpha}\supermartin_I(\sit\cdot)\Big)\leq\uex_I(1)\leq1=\supermartin(\sit),
\end{equation*}
where the first inequality holds because \(p_{\abs{\sit}}\in I\), and the second and third inequalities follow from coherence properties~\ref{axiom:coherence:monotonicity} and~\ref{axiom:coherence:bounds}, respectively.
Finally, if \(\abs{\sit}<m\), then \(\ex_{p_{\abs{\sit}}}(\supermartin(\sit\cdot))=\ex_{p_{\abs{\sit}}}(1)=1=\supermartin(\sit)\).
So we can conclude that \(\smash{\ex_{\frcstsystem_{\sim\nicefrac{1}{2}}(\sit)}(\supermartin(s\cdot))=\ex_{p_{\abs{\sit}}}(\supermartin(s\cdot))\leq\supermartin(s)}\) for all~\(\sit\in\sits\).
Hence, \(\supermartin\) is a {\lscomp} test supermartingale for~\(\frcstsystem_{\sim\nicefrac{1}{2}}\).
However, by construction, \(\supermartin\) is unbounded above on~\(\pth\), simply because \(\supermartin_I\) is unbounded above on~\(\pth\) and \(\alpha\) is positive.
This contradicts the fact that \(\pth\) is \(\mlallowables\)-random for~\(\frcstsystem_{\sim\nicefrac{1}{2}}\).
\end{proof}

We complete the argument by showing that
\begin{equation*}
\text{\(\pinterval\notin\constantschnorrrandom\) for any~\(\pinterval\in\imprecisefrcsts\) such that \(\lp\geq\nicefrac{1}{2}\) or \(\up\leq\nicefrac{1}{2}\)}. 
\end{equation*}
Taking into account Proposition~\ref{prop:constantcalibrated:nestedallowables}, this will then also tell us that \(\pinterval\notin\constantrandom{\allowables}\), for all \(\poscallowables\subseteq\allowables\subseteq\mlallowables\). So it implies in particular that \(\{\nicefrac{1}{2}\}\notin\constantschnorrrandom\) and \(\set{\nicefrac{1}{2}}\notin\constantrandom{\allowables}\), meaning that the sequence isn't Schnorr random in the classical `fair coin' sense, nor computably random or (weakly) {\ML} random.

The proof is based on ideas involving Hellinger-like divergences in a beautiful paper by Volodya Vovk \cite{vovk2009:merging}: if the forecast sequences produced by two precise forecasting systems along a path~\(\pth\) lie `far enough' from each other, then it is possible to construct simple test supermartingales for these respective forecasting systems whose product becomes unbounded on~\(\pth\), implying that~\(\pth\) can't be random for both forecasting systems.
Here, we show that this idea can be extended to the case where one of the forecasting systems is imprecise.
We will have occasion to use this proof method again, in our proof of Theorem~\ref{thm:inherentlyimprecise}.

\begin{proof}[Proof that \(\pinterval\notin\constantschnorrrandom\) for any~\(\pinterval\in\imprecisefrcsts\) such that \(\lp\geq\nicefrac{1}{2}\) or \(\up\leq\nicefrac{1}{2}\)]
We only prove the result for~\(\lp\geq\nicefrac{1}{2}\); the proof for the other case is entirely analogous, the main difference with the argument below being that we then need to focus on the even rather than the odd indices.

Let \(I\coloneqq\pinterval\).
Then, by assumption, \(I\subseteq[\nicefrac{1}{2},1]\).
Consider the two gamble processes~\(\multprocess_I\) and \(\multprocess_{\sim\nicefrac{1}{2}}\), defined for all~\(\sit\in\sits\) by
\begin{equation*}
\multprocess_I(\sit)
\coloneqq
\begin{cases}
f_{\nicefrac{1}{2},p_{\abs{\sit}}}
&\text{if \(\abs{s}\) is odd}\\
1
&\text{if \(\abs{s}\) is even}
\end{cases}
\quad\text{and}\quad
\multprocess_{\sim\nicefrac{1}{2}}(\sit)
\coloneqq
\begin{cases}
f_{p_{\abs{\sit}},\nicefrac{1}{2}}
&\text{if \(\abs{s}\) is odd}\\
1
&\text{if \(\abs{s}\) is even},
\end{cases}
\end{equation*}
where, for any~\(\alpha,\beta\in(0,1)\), we define the gamble~\(f_{\alpha,\beta}\) on~\(\outcomes\) by
\begin{equation}\label{eq:workhorse:gambles}
f_{\alpha,\beta}(1)
\coloneqq
\frac{\sqrt{\nicefrac{\beta}{\alpha}}}{\sqrt{\alpha\beta}+\sqrt{(1-\alpha)(1-\beta)}}
\text{ and }
f_{\alpha,\beta}(0)
\coloneqq
\frac{\sqrt{\nicefrac{1-\beta}{1-\alpha}}}{\sqrt{\alpha\beta}+\sqrt{(1-\alpha)(1-\beta)}}.
\end{equation}
These gamble processes are {\comp} because the sequence~\((p_n)_{n\in\naturalswithzero}\) is {\comp} and because checking whether \(\abs{\sit}\) is odd is decidable.
Furthermore, due to Lemma~\ref{lem:hulpresultaatvoorinherentlyIP}\ref{it:hulpresultaatvoorinherentlyIP:expectation}, we also know that they're positive.
So we find that \(\multprocess_I\) and \(\multprocess_{\sim\nicefrac{1}{2}}\) are {\comp} multiplier processes.
We now proceed to show that they're in fact {\comp} supermartingale multipliers for~\(\constantfrcstsystem[I]\) and~\(\frcstsystem_{\sim\nicefrac{1}{2}}\), respectively.
To see that~\(\multprocess_I\) is a supermartingale multiplier for~\(\constantfrcstsystem[I]\), observe that
\begin{equation*}
\uex_I(\multprocess_I(\sit))
=
\begin{cases}
\uex_I(f_{\nicefrac{1}{2},p_{\abs{\sit}}})
\leq1
&\text{if \(\abs{s}\) is odd}\\
\uex_I(1)\leq1
&\text{if \(\abs{s}\) is even}
\end{cases}
\quad\text{for all~\(\sit\in\sits\),}
\end{equation*}
where the odd case follows from Lemma~\ref{lem:hulpresultaatvoorinherentlyIP}\ref{it:hulpresultaatvoorinherentlyIP:intervalabove} because then \(p_{\abs{\sit}}<\nicefrac{1}{2}\leq\lp\), and the even case follows from coherence property~\ref{axiom:coherence:bounds}.
To see that \(\multprocess_{\sim\nicefrac{1}{2}}\) is a supermartingale multiplier for~\(\frcstsystem_{\sim\nicefrac{1}{2}}\), observe that
\begin{equation*}
\ex_{p_{\abs{\sit}}}(\multprocess_{\sim\nicefrac{1}{2}}(\sit))
=
\begin{cases}
\ex_{p_{\abs{\sit}}}(f_{p_{\abs{\sit}},\nicefrac{1}{2}})
=1
&\text{if \(\abs{s}\) is odd}\\
\ex_{p_{\abs{\sit}}}(1)=1
&\text{if \(\abs{s}\) is even}
\end{cases}
\quad\text{for all~\(\sit\in\sits\),}
\end{equation*}
using Lemma~\ref{lem:hulpresultaatvoorinherentlyIP}\ref{it:hulpresultaatvoorinherentlyIP:expectation} for the odd case.
Taking into account Proposition~\ref{prop:computable:from:multiplier}, we conclude from the above that~\(\mint[\multprocess]_I\) and \(\mint[\multprocess]_{\sim\nicefrac{1}{2}}\) are {\comp} test supermartingales for~\(\constantfrcstsystem[I]\) and \(\frcstsystem_{\sim\nicefrac{1}{2}}\), respectively.

Let us now take a look at the product of~\(\mint[\multprocess]_I\) and \(\mint[\multprocess]_{\sim\nicefrac{1}{2}}\).
We start by observing that, for all~\(n\in\naturalswithzero\),
\begin{equation*}
\frac{1}{1-\frac{1}{4}\big(\frac{1}{2}-p_n\big)^2}
=\frac{1}{1-\frac{1}{4}\big(-(-1)^{n}\delta_{n}\big)^2}
=\frac{1}{1-\frac{1}{4}\delta_{n}^2}
=\frac{1}{1-\frac{1}{4}\frac{8}{n+33}}
=\frac{1}{1-\frac{2}{n+33}}
=\frac{n+33}{n+31}.\\[5pt]
\end{equation*}
Because of Lemma~\ref{lem:hulpresultaatvoorinherentlyIP}\ref{it:hulpresultaatvoorinherentlyIP:product}, this implies that, for all~\(n\in\naturalswithzero\),
\begin{equation*}
\multprocess_I(\pthto{n})\multprocess_{\sim\nicefrac{1}{2}}(\pthto{n})
=
\begin{cases}
f_{\nicefrac{1}{2},p_n}f_{p_n,\nicefrac{1}{2}}\geq\group[\big]{1-\frac{1}{4}(\nicefrac{1}{2}-p_n)^2}^{-1}=\frac{n+33}{n+31}
&\text{if \(n\) is odd}\\
1
&\text{if \(n\) is even.}
\end{cases}
\end{equation*}
Hence, for all~\(n\in\naturals\),
\begin{align*}
\mint[\multprocess]_I(\pthto{2n})\mint[\multprocess]_{\sim\nicefrac{1}{2}}(\pthto{2n})
&=\smashoperator[r]{\prod_{k=0}^{2n-1}}\multprocess_I(\pthto{k})(\pthat{k+1})\multprocess_{\sim\nicefrac{1}{2}}(\pthto{k})(\pthat{k+1})
\geq\smashoperator{\prod_{\substack{k=0\\k\text{ odd}}}^{2n-1}}\frac{k+33}{k+31}
=\frac{2n+32}{32}.
\end{align*}
Because the map~\(\realordering\colon\naturalswithzero\to\nonnegreals\colon n\mapsto\frac{n+32}{32}\) is a real growth function,  Proposition~\ref{prop:computably:unbounded:various}\ref{it:computably:unbounded:real} guarantees that the product~\(\mint[\multprocess_I]\mint[\multprocess_{\sim\nicefrac{1}{2}}]\) is computably unbounded on~\(\pth\), so Proposition~\ref{prop:computably:unbounded:product} tells us that at least one of the factor supermartingales~\(\mint[\multprocess_I]\) and~\(\mint[\multprocess_{\sim\nicefrac{1}{2}}]\) must be computably unbounded on~\(\pth\) too.
But since \(\mint[\multprocess]_{\sim\nicefrac{1}{2}}\) is a {\comp}---and therefore also {\lscomp} due to Proposition~\ref{prop:computable:upper:lower}---test supermartingale multiplier for the forecasting system~\(\frcstsystem_{\sim\nicefrac{1}{2}}\), it follows from the assumed \(\mlallowables\)-randomness of \(\pth\) for~\(\frcstsystem_{\sim\nicefrac{1}{2}}\) that \(\mint[\multprocess_{\sim\nicefrac{1}{2}}]\) can't be computably unbounded on~\(\pth\), so \(\mint[\multprocess]_I\) must be.
Since \(\mint[\multprocess]_I\) is a {\comp} supermartingale for~\(\constantfrcstsystem[I]\), we conclude that, indeed, \(\pinterval=I\notin\constantschnorrrandom\).
\end{proof}

\begin{lemma}\label{lem:hulpresultaatvoorinherentlyIP}
For any~\(\alpha,\beta\in(0,1)\), we consider the gamble~\(f_{\alpha,\beta}\) on~\(\outcomes\) defined in Equation~\eqref{eq:workhorse:gambles}.
Then for any~\(\alpha,\beta\in(0,1)\) and \(I=\pinterval\in\imprecisefrcsts\), the following statements hold:
\begin{enumerate}[label=\upshape(\roman*),leftmargin=*]
\item\label{it:hulpresultaatvoorinherentlyIP:expectation} \(\ex_{\alpha}(f_{\alpha,\beta})=1\), \(f_{\alpha,\beta}(0)>0\) and~\(f_{\alpha,\beta}(1)>0\);
\item\label{it:hulpresultaatvoorinherentlyIP:dominance}  \(f_{\alpha,\beta}(1)>f_{\alpha,\beta}(0)\) if and only if \(\alpha<\beta\);
\item\label{it:hulpresultaatvoorinherentlyIP:intervalbelow} if \(\up\leq\alpha<\beta\), then \(\uex_I(f_{\alpha,\beta})\leq1\);
\item\label{it:hulpresultaatvoorinherentlyIP:intervalabove} if \(\alpha<\beta\leq\lp\), then \(\uex_I(f_{\beta,\alpha})\leq1\);
\item\label{it:hulpresultaatvoorinherentlyIP:product} \(f_{\alpha,\beta}(0)f_{\beta,\alpha}(0)=f_{\alpha,\beta}(1)f_{\beta,\alpha}(1)\geq\smash{\group[\big]{1-\frac{1}{4}(\alpha-\beta)^2}^{-1}}\).
\end{enumerate}
\end{lemma}

\begin{proof}
Statement~\ref{it:hulpresultaatvoorinherentlyIP:expectation} is an immediate consequence of the definition of~\(f_{\alpha,\beta}\) and \(\ex_\alpha\) and the fact that \(\alpha,\beta\in(0,1)\).
For statement~\ref{it:hulpresultaatvoorinherentlyIP:dominance}, observe that, indeed, since \(\alpha,\beta\in(0,1)\),
\begin{equation*}
f_{\alpha,\beta}(1)>f_{\alpha,\beta}(0)
\ifandonlyif
\sqrt{\nicefrac{\beta}{\alpha}}>\sqrt{\nicefrac{1-\beta}{1-\alpha}}
\ifandonlyif
\beta(1-\alpha)>\alpha(1-\beta)
\ifandonlyif
\beta>\alpha.
\end{equation*}

For statements~\ref{it:hulpresultaatvoorinherentlyIP:intervalbelow} and~\ref{it:hulpresultaatvoorinherentlyIP:intervalabove}, first observe that it follows from \(\alpha<\beta\) and statement~\ref{it:hulpresultaatvoorinherentlyIP:dominance} that \(f_{\alpha,\beta}(1)>f_{\alpha,\beta}(0)\) and \(f_{\beta,\alpha}(1)\leq f_{\beta,\alpha}(0)\).
Statement~\ref{it:hulpresultaatvoorinherentlyIP:intervalbelow} now follows because
\begin{equation*}
\uex_I(f_{\alpha,\beta})=\ex_{\up}(f_{\alpha,\beta})\leq\ex_\alpha(f_{\alpha,\beta})=1,
\end{equation*}
where the first equality and the inequality follow from Equations~\eqref{eq:local:upper} and~\eqref{eq:local:linear}, respectively, and the fact that \(f_{\alpha,\beta}(1)>f_{\alpha,\beta}(0)\), and where the last equality follows from statement~\ref{it:hulpresultaatvoorinherentlyIP:expectation}.
Statement~\ref{it:hulpresultaatvoorinherentlyIP:intervalabove} follows because
\begin{equation*}
\uex_I(f_{\beta,\alpha})=\ex_{\lp}(f_{\beta,\alpha})\leq\ex_\beta(f_{\beta,\alpha})=1,
\end{equation*}
where the first equality and the inequality follow from Equations~\eqref{eq:local:upper} and~\eqref{eq:local:linear}, respectively, and the fact that \(f_{\beta,\alpha}(1)\leq f_{\beta,\alpha}(0)\), and where the last equality follows from statement~\ref{it:hulpresultaatvoorinherentlyIP:expectation} with \(\alpha\) and \(\beta\) interchanged.

Statement~\ref{it:hulpresultaatvoorinherentlyIP:product} follows from Lemma~\ref{lem:alphabetainequalities} because
\begin{equation*}
f_{\alpha,\beta}(0)f_{\beta,\alpha}(0)
=f_{\alpha,\beta}(1)f_{\beta,\alpha}(1)
=\frac{1}{\group[\Big]{\sqrt{\alpha\beta}+\sqrt{(1-\alpha)(1-\beta)}}^2}.
\qedhere
\end{equation*}
\end{proof}

\begin{lemma}\label{lem:alphabetainequalities}
For any~\(\alpha,\beta\in(0,1)\), we have that
\begin{equation*}
0<\group[\Big]{\sqrt{\alpha\beta}+\sqrt{(1-\alpha)(1-\beta)}}^2
\leq1-\frac{1}{4}(\alpha-\beta)^2.
\end{equation*}
\end{lemma}

\begin{proof}
The first inequality follows trivially from \(\alpha,\beta\in(0,1)\).
To prove the second, let
\begin{equation*}
a\coloneqq\alpha+\beta-2\alpha\beta=\alpha(1-\beta)+\beta(1-\alpha)>0
\text{ and }
b\coloneqq2\sqrt{\alpha\beta(1-\alpha)(1-\beta)}>0.
\end{equation*}
First observe that
\begin{align}\label{eq:lem:alphabetainequalities:2}
&\group[\Big]{\sqrt{\alpha\beta}+\sqrt{(1-\alpha)(1-\beta)}}^2\notag\\
&\hspace{2.3cm}=\alpha\beta+(1-\alpha)(1-\beta)+2\sqrt{\alpha\beta}\sqrt{(1-\alpha)(1-\beta)}
=1+b-a
\end{align}
and
\begin{align}\label{eq:lem:alphabetainequalities:3}
a^2-b^2
&=\big(\alpha+\beta-2\alpha\beta\big)^2
-\group[\Big]{2\sqrt{\alpha\beta(1-\alpha)(1-\beta)}}^2\notag\\
&=\group[\big]{\alpha^2+\beta^2+4\alpha^2\beta^2+2\alpha\beta-4\alpha^2\beta-4\alpha\beta^2}-4\alpha\beta(1-\alpha)(1-\beta)\notag\\
&=\alpha^2+\beta^2-2\alpha\beta=(\alpha-\beta)^2.
\end{align}
Next, we prove that \(a-b\geq\frac{1}{4}(\alpha-\beta)^2\).
On the one hand, since \(a>0\) and \(b>0\), we know that \(a+b>0\). On the other hand, we also know that
\begin{equation}\label{eq:lem:alphabetainequalities:4}
a+b=\alpha(1-\beta)+\beta(1-\alpha)+2\sqrt{\alpha\beta(1-\alpha)(1-\beta)}\leq 1+1+2=4.
\end{equation}
We therefore find that
\begin{equation*}
a-b
=\frac{(a-b)(a+b)}{a+b}
=\frac{a^2-b^2}{a+b}
=\frac{(\alpha-\beta)^2}{a+b}
\geq\frac{1}{4}(\alpha-\beta)^2,
\end{equation*}
using Equation~\eqref{eq:lem:alphabetainequalities:3} for the third equality and Equation~\eqref{eq:lem:alphabetainequalities:4} for the inequality.
Combined with Equation~\eqref{eq:lem:alphabetainequalities:2}, it follows that, indeed,
\begin{equation*}
\group[\Big]{\sqrt{\alpha\beta}+\sqrt{(1-\alpha)(1-\beta)}}^2
=1-a+b\leq1-\frac{1}{4}(\alpha-\beta)^2.
\qedhere
\end{equation*}
\end{proof}

\section{Imprecision can't be explained away}\label{sec:inherently}
The examples in the previous section illustrate that randomness associated with a non-stationary precise forecasting system can also be `described' as randomness for a simpler, stationary but then necessarily imprecise, forecasting system.
This observation might lead to the suspicion that all stationary imprecise forms of randomness can be `explained away' as such simpler representations of non-stationary but precise forms of randomness.
This would imply that the imprecision---or loss of precision---in the stationary forecasts isn't essential, and can always be dismissed as a mere artefact, a simple effect of using a stationary representation that isn't powerful enough to allow for the ideal representation, which must be, one would suspect, always precise but non-stationary.

We mean to show in this section that this suspicion is misguided, and even flat out wrong when we focus on {\comp} forecasting systems: we will see that there are paths that are random for a {\comp} stationary interval forecasting system that are never random for any {\comp} precise forecasting system, be it stationary or not.
This serves to further corroborate our claim that randomness is indeed inherently imprecise, as its imprecision can't be explained away as an effect of oversimplification.
The imprecision involved is furthermore non-negligible, and can be made arbitrarily large, because besides excluding the possibility of randomness of such paths for precise {\comp} forecasting systems, we also show they can't be random for any {\comp} forecasting system whose highest imprecision is smaller than that of the original, stationary one.

\begin{theorem}[Imprecision can't be explained away]\label{thm:inherentlyimprecise}
Consider any set of allowable test processes~\(\allowables\), and any interval forecast~\(I=\pinterval\in\imprecisefrcsts\).
Then there is path~\(\pth\in\pths\) that is \(\allowables\)-random---and therefore also Schnorr random---for the stationary interval forecast~\(I\), but that is never Schnorr random---and therefore never \(\allowables\)-random---for any {\comp} forecasting system~\(\frcstsystem\) whose highest imprecision is smaller than that of~\(I\), in the specific sense that \(\sup_{\sit\in\sits}\sqgroup[\big]{\ufrcstsystem(\sit)-\lfrcstsystem(\sit)}<\up-\lp\).
\end{theorem}
\noindent Our argument is crucially inspired by Volodya Vovk, who hit upon the essential idea and provided a first sketch for it.
We use the same basic idea, but follow an argumentation and construction that is different in a number of ways, in order to also---contrary to his approach---deal with precise {\comp} forecasts that may be irrational or become zero, and more importantly, with imprecise {\comp} forecasts whose imprecision is smaller than that of the stationary one.
Our proof method also has a more constructive flavour than his, which was based on an almost sure convergence argument.
Interestingly, our more constructive line of reasoning is inspired by the ideas and techniques he proposed in another paper \cite{vovk2009:merging}, and whose extension to our imprecise context we've already used for the example in Section~\ref{sec:example:almost:half}.
Stripped to its bare essentials, the argument is actually quite simple. 
We construct a precise forecasting system~\(\frcstsystem_{\lp,\up}\) whose forecasts are included in~\(I\) and infinitely often lie `far enough' from each of the countably many computable forecasting systems~\(\frcstsystem_m\) whose highest imprecision is smaller than that of~\(I\).
This then guarantees that no path can be simultaneously random for~\(\frcstsystem_{\lp,\up}\)---and therefore for~\(\constantfrcstsystem[I]\)---and for any such~\(\frcstsystem_m\).

\begin{proof}[Proof of Theorem~\ref{thm:inherentlyimprecise}]
To start the argument, we consider any recursive map~\(\lambda\colon\naturalswithzero\to\naturalswithzero\) such that for each~\(m\in\naturalswithzero\) there are infinitely many~\(n\in\naturalswithzero\) that are mapped to~\(m\), meaning that \(\lambda(n)=m\).
For instance, \(\lambda(n)\) could be the number of trailing zeroes in the binary expansion of~\(n+1\), so \(\lambda(n)\coloneqq\max\cset{k\in\naturalswithzero}{(n+1)2^{-k}\in\naturals}\) for all~\(n\in\naturalswithzero\), and consequently \(\lambda^{-1}(\set{m})=\cset{2^m(2\ell+1)-1}{\ell\in\naturalswithzero}\) for all~\(m\in\naturalswithzero\).

We also let \(\frcstsystem_0\), \(\frcstsystem_1\), \dots, \(\frcstsystem_n\), \dots\ be any enumeration of the (countably many) {\comp} forecasting systems, and we use this enumeration to let
\begin{equation*}
\compfrcstsystems[I]
\coloneqq\cset[\bigg]{m\in\naturalswithzero}{\sup_{\sit\in\sits}\sqgroup[\big]{\ufrcstsystem_m(\sit)-\lfrcstsystem_m(\sit)}<\up-\lp}
\end{equation*}
identify the set of all {\comp} forecasting systems whose highest imprecision is smaller than that of the interval forecast~\(I\).

We're now going to fix any~\(m\in\compfrcstsystems[I]\), or in other words, any such {\comp} forecasting system~\(\frcstsystem_m\).
Let \(0<\epsilon_m<1\) be any rational number [which there always is] such that
\begin{equation}\label{eq:theo:inherentlyimprecise:1}
\sup_{\sit\in\sits}\sqgroup[\big]{\ufrcstsystem_m(\sit)-\lfrcstsystem_m(\sit)}+6\epsilon_m<\up-\lp.
\end{equation}
Also, let \(\lp[m]\) and \(\up[m]\) be any two rational numbers [which there always are] such that
\begin{equation*}
\lp<\lp[m]<\lp+\epsilon_m
\text{ and }
\up-\epsilon_m<\up[m]<\up,
\end{equation*}
and consider any~\(N_m\in\naturals\) such that \(2^{-N_m}<\epsilon_m\).
Since the real process~\(\ufrcstsystem_m\) is {\comp} [because the forecasting system~\(\frcstsystem_m\) is], we know from Proposition~\ref{prop:computable:simplified} and the definition of a recursive net of rational numbers that there are three recursive maps~\(a_m\), \(b_m\) and \(\varsigma_m\) from~\(\sits\times\naturalswithzero\) to~\(\naturalswithzero\) such that
\begin{equation*}
b_m(\sit,n)>0
\text{ and }
\abs[\bigg]{(-1)^{\varsigma_m(\sit,n)}\frac{a_m(\sit,n)}{b_m(\sit,n)}-\ufrcstsystem_m(\sit)}
\leq2^{-n}
\text{ for all~\(\sit\in\sits\) and \(n\in\naturalswithzero\)}.
\end{equation*}
Hence, if we let \(\ufrcstsystem'_m\) be the rational-valued process defined by
\begin{equation*}
\ufrcstsystem'_m(\sit)\coloneqq (-1)^{\varsigma_m(\sit,N_m)}\frac{a_m(\sit,N_m)}{b_m(\sit,N_m)}
\text{ for all~\(\sit\in\sits\)},
\end{equation*}
then clearly \(\abs{\ufrcstsystem'_m(\sit)-\ufrcstsystem_m(\sit)}\leq2^{-N_m}<\epsilon_m\) for all~\(\sit\in\sits\).

We also establish a number of inequalities that will be important further on in this proof.
On the one hand, we have that
\begin{equation}\label{eq:theo:inherentlyimprecise:2}
0\leq\lp<\lp[m]<\lp+\epsilon_m<\lp[m]+\epsilon_m<\lp+2\epsilon_m,
\end{equation}
where the first inequality holds because \(I\subseteq[0,1]\), the second and third inequalities follow from our choice of \(\lp[m]\), the fourth inequality follows from the second, and the fifth inequality follows from the third.
On the other hand, we have that
\begin{equation}\label{eq:theo:inherentlyimprecise:3}
\up-2\epsilon_m<\up[m]-\epsilon_m<\up-\epsilon_m<\up[m]<\up\leq1,
\end{equation}
where the last inequality holds because \(I\subseteq[0,1]\), the third and fourth inequalities follow from our choice of \(\up[m]\), the second inequality follows from the fourth, and the first inequality follows from the third.
Furthermore, since \(\sup_{\sit\in\sits}\sqgroup[\big]{\ufrcstsystem_m(\sit)-\lfrcstsystem_m(\sit)}\geq0\), it follows from Equation~\eqref{eq:theo:inherentlyimprecise:1} that \(6\epsilon_m<\up-\lp\), which implies that \(\lp+2\epsilon_m<\lp+4\epsilon_m<\up-2\epsilon_m\).
Combining these inequalities with the ones in Equations~\eqref{eq:theo:inherentlyimprecise:2} and~\eqref{eq:theo:inherentlyimprecise:3}, we finally get that
\begin{equation}\label{eq:theo:inherentlyimprecise:4}
0\leq\lp<\lp[m]<\lp[m]+\epsilon_m<\up[m]-\epsilon_m<\up[m]<\up\leq1.
\end{equation}

With this set-up phase completed, we're now ready to use the map~\(\lambda\), and the \(\lp[m]\), \(\up[m]\), \(\epsilon_m\) and \(\ufrcstsystem'_m\) we have just determined for any~\(m\in\compfrcstsystems[I]\), to define the following precise forecasting system~\(\frcstsystem_{\lp,\up}\):\footnote{Recall that we don't distinguish between a singleton and its single element.}
\begin{equation}\label{eq:inherently:precise:noncomputable}
\frcstsystem_{\lp,\up}(\sit)
\coloneqq
\begin{cases}
\up[\lambda(\dist{\sit})]&\text{if \(\lambda(\dist{\sit})\in\compfrcstsystems[I]\) and \(\ufrcstsystem'_{\lambda(\dist{\sit})}(\sit)\leq\up[\lambda(\dist{\sit})]-2\epsilon_{\lambda(\dist{\sit})}\)}\\
\lp[\lambda(\dist{\sit})]&\text{if \(\lambda(\dist{\sit})\in\compfrcstsystems[I]\) and \(\ufrcstsystem'_{\lambda(\dist{\sit})}(\sit)>\up[\lambda(\dist{\sit})]-2\epsilon_{\lambda(\dist{\sit})}\)}\\
\up&\text{if \(\lambda(\dist{\sit})\notin\compfrcstsystems[I]\)}
\end{cases}
\quad\text{for all~\(\sit\in\sits\)}.
\end{equation}
We also consider any path~\(\pth\in\pths\) that is \(\allowables\)-random---and therefore, due to Proposition~\ref{prop:random:implies:schnorrandom}, also Schnorr random---for this precise forecasting system~\(\frcstsystem_{\lp,\up}\).
We know from Corollary~\ref{cor:consistency} that there is at least one such path; in fact, we know from Theorem~\ref{thm:consistency} that the set of all such paths has lower probability one in the probability tree associated with the precise forecasting system~\(\frcstsystem_{\lp,\up}\).

For any~\(m\in\compfrcstsystems[I]\), we know from Equation~\eqref{eq:theo:inherentlyimprecise:4} that \(\lp<\lp[m]<\up[m]<\up\), so it follows from Equation~\eqref{eq:inherently:precise:noncomputable} that \(\frcstsystem_{\lp,\up}(\sit)\in\pinterval=I\) for all~\(\sit\in\sits\).
This implies that \(\frcstsystem_{\lp,\up}\subseteq\constantfrcstsystem[I]\), where \(\constantfrcstsystem[I]\) is the stationary forecasting system associated with the constant interval forecast~\(I\).
Since \(\pth\) was assumed to be \(\allowables\)-random for~\(\frcstsystem_{\lp,\up}\), it follows from Proposition~\ref{prop:nestedfrcstsystems} that~\(\pth\) is also \(\allowables\)-random for~\(\constantfrcstsystem[I]\).

We will be done if we can show that the path~\(\pth\) isn't Schnorr random for any {\comp} forecasting system~\(\frcstsystem_m\) whose highest imprecision is less than that of~\(I\).
That is, if we can show that the path~\(\pth\) isn't Schnorr random for~\(\frcstsystem_m\) for any~\(m\in\compfrcstsystems[I]\).
This is what we now set out to do.

To this end, we consider an arbitrary but fixed~\(m\in\compfrcstsystems[I]\) and are going to construct a {\comp} test supermartingale for~\(\frcstsystem_m\) that is computably unbounded on~\(\pth\).
Since we know from Equation~\eqref{eq:theo:inherentlyimprecise:4} that \(0<\lp[m]<\lp[m]+\epsilon_m<\up[m]-\epsilon_m<\up[m]<1\), we can define a multiplier process \(\multprocess_m\) by
\begin{equation*}
\multprocess_m(\sit)
\coloneqq
\begin{cases}
f_{\up[m]-\epsilon_m,\up[m]}
&\text{if \(\lambda(\dist{\sit})=m\) and \(\ufrcstsystem'_{m}(\sit)\leq\up[m]-2\epsilon_{m}\)}\\
f_{\lp[m]+\epsilon_m,\lp[m]}
&\text{if \(\lambda(\dist{\sit})=m\) and \(\ufrcstsystem'_{m}(\sit)>\up[m]-2\epsilon_{m}\)}\\
1&\text{if \(\lambda(\dist{\sit})\neq m\)}
\end{cases}
\quad\text{for all~\(\sit\in\sits\)},
\end{equation*}
where, for any~\(\alpha,\beta\in(0,1)\), the gamble~\(f_{\alpha,\beta}\) on~\(\outcomes\) is defined by Equation~\eqref{eq:workhorse:gambles} in Section~\ref{sec:non-stationarity}.
For given~\(\lp[m]\), \(\up[m]\) and~\(\epsilon_m\), this multiplier process \(\multprocess_m\) is {\comp} because \(\lp[m]\), \(\up[m]\) and~\(\epsilon_m\) are rational and because the recursive character of~\(\lambda\), \(a_m\), \(b_m\) and~\(\varsigma_m\) guarantees that the equalities and inequalities in this expression are decidable.
\(\multprocess_m\) is furthermore positive by Lemma~\ref{lem:hulpresultaatvoorinherentlyIP}\ref{it:hulpresultaatvoorinherentlyIP:expectation}.

To prove that \(\multprocess_m\) is a supermartingale multiplier for~\(\frcstsystem_m\), we show that \(\uex_{\frcstsystem_m(\sit)}(\multprocess_m(\sit))\leq1\) for all~\(\sit\in\sits\).
There are, of course, three possible cases.

If \(\lambda(\abs{\sit})\neq m\), it is immediate that \(\uex_{\frcstsystem_m(\sit)}(\multprocess_m(\sit))=\uex_{\frcstsystem_m(\sit)}(1)=1\) [use~\ref{axiom:coherence:bounds} for the last equality].

If \(\lambda(\abs{\sit})=m\) and \(\ufrcstsystem'_{m}(\sit)\leq\up[m]-2\epsilon_{m}\), then also \(\ufrcstsystem_m(\sit)<\ufrcstsystem'_m(\sit)+\epsilon_m\leq\up[m]-\epsilon_m\) because \(\abs{\ufrcstsystem'_m(\sit)-\ufrcstsystem_m(\sit)}<\epsilon_m\).
It therefore follows from Lemma~\ref{lem:hulpresultaatvoorinherentlyIP}\ref{it:hulpresultaatvoorinherentlyIP:intervalbelow} that \(\uex_{\frcstsystem_m(\sit)}(\multprocess_m(\sit))=\uex_{\frcstsystem_m(\sit)}(f_{\up[m]-\epsilon_m,\up[m]})\leq1\).

Finally, we consider the case that \(\lambda(\abs{\sit})=m\) and \(\ufrcstsystem'_{m}(\sit)>\up[m]-2\epsilon_{m}\).
Since Equation~\eqref{eq:theo:inherentlyimprecise:1}
implies that \(\ufrcstsystem_m(\sit)-\lfrcstsystem_m(\sit)<\up-\lp-6\epsilon_m\), we then find that
\begin{align*}
\lfrcstsystem_m(\sit)
>\ufrcstsystem_m(\sit)-\up+\lp+6\epsilon_m
&>\ufrcstsystem'_m(\sit)-\up+\lp+5\epsilon_m
>\up[m]-2\epsilon_{m}-\up+\lp+5\epsilon_m\\
&=\up[m]-\up+\lp+3\epsilon_m
>\lp+2\epsilon_m
>\lp[m]+\epsilon_m
>\lp[m],
\end{align*}
where the second inequality holds because \(\abs{\ufrcstsystem'_m(\sit)-\ufrcstsystem_m(\sit)}<\epsilon_m\), the fourth inequality because \(\up[m]\) was chosen to make sure that \(\up-\epsilon_m<\up[m]\), and the fifth inequality because \(\lp[m]\) was chosen to make sure that \(\lp[m]<\lp+\epsilon_m\).
It therefore follows from Lemma~\ref{lem:hulpresultaatvoorinherentlyIP}\ref{it:hulpresultaatvoorinherentlyIP:intervalabove} that, also in this case, \(\uex_{\frcstsystem_m(\sit)}(\multprocess_m(\sit))=\uex_{\frcstsystem_m(\sit)}(f_{\lp[m]+\epsilon_m,\lp[m]})\leq1\).

So \(\multprocess_m\) is indeed a supermartingale multiplier for~\(\frcstsystem_m\).
Since we had already established that \(\multprocess_m\) is {\comp} and positive, it follows that \(\mint[\multprocess_m]\) is a positive {\comp} test supermartingale for~\(\frcstsystem_m\), also taking into account Proposition~\ref{prop:computable:from:multiplier}.

We're clearly done if we can show that this \(\mint[\multprocess_m]\) is computably unbounded on~\(\pth\).
To do so, consider the multiplier process~\(\multprocess_{m,\lp,\up}\) defined by
\begin{equation*}
\multprocess_{m,\lp,\up}(\sit)
\coloneqq
\begin{cases}
f_{\up[m],\up[m]-\epsilon_m}
&\text{if \(\lambda(\dist{\sit})=m\) and \(\ufrcstsystem'_{m}(\sit)\leq\up[m]-2\epsilon_{m}\)}\\
f_{\lp[m],\lp[m]+\epsilon_m}
&\text{if \(\lambda(\dist{\sit})=m\) and \(\ufrcstsystem'_{m}(\sit)>\up[m]-2\epsilon_{m}\)}\\
1&\text{if \(\lambda(\dist{\sit})\neq m\)}
\end{cases}
\quad\text{for all~\(\sit\in\sits\)}.
\end{equation*}

We first prove that, for given~\(\lp[m]\), \(\up[m]\) and~\(\epsilon_m\), this~\(\multprocess_{m,\lp,\up}\) is a positive {\comp} supermartingale multiplier for the forecasting system~\(\frcstsystem_{\lp,\up}\).
The argumentation is fairly similar to the one given above for~\(\multprocess_m\) and~\(\frcstsystem_m\).
Computability follows from the rationality of~\(\lp[m]\), \(\up[m]\) and~\(\epsilon_m\), and the recursive character of the maps~\(\lambda\), \(a_m\), \(b_m\) and~\(\varsigma_m\).
Positivity follows from Lemma~\ref{lem:hulpresultaatvoorinherentlyIP}\ref{it:hulpresultaatvoorinherentlyIP:expectation}.
To prove that \(\multprocess_{m,\lp,\up}\) is a supermartingale multiplier for~\(\frcstsystem_{\lp,\up}\), we need to show that \(\ex_{\frcstsystem_{\lp,\up}(\sit)}(\multprocess_{m,\lp,\up}(\sit))\leq1\) for all~\(\sit\in\sits\).

The case that \(\lambda(\abs{\sit})\neq m\) is again trivial.

If \(\lambda(\abs{\sit})=m\) and \(\ufrcstsystem'_{m}(\sit)\leq\up[m]-2\epsilon_m\), then \(\frcstsystem_{\lp,\up}(\sit)=\up[m]\) and \(\multprocess_{m,\lp,\up}(\sit)=f_{\up[m],\up[m]-\epsilon_m}\), so it follows from Lemma~\ref{lem:hulpresultaatvoorinherentlyIP}\ref{it:hulpresultaatvoorinherentlyIP:expectation} that \(\ex_{\frcstsystem_{\lp,\up}(\sit)}(\multprocess_{m,\lp,\up}(\sit))=\ex_{\up[m]}(f_{\up[m],\up[m]-\epsilon_m})=1\).

Similarly, if \(\lambda(\abs{\sit})=m\) and \(\ufrcstsystem'_{m}(\sit)>\up[m]-2\epsilon_m\), then \(\frcstsystem_{\lp,\up}(\sit)=\lp[m]\) and \(\multprocess_{m,\lp,\up}(\sit)=f_{\lp[m],\lp[m]+\epsilon_m}\).
Hence, \(\ex_{\frcstsystem_{\lp,\up}(\sit)}(\multprocess_{m,\lp,\up}(\sit))=\ex_{\lp[m]}(f_{\lp[m],\lp[m]+\epsilon_m})=1\), again by Lemma~\ref{lem:hulpresultaatvoorinherentlyIP}\ref{it:hulpresultaatvoorinherentlyIP:expectation}.

So, \(\multprocess_{m,\lp,\up}\) is indeed a positive {\comp} supermartingale multiplier for~\(\frcstsystem_{\lp,\up}\).
Taking into account Proposition~\ref{prop:computable:from:multiplier}, we can conclude that \(\mint[\multprocess_{m,\lp,\up}]\) is a positive {\comp} test supermartingale for~\(\frcstsystem_{\lp,\up}\).
This {\comp} test supermartingale~\(\mint[\multprocess_{m,\lp,\up}]\) must furthermore be bounded above on~\(\pth\), because of the assumed \(\allowables\)-randomness of the path~\(\pth\) for~\(\frcstsystem_{\lp,\up}\).

The idea for the rest of the proof is now that we're going to show that the product process~\(\mint[\multprocess_{m,\lp,\up}]\mint[\multprocess_m]\) is computably unbounded on~\(\pth\), and therefore \(\mint[\multprocess_m]\) must be as well.

Let \(B\) be any rational upper bound on~\(\mint[\multprocess_{m,\lp,\up}]\) along~\(\pth\).
Consider the rational number \(\delta\coloneqq\smash{\group[\big]{1-\frac{1}{4}\epsilon_m^2}^{-1}}\).
Then \(\delta>1\) because \(0<\epsilon_m<1\), and Lemma~\ref{lem:hulpresultaatvoorinherentlyIP}\ref{it:hulpresultaatvoorinherentlyIP:product} guarantees that
\begin{equation*}
\left\{
\begin{aligned}
f_{\up[m]-\epsilon_m,\up[m]}(1)f_{\up[m],\up[m]-\epsilon_m}(1)
&=f_{\up[m]-\epsilon_m,\up[m]}(0)f_{\up[m],\up[m]-\epsilon_m}(0)
\geq\delta\\
f_{\lp[m]+\epsilon_m,\lp[m]}(1)f_{\lp[m],\lp[m]+\epsilon_m}(1)
&=f_{\lp[m]+\epsilon_m,\lp[m]}(0)f_{\lp[m],\lp[m]+\epsilon_m}(0)
\geq\delta.
\end{aligned}
\right.
\end{equation*}
Hence, for any~\(\sit\in\sits\), we find that \(\multprocess_{m}(\sit)\multprocess_{m,\lp,\up}(\sit)\geq\delta\) if \(\lambda(\abs{\sit})=m\), and, otherwise, \(\multprocess_{m}(\sit)\multprocess_{m,\lp,\up}(\sit)=1\).
Consider the map~\(\ordering_m\colon\naturalswithzero\to\naturalswithzero\), defined by
\begin{equation*}
\ordering_m(n)
\coloneqq
\abs[\big]{\cset{k\in\set{0,\dots,n-1}}{\lambda(k)=m}}
\text{ for all~\(n\in\naturalswithzero\)},
\end{equation*}
which is clearly non-decreasing, recursive because \(\lambda\) is, and unbounded because there are infinitely many~\(k\in\naturalswithzero\) for which~\(\lambda(k)=m\).
Then, for any~\(n\in\naturalswithzero\),
\begin{align*}
\mint[\multprocess_{m}](\pthto{n})\mint[\multprocess_{m,\lp,\up}](\pthto{n})
&=\prod_{k=0}^{n-1}\multprocess_{m}(\pthto{k})(\pth_{k+1})\prod_{k=0}^{n-1}\multprocess_{m,\lp,\up}(\pthto{k})(\pth_{k+1})\\
&=\prod_{k=0}^{n-1}\sqgroup[\big]{\multprocess_{m}(\pthto{k})(\pth_{k+1})\multprocess_{m,\lp,\up}(\pthto{k})(\pth_{k+1})}
\geq\smashoperator[r]{\prod_{\substack{k=0\\\lambda(k)=m}}^{n-1}}\delta
=\delta^{\ordering_m(n)},
\end{align*}
and therefore, since \(\mint[\multprocess_{m,\lp,\up}]\) is positive and bounded above by~\(B\) along~\(\pth\), also
\begin{equation*}
\mint[\multprocess_{m}](\pthto{n})
\geq\frac{\delta^{r(n)}}{\mint[\multprocess_{m,\lp,\up}](\pthto{n})}\geq B^{-1}\delta^{\ordering_m(n)}.
\end{equation*}
Now let the map~\(\realordering_m\colon\naturalswithzero\to\nonnegreals\) be defined by~\(\realordering_m(n)\coloneqq B^{-1}\delta^{\ordering_m(n)}\) for all~\(n\in\naturalswithzero\).
Then \(\realordering\) is {\comp} because \(\ordering_m\) is recursive and because \(\delta\) and \(B\) are rational, and \(\realordering\) is non-decreasing and unbounded because \(\ordering_m\) is non-decreasing and unbounded, and because \(\delta>1\).
So, \(\realordering_m\) is a real growth function for which \(\mint[\multprocess_{m}](\pthto{n})\geq\realordering_m(n)\) for all~\(n\in\naturalswithzero\), and therefore also \(\limsup_{n\to\infty}\sqgroup{\mint[\multprocess_{m}](\pthto{n})-\realordering_m(n)}\geq0\).
We find that the {\comp} test supermartingale~\(\mint[\multprocess_{m}]\) for~\(\frcstsystem_m\) is indeed computably unbounded on~\(\pth\), by Proposition~\ref{prop:computably:unbounded:various}\ref{it:computably:unbounded:real}.
\end{proof}

For an example showing that the computability condition in this result can't be dropped, and a discussion on the theoretical and practical relevance of this condition, we refer to recent work by Persiau and ourselves~\cite{persiau2021:nonstationary}.

\section{The meagreness of random sequences}\label{sec:meagreness}
In yet another beautiful paper we came across while researching this topic, Muchnik, Semenov and Uspensky \cite{muchnik1998:metaphysics:of:randomness} showed that the set of all paths that correspond to a precise stationary forecast is meagre.

The essence of their argument is the following.
They call a path~\(\pth\) \emph{lawful} if there is some algorithm that, given as input any situation~\(\sit\) on the path~\(\pth\), outputs a non-trivial finite set~\(R(\sit)\) of situations~\(\altsit\sfollows\sit\) that strictly follow that situation~\(\sit\), such that one of these `extensions'~\(\altsit\) is also on the path---meaning that \(\pth\in\exact{\altsit}\).
By `non-trivial', they mean that~\(R(\sit)\) is restrictive: it actually eliminates possible extensions.
They then go on to show that the set of all lawful paths is meagre, and finally, that random paths, because they satisfy the law of large numbers, are lawful.

In this section, we show that we can extend this argument to imprecise stationary forecasts.
This will show that, while, due to Theorem~\ref{thm:consistency}, almost all paths are random for a forecasting system---so the random paths are legion---in a \emph{measure-theoretic} sense, in the specific \emph{topological} sense of meagreness, they are few.

First of all, let us give a definition of lawfulness that makes the formulation above more precise; see also Figure~\ref{fig:lawfulness}.
A \emph{partial} function on a domain~\(D\) is a function that need not be defined on all elements of~\(D\), so need not be a map.

\begin{definition}[Lawfulness {\protect\cite[Definition~2.1]{muchnik1998:metaphysics:of:randomness}}]\label{def:lawfulness}
We call \emph{algorithm} any recursive (partial) function~\(R\) from~\(\sits\) to the collection of finite subsets of~\(\sits\).
A path~\(\pth\in\pths\) is called \emph{lawful for an algorithm~\(R\)} if for all~\(m\in\naturalswithzero\):
\begin{enumerate}[label=\upshape(\roman*),leftmargin=*,noitemsep,topsep=0pt]
\item\label{it:lawfulness:defined} \(R\) is defined in the situation~\(\pthto{m}\);
\item\label{it:lawfulness:extends} \(R(\pthto{m})\) is a non-empty finite subset of~\(\sits\) such that \(\pthto{m}\sprecedes\altsit\) for all~\(\altsit\in R(\pthto{m})\);
\item\label{it:lawfulness:nontrivial} \(R(\pthto{m})\) is \emph{non-trivial}: \(\bigcup_{\altsit\in R(\pthto{m})}\exact{t}\subset\exact{\pthto{m}}\);
\item\label{it:lawfulness:consistent} there is some~\(\altsit\in R(\pthto{m})\) such that \(\pth\in\exact{\altsit}\).
\end{enumerate}
A path~\(\pth\in\pths\) is called \emph{lawful} if it is lawful for some algorithm~\(R\).
A path that isn't lawful is called \emph{lawless}.
\end{definition}

\begin{figure}[h]
\begin{center}
\pgfdeclarelayer{background}
\pgfsetlayers{background,main}
\begin{tikzpicture}[scale=.5]
\tikzstyle{level 1}=[sibling distance=8em]
\tikzstyle{level 2}=[sibling distance=4em]
\tikzstyle{level 3}=[sibling distance=2em]
\coordinate (sit) at (0,0);
\coordinate (oneleft) at (-2cm,0cm);
\coordinate (twoleft) at (-4cm,0cm);
\draw[red] (twoleft) -- (oneleft) -- (sit);
\node[circle,inner sep=1.25pt,fill=red] at (oneleft) {};
\node[circle,inner sep=1.25pt,fill=white] at (twoleft) {\(\cdots\)};
\node[circle,inner sep=1.25pt,fill=red] at (sit) {}
[grow=right,level distance=2cm]
child[black] {node[circle,inner sep=1pt,fill=blue] (a) {}
  child {node[circle,inner sep=1pt,fill=blue] (aa) {}
    child {node[circle,inner sep=1pt,fill=blue] (aaa) {}}
    child {node[circle,inner sep=1pt,fill=blue] (aab) {}}
  }
  child {node[circle,inner sep=1pt,fill=blue] (ab) {}
    child {node[circle,inner sep=1pt,fill=blue] (aba) {}}
    child {node[circle,inner sep=1pt,fill=blue] (abb) {}}
  }
}
child[red] {node[circle,inner sep=1.25pt,fill=red] (b) {}
  child {node[circle,inner sep=1.25pt,fill=red] (ba) {}
    child[black] {node[circle,inner sep=1pt,fill=blue] (baa) {}}
    child {node[circle,inner sep=1.25pt,fill=red] (bab) {}}
  }
  child[black] {node[circle,inner sep=1pt,fill=blue] (bb) {}
    child {node[circle,inner sep=1pt,fill=blue] (bba) {}}
    child {node[circle,inner sep=1pt,fill=blue] (bbb) {}}
  }
};
\coordinate (oneright) at ($(bab)+(2cm,0cm)$);
\coordinate (tworight) at ($(bab)+(4cm,0cm)$);
\draw[red] (bab) -- node[above,midway,red] {\small\(\pth\)} (oneright) -- (tworight);
\node[circle,inner sep=1.25pt,fill=red] at (oneright) {};
\node[circle,inner sep=1.25pt,fill=white] at (tworight) {\(\cdots\)};
\node[above] at (sit) {\small\(\sit\)};
\begin{pgfonlayer}{background}
\draw[lightgray,line width=8pt,line join=round,line cap=round] (bab) -- (ab) -- (aa);
\node[gray,above left] at (ab) {\small\(R(\sit)\)};
\end{pgfonlayer}
\end{tikzpicture}
\end{center}
\caption{Visualisation of the main ideas behind the lawfulness of a path~\(\pth\)}
\label{fig:lawfulness}
\end{figure}
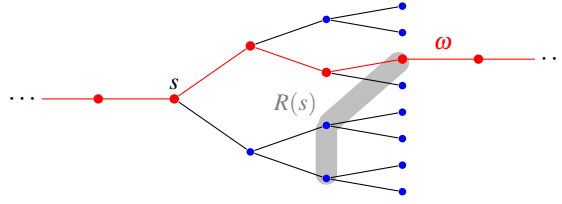

A set of paths \(A\subseteq\pths\) is \emph{nowhere dense in~\(\pths\)} \cite{muchnik1998:metaphysics:of:randomness} if for every \(\sit\in\sits\), there is some~\(\altsit\in\sits\) such that \(\sit\precedes\altsit\) and \(A\cap\exact{\altsit}=\emptyset\).
A set of paths \(B\subseteq\pths\) is then called \emph{meagre}, or \emph{first category}, it it is a countable union of nowhere dense sets.
We will rely on the following central result in Ref.~\cite{muchnik1998:metaphysics:of:randomness}.

\begin{theorem}[{\protect\cite[Corollary~2.3]{muchnik1998:metaphysics:of:randomness}}]\label{thm:meagre}
Any subset of \(\pths\) containing only lawful paths is meagre.
\end{theorem}

To prove that a set of random paths is meagre, it therefore suffices to prove that these random paths are all lawful.
This turns out to be not too difficult, because the following proposition shows that relative frequencies along lawless paths behave very `wildly'.

\begin{proposition}\label{prop:lawful:liminfs:and:limsups}
Let \(\pth\in\pths\) be a lawless path.
Then
\begin{equation*}
\liminf_{n\to\infty}\frac{1}{n}\smashoperator{\sum_{k=1}^n}\pthat{k}=0
\text{ and }
\limsup_{n\to\infty}\frac{1}{n}\smashoperator{\sum_{k=1}^n}\pthat{k}=1.
\end{equation*}
\end{proposition}

\begin{proof}
We give the proof for the \(\limsup\).
The proof for the \(\liminf\) is completely analogous.

Assume {\itshape ex absurdo} that \(\limsup_{n\to\infty}\frac{1}{n}\sum_{k=1}^n\pthat{k}<1\).
Then there are~\(n_o,r_o\in\naturals\) with~\(r_o>1\) and~\(n_o>1\) such that
\begin{equation}\label{eq:limsup:not:one}
\frac{1}{m}\smashoperator{\sum_{k=1}^m}\pthat{k}<1-\frac{1}{r_o}
\text{ for all~\(m\geq n_o\)}.
\end{equation}

For any given~\(n,r\in\naturals\) with~\(r>1\) and~\(n>1\), we construct an algorithm~\(R_{n,r}\) as follows.
For any~\(\sit\in\sits\) with~\(m\coloneqq\dist{\sit}\geq n\), we let \(R_{n,r}(\sit)\) be the set of all situations~\(\altsit\in\sits\) such that
\begin{equation}\label{eq:algorithm:conditions}
\dist{\altsit}=rm
\text{ and }
\sit\precedes\altsit
\text{ and }
\smashoperator{\sum_{k=1}^{rm}}\altsitat{k}<mr-m.
\end{equation}
We also let \(R_{n,r}(\sit)\coloneqq R_{n,r}(\pthto{n})\) for all~\(\sit\in\sits\) with~\(\dist{\sit}<n\).
It is clear that \(R_{n,r}\) is a recursive map, because the conditions in Equation~\eqref{eq:algorithm:conditions} are decidable.

If we can now show that~\(\pth\) is lawful for~\(R_{n_o,r_o}\), and therefore lawful, we will have a contradiction.
To do so, fix any~\(m\in\naturalswithzero\).
It follows from the construction of the~\(R_{n,r}\) that
\begin{multline}\label{eq:algorithm:on:path}
R_{n_o,r_o}(\pthto{m})\\
=
\cset[\bigg]{\altsit\in\sits}
{\dist{\altsit}=r_o\max\set{n_o,m},
\pthto{\max\set{n_o,m}}\precedes\altsit
\text{ and }
\frac{1}{r_o\max\set{n_o,m}}\smashoperator{\sum_{k=1}^{r_o\max\set{n_o,m}}}\altsitat{k}<1-\frac{1}{r_o}},
\end{multline}
so \(R_{n_o,r_o}\) is indeed defined on~\(\pthto{m}\), which implies that~\(R_{n_o,r_o}\) satisfies requirement~\ref{it:lawfulness:defined} in Definition~\ref{def:lawfulness}.

Combining Equations~\eqref{eq:limsup:not:one} and~\eqref{eq:algorithm:on:path} tells us that \(\pthto{r_o\max\set{n_o,m}}\in R_{n_o,r_o}(\pthto{m})\), so requirement~\ref{it:lawfulness:consistent} in Definition~\ref{def:lawfulness} is also satisfied.

We also gather from Equation~\eqref{eq:algorithm:on:path} that the depth \(\dist{t}\) of all situations~\(t\) in~\(R_{n_o,r_o}(\pthto{m})\) is~\(r_o\max\set{n_o,m}\), so \(R_{n_o,r_o}(\pthto{m})\) is finite.
Since \(r_o>1\), we see that \(r_o\max\set{n_o,m}>m\), so all situations in~\(R_{n_o,r_o}(\pthto{m})\) are strictly preceded by~\(\pthto{m}\).
Moreover, we have just proved above that~\(R_{n_o,r_o}(\pthto{m})\) is non-empty.
This tells us that~\(R_{n_o,r_o}\) satisfies requirement~\ref{it:lawfulness:extends} in Definition~\ref{def:lawfulness}.

Finally, consider the situation~\(\ualtsit\) defined by~\(\dist{\ualtsit}\coloneqq r_o\max\set{n_o,m}\) and
\begin{equation*}
\ualtsitto{\max\set{n_o,m}}\coloneqq\pthto{\max\set{n_o,m}}
\text{ and }
\ualtsitat{k}\coloneqq1
\text{ for~\(\max\set{n_o,m}+1\leq k\leq r_o\max\set{n_o,m}\)},
\end{equation*}
then \(\ualtsit\notin R_{n_o,r_o}(\pthto{m})\) by Equation~\eqref{eq:algorithm:on:path}, because
\begin{equation*}
\frac{1}{r_o\max\set{n_o,m}}\smashoperator{\sum_{k=1}^{r_o\max\set{n_o,m}}}\ualtsitat{k}
\geq\frac{r_o\max\set{n_o,m}-\max\set{n_o,m}}{r_o\max\set{n_o,m}}
=1-\frac{1}{r_o}
\end{equation*}
where the inequality follows from~\(\ualtsitat{k}\geq0\) for~\(1\leq k\leq\max\set{n_o,m}\).
This tells us that~\(R_{n_o,r_o}\) satisfies requirement~\ref{it:lawfulness:nontrivial} in Definition~\ref{def:lawfulness}.

We conclude that \(\pth\) is indeed lawful for~\(R_{n_o,r_o}\).
\end{proof}

So, in order to prove our result, it now suffices to consider that relative frequencies along random paths can't behave so wildly, because they're constrained by our `weak {\comp} stochasticity' result in Corollary~\ref{cor:well:calibrated:constant:schnorr}.
\emph{Random paths are typically lawful.}

\begin{theorem}\label{thm:the:set:of:random:sequences:is:meagre}
Let \(I=\pinterval\in\imprecisefrcsts\) be any closed subinterval of \(\frcsts\) strictly included in~\(\frcsts\), so \(\lp>0\) or \(\up<1\).
Then the set of all paths that are \(\allowables\)-random for the stationary forecasting system~\(\constantfrcstsystem[I]\) is meagre.
Similarly, the set of all Schnorr random paths for~\(\constantfrcstsystem[I]\) is meagre.
\end{theorem}

\begin{proof}
Consider any~\(\pth\) that is \(\allowables\)-random for the stationary forecasting system~\(\constantfrcstsystem[I]\).
Then it clearly suffices to show that \(\pth\) is lawful, by Theorem~\ref{thm:meagre}.
Assume {\itshape ex absurdo} that it is lawless.
By combining Proposition~\ref{prop:lawful:liminfs:and:limsups} with Corollary~\ref{cor:well:calibrated:constant:schnorr} [with recursive selection function \(\selectionfunction\coloneqq1\)], we find that \(\lp\leq0\) and \(\up\geq1\), a contradiction.
The proof for Schnorr randomness is identical.
\end{proof}

Our argumentation shows that the important distinction for random paths does not lie between precise and imprecise stationary forecasts, but rather between vacuous and non-vacuous forecasts: for any non-vacuous stationary forecast, the set of random paths is meagre, whereas for the vacuous stationary forecast, all paths are random, and therefore the corresponding set of random paths is \emph{co-meagre}---the complement of a meagre set.
We also see that the paths that are random for non-vacuous interval forecasts are `equally rare' as those that are random for precise forecasts, which, we believe, only adds to their mathematical interest.

\section{Conclusion}
The probability of an event is often seen as a precise, or at least ideally precise, number.
Apart from a few notable exceptions in earlier accounts \cite{shafer1978,dempster1967,shafer1976,koopman1940}, a more determined investigation into reasons for letting go of this idealisation, and into mathematical ways to achieve this, only started in the later decades of the 20th century \cite{levi1980a,walley1991,seidenfeld1995,shafer2001,seidenfeld1999,kyburg1988}; see also Refs.~\cite{augustin2013:itip,troffaes2013:lp} for overviews.
Most of this work centred on the decision-theoretic and epistemic aspects of probability \cite{levi1980a,walley1991}, while some contributions were more agnostic in this respect \cite{shafer2001,shafer2019:book}, but a few attempts were also made \cite{walley1982a,fierens2009:frequentist,fierens2009:chaotic,gorban2016:statisticalstability} to justify letting go of precision also for probabilities with a more physical, or frequentist, interpretation.

We believe this paper is the first systematic attempt at reconciling imprecision with the study of (frequency-based or algorithmic) randomness along the lines of von Mises \cite{mises1981}, Church \cite{church1940:random:sequence}, Kolmogorov \cite{kolmogorov1965}, Ville \cite{ville1939}, {\ML} \cite{martinlof1966:random:sequences}, Levin~\cite{levin1973:random:sequence} and Schnorr \cite{schnorr1971,schnorr1973}, to name only a few of the early protagonists.
We show that this is both possible and interesting.
We see that, besides the sequences that are random for precise forecasts, new realms of sequences arise that are random for interval forecasts.
They have intriguing properties, and are as topologically rare as their precise counterparts, in the sense that they also constitute meagre sets.
Even with the limited number of examples we have examined here, it should be apparent that incorporating imprecision---or interval forecasts---into the study of randomness allows for more mathematical structure to arise.
This is relevant in and of itself, but we would argue that our treatment also allows us to better understand and place, as special cases, the existing results in the precise limit.
This, by the way, also holds true for the more epistemic accounts of imprecision in probability; see for instance Ref.~\cite{troffaes2013:lp} for a more detailed account.

This leads us to our rather provocative title for this paper.
A number of people, while not averse to the idea of allowing for imprecision in the study of randomness, advised us to tone down our claim that `Randomness is inherently imprecise', and suggested replacing it by something weaker, such as `Some aspects of randomness cannot be adequately dealt with using precise forecasts only'.  
Evidently, we have decided against that, and it behoves us here to explain our reasons for doing so, other than the plainly polemic or merely rhetorical ones.  
Simply stated, we are actually convinced that the statement in our title holds true, and that there is more to randomness than the by now classical account for precise forecasts would have us suspect.

First off, we have seen in Section~\ref{sec:non-stationarity} that on the one hand, `imprecise randomness' can arise as a useful stationary model simplification when dealing with non-stationarity, which points to practical reasons for allowing imprecision in the study of randomness.
But, on the other hand, we've also been led, in Section~\ref{sec:inherently}, to the conclusion that imprecise randomness has a more fundamental role, because there are sequences that are random for a given {\comp} interval forecast, but not for any {\comp} (more) precise forecast.

Of course, we agree that the \emph{initial ideas} for characterising what randomness is, were based on probabilistic limit laws such as the convergence of relative frequencies, which as Corollaries~\ref{cor:well:calibrated:constant} and~\ref{cor:well:calibrated:constant:schnorr} suggest, are harder to guarantee in exactly the same form in our imprecise context.
But the seminal martingale-theoretic accounts of Ville~\cite{ville1939}, Schnorr~\cite{schnorr1971,schnorr1973} and Levin~\cite{levin1973:random:sequence}, as well as the prequential approach by Vovk and Shen~\cite{vovk2010:randomness}, have opened up the path towards useful alternative characterisations, which are much more amenable to letting go of the ideal of precision, as we have shown here. 

We see very few reasons for holding on to that perceived ideal, neither from a practical (forecasting and calibration) point of view, nor from a more formal mathematical stance.
We have already hinted above at the formal mathematical reasons for allowing for imprecision in the study of randomness: restricting ourselves to precise forecasts hides \emph{interesting and useful} mathematical structure that, as we see in this paper but also have come to witness in our most recent---still largely unpublished---research efforts, reveals relevant facts even about the precise aspects of randomness.
Again, such structural arguments for allowing for imprecision can, incidentally, also be brought to bear in more epistemic and decision-theoretic accounts of uncertainty; see for instance Ref.~\cite{troffaes2013:lp} for examples and related discussion.
And as randomness in more recent accounts has been linked with forecasting and calibration, many of the arguments for allowing for imprecision and indecision (see for instance~\cite[Chapter~5]{walley1991} for extensive discussion) in epistemic uncertainty modelling become relevant to the study of randomness also.

One aspect of what we are saying then, can be summarised as follows: (algorithmic) randomness seems to have acquired a broader meaning and to have become more strongly connected to other issues than only probabilistic convergence laws, and this change of focus makes allowing for imprecision in its foundations much more intuitive and less far-fetched than it may once have seemed.   

What else do we mean when we say that `randomness is inherently imprecise'?
Randomness, as we perceive it, is about outcome sequences (paths) and forecasting systems `going together well'.
And this going together well has certain implications, which for a given forecasting system, impose restrictions on the behaviour of the successive outcomes in a random path, such as the limit laws in Theorems~\ref{thm:well:calibrated:general} and~\ref{thm:well:calibrated:general:schnorr}, and in Corollaries~\ref{cor:well:calibrated:constant} and~\ref{cor:well:calibrated:constant:schnorr} for stationary forecasts.
It is of crucial importance to our argument that these restrictions do not all of a sudden disappear when going from point to interval forecasts.
They weaken, perhaps, but don't disappear, as is also made very clear by our discussion in Section~\ref{sec:meagreness}: the hard cut-off there lies not between precise and imprecise forecasts, but between non-vacuous and vacuous ones.
For \emph{any non-vacuous} stationary forecast, be it precise or imprecise, the restrictions on the corresponding sets of random paths are substantial, and result in these sets being meagre, as Theorem~\ref{thm:the:set:of:random:sequences:is:meagre} testifies. 
It is only for vacuous forecasts that the restrictions disappear and that the corresponding randomness notion becomes vacuous as well: all paths are random for such forecasts.
Thus, the difference between non-vacuous imprecise and precise randomness may be a matter of degree, perhaps and in certain respects, but not of quality.
Why then single out the precise limit case as the only one worthy of the moniker `randomness'?

This work may seem promising, but we're well aware that it is only a humble beginning.
We see many extensions in many directions, so let us briefly discuss a few.

First of all, our preliminary exploration suggests that it will be possible to formulate equivalent randomness definitions in terms of \emph{randomness tests}, rather than supermartingales.
We believe it would be relevant to work this out in much more detail.

Secondly, the approach we follow here is not prequential: we assume that our Forecaster specifies an entire forecasting system~\(\frcstsystem\), or in other words an interval forecast in all possible situations~\((\xvaltolong[n])\), rather than only interval forecasts in those situations~\((\zvaltolong[n])\) of the sequence~\(\pth=(\zvaltolong[n],\dots)\) whose potential randomness we're considering.
The \emph{prequential approach}, which we eventually will want to come to, looks at the randomness of a sequence of interval forecasts and outcomes~\((I_1,z_1,I_2,z_2,\dots,I_n,z_n,\dots)\), where each \(I_k\) is an interval forecast for the as yet unknown \(\randomoutcome[k]\), which is afterwards revealed to be \(z_k\), without the need for a specification of forecasts in other situations that are never reached; see the paper by Vovk and Shen \cite{vovk2010:randomness} for an account of how this works for precise forecasts and {\ML} randomness.

Thirdly, we perceive the need to connect our work more firmly with earlier approaches to associating imprecision with randomness through unstable relative frequencies and non-stationarity, most notably by Terrence Fine's group \cite{walley1982a,fierens2009:frequentist,fierens2009:chaotic}.

And finally, and perhaps most importantly, we believe this research could be a very early starting point for a more systematic approach to statistics that takes imprecise or set-valued parameters more seriously, when learning from finite amounts of data.
Ahead of this, more work must be done to extend our mathematical formulation to non-binary outcomes, to name just one important generalisation; see Ref.~\cite{persiau2020:randomness:more:than:probabilities} for a step in this direction.

\section*{Acknowledgements}
This paper has taken a very long time to write.
Our research on this topic started with discussions between Gert and Philip Dawid about what prequential interval forecasting would look like, during a joint stay at Durham University in late 2014.
Gert, and Jasper who joined in late 2015, wrote an early prequential version of the present paper during a joint research visit to the University of Strathclyde and Durham University in May 2016, trying to extend the results in Refs.~\cite{vovk1987:randomness,vovk2010:randomness,vovk2009:merging} to make them allow for interval forecasts.
In an email exchange, Volodya Vovk pointed out a number of difficulties with our approach, which we were able to resolve by letting go of its prequential emphasis, at least for the time being.
This was done during research visits of Gert to Jasper at IDSIA in Lugano in late 2016 and early 2017.
This paper then lay dormant for a while, while Gert took up a position as director of studies between 2016 and 2019, and both of us were more strongly focused on issues dealing with coherent choice functions.
We finished the conceptual work on this paper during a joint research stay in Siracusa in January 2020, and wrote it all down later at home during successive Covid-19 lockdowns, in the spring of 2020, and the early months of 2021.

As with most of our joint work, there is no telling, after a while, which of us had what idea, or did what, exactly.
We have both contributed equally to this paper.
But since a paper must have a first author, we decided it should be the one who took the first significant steps: Gert, in this case.

We are grateful to Philip Dawid and Volodya Vovk for their inspiring and helpful comments and guidance, and to Gert Vermeulen for introducing us to the wonders of Archimedes' ancient home town.
Teddy Seidenfeld, Glenn Shafer and Alexander Shen, as well as a number of anonymous reviewers, have helped with useful suggestions and constructive criticism.
Gert's research and travel were partly funded through project number G012512N of the Research Foundation -- Flanders (FWO). Jasper was a Post-Doctoral Fellow of the FWO when much of this research was being done, and he wishes to acknowledge its financial support.
In more recent years, Gert and Jasper's work was also supported by H2020-MSCA-ITN-2016 UTOPIAE, grant agreement 722734.


\appendix
\section{Proofs of results about lower and upper expectations, computability and growth functions}

\begin{proof}[Proof of the claims about infimum/minimum selling prices]
In Section~\ref{sec:single:forecast}, we advanced the claim that working with infimum acceptable selling prices is, in the context of this paper, equivalent to working with minimum acceptable selling prices, and similarly for supremum and maximum acceptable buying prices.
Let us spend some effort here to understand why that is.

If \(\uex_I(f)\) is interpreted as a \emph{minimum} acceptable selling price for~\(f(\randomoutcome)\), this means that Forecaster is willing to sell the uncertain reward~\(f(\randomoutcome)\) for any price~\(\beta\) down to \emph{and including}~\(\uex_I(f)\).
If, on the other hand, \(\uex_I(f)\) is only interpreted as an \emph{infimum} acceptable selling price for~\(f(\randomoutcome)\), then this means that Forecaster is willing to sell the uncertain reward~\(f(\randomoutcome)\) for any price~\(\beta\) strictly higher than~\(\uex_I(f)\), but \emph{nothing is stated} about his actual willingness to sell for the price~\(\uex_I(f)\) itself.

To clarify this idea, let us denote by \(\sells\) the set of all gambles \(f\colon\outcomes\to\reals\) that Forecaster accepts to give away, and that are therefore available to Sceptic.
If pay-offs are expressed in units of linear utility, this set~\(\sells\) will arguably be a convex cone, and it will include the non-positive gambles \(f\leq0\), simply because it is rational for Forecaster to give away a partial loss \cite{walley1991,troffaes2013:lp,quaeghebeur2012:itip}.

An upper expectation~\(\uex_I(f)\) is typically interpreted as Forecaster's infimum acceptable selling price for the gamble~\(f(\randomoutcome)\), meaning that \cite{walley1991,troffaes2013:lp}
\begin{align}
\uex_I(f)
&=\inf\cset{\beta\in\reals}{\text{Forecaster accepts to sell \(f(\randomoutcome)\) for price~\(\beta\)}}\notag\\
&=\inf\cset{\beta\in\reals}{f-\beta\in\sells}\label{eq:selling:price}.
\end{align}
We then see that the functional~\(\uex_I\) can be used to characterise the convex cone~\(\sells\), but \emph{only up to border behaviour}, as Figure~\ref{fig:infimumvsminimum} and the following argumentation clarify.
Indeed, we can readily infer the following implications from Equation~\eqref{eq:selling:price} and the fact that \(\sells\) is a convex cone that includes all non-positive gambles:
\begin{equation*}
\group{f\in\sells\then\uex_I(f)\leq0}
\text{ and }
\group[\big]{\group{f\leq0\text{ or }\uex_I(f)<0}\then f\in\sells}.
\end{equation*}
So we see that the \emph{marginal gambles}~\(g\)---those gambles for which~\(\uex_I(g)=0\)---are the only ones for which the difference in interpretation between infimum and minimum acceptable buying prices matters in the context of this paper.
When \(\uex_I(g)\) is interpreted as a \emph{minimum} acceptable selling price, this implies that Forecaster will actually make the marginal gamble~\(g\) available to Sceptic.
But when \(\uex_I(g)\) is interpreted only as an \emph{infimum} acceptable selling price, then the fact that \(\uex_I(g)=0\) tells us nothing about whether Forecaster will make~\(g\) available to Sceptic: he may, or he may not.

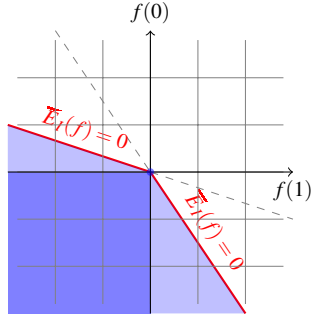
\begin{figure}[h]
\centering
\begin{tikzpicture}[scale=1.25]\footnotesize
\coordinate (xaxis) at (1.5,0);
\coordinate (yaxis) at (0,1.5);
\coordinate (origin) at (0,0);
\coordinate (bottomleft) at (-1.5,-1.5);
\coordinate (bottomright) at (1,-1.5);
\coordinate (bottom) at (0,-1.5);
\coordinate (left) at (-1.5,0);
\coordinate (aboveleft) at (-1.5,0.5);
\fill[nearly transparent,blue] (origin) -- (belowright) -- (bottom) -- cycle;
\coordinate (belowright) at (1,-1.5);
\draw[help lines,dashed] (origin) -- (1.5,-0.5);
\draw[help lines,dashed] (-1,1.5) -- (origin);
\draw[red,thick] (aboveleft) -- node[midway,above,rotate=-18.43] {\footnotesize \(\uex_{I}(f)=0\)} (origin);
\draw[red,thick] (origin) -- node[midway,above,rotate=-56.31] {\footnotesize \(\uex_{I}(f)=0\)} (belowright);
\fill[nearly transparent,blue] (origin) -- (left) -- (aboveleft) -- cycle;
\fill[nearly transparent,blue] (origin) -- (bottomright) -- (bottom) -- cycle;
\fill[semitransparent,blue] (origin) -- (left) -- (bottomleft) -- (bottom) -- cycle;
\draw[step=.5cm,gray,very thin] (-1.4,-1.4) grid (1.4,1.4);
\draw[->] (-1.5,0) -- (xaxis) node[below] {\footnotesize\(f(1)\)};
\draw[->] (0,-1.5) -- (yaxis) node[above] {\footnotesize\(f(0)\)};
\node[circle,fill,semitransparent,blue,inner sep=1pt] at (origin) {};
\end{tikzpicture}
\caption{Depiction of the convex cone of gambles that Forecaster accepts to give away, and that are therefore available to Sceptic.
The region that is shaded (lighter or darker) blue, with the exclusion of the border (in red), depicts the gambles~\(f\) that Forecaster will definitely accept to give away, because they have a negative infimum acceptable selling price~\(\uex_I(f)<0\) or because they're non-positive (the darker blue ones). The marginal gambles~\(f\) (in red) are the ones for which the infimum acceptable selling price~\(\uex_I(f)\) is zero.}
\label{fig:infimumvsminimum}
\end{figure}

The interpretation of~\(\uex_I(f)\) as a \emph{minimum} acceptable selling price is essentially reflected in Equation~\eqref{eq:supermartingale}, where we define a supermartingale~\(\supermartin\) by requiring that in each situation~\(\sit\) its process difference \(\adddelta\supermartin(\sit)\) \emph{must be available to Sceptic}, because Forecaster is willing to sell it (to Sceptic) for some price lower than or equal to~\(0\):
\begin{equation*}
\uex_{\frcstsystem(\sit)}(\adddelta\supermartin(\sit))\leq0.
\end{equation*}
Were we to interpret~\(\uex_I(f)\) only as an \emph{infimum} acceptable selling price, this would essentially mean that in each situation~\(\sit\) the uncertain reward~\(\adddelta\supermartin(\sit)-\delta\) must be available to Sceptic for all~\(\delta>0\), but \emph{not necessarily} for~\(\delta=0\); but now, of course, as indicated above, we must also take into account that a non-positive~\(\adddelta\supermartin(\sit)\leq0\) will also be available to Sceptic.
We could have this reflected conservatively in the \emph{strict supermartingale condition}:
\begin{equation}\label{eq:strict:supermartingale}
\adddelta\supermartin(\sit)\leq0\text{ or }\uex_{\frcstsystem(\sit)}(\adddelta\supermartin(\sit))<0
\text{ for all \(\sit\in\sits\)}.
\end{equation}
We will call any process satisfying this requirement~\eqref{eq:strict:supermartingale} a \emph{strict supermartingale} for~\(\frcstsystem\).
Proving our statement about the equivalence of infimum and minimum acceptable selling prices in this randomness context then amounts to showing that our randomness notions remain unaffected by replacing supermartingales with strict supermartingales in their definition.

Because the strict supermartingale condition is stronger than the supermartingale condition, it is clearly enough to show that any path~\(\pth\) that is not random in the supermartingale sense, also can't be random in the strict supermartingale sense.
To prove this, consider any test supermartingale~\(\test\), and define the real process~\(\test'\) by
\begin{equation*}
\test'(\sit)
\coloneqq\dfrac{\test(\sit)+\frac{1}{\dist{\sit}+1}}{2}
\text{ for all \(\sit\in\sits\)}.
\end{equation*}
Since
\begin{equation*}
\adddelta\test'(\sit)
=\frac12\adddelta\test(\sit)+\frac12\group[\bigg]{\frac{1}{\dist{\sit}+2}-\frac{1}{\dist{\sit}+1}}
=\frac12\adddelta\test(\sit)-\frac12\frac{1}{(\dist{\sit}+1)(\dist{\sit}+2)},
\end{equation*}
it follows readily [from~\ref{axiom:coherence:homogeneity} and~\ref{axiom:coherence:constantadditivity}] that \(\test'\) is a strict test supermartingale.
It is moreover (lower semi)computable when \(\test\) is, and it is (computably) unbounded on the same paths as~\(\test\).
This implies that our notions of {\ML} randomness, {\comp} randomness and Schnorr randomness indeed remain unaffected by replacing supermartingales with strict supermartingales in their definition.

For weak {\ML} randomness, we need a slightly different argument.
Consider any {\lscomp} supermartingale multiplier \(\multprocess\), and the related multiplier process \(\multprocess'\) defined by
\begin{equation*}
\multprocess'(\sit)\coloneqq\multprocess(\sit)\dfrac{(\dist{\sit}+1)(\dist{\sit}+3)}{(\dist{\sit}+2)^2}
\text{ for all \(\sit\in\sits\)}.
\end{equation*}
Then \(\multprocess'\) is clearly also {\lscomp}, and
\begin{equation}\label{eq:strict:multiplier:supermartingale:condition}
\uex_{\frcstsystem(\sit)}(\multprocess'(\sit))
=\underset{<1}{\underbrace{\dfrac{(\dist{\sit}+1)(\dist{\sit}+3)}{(\dist{\sit}+2)^2}}}
\uex_{\frcstsystem(\sit)}(\multprocess(\sit))
<1
\text{ for all \(\sit\in\sits\)},
\end{equation}
where the equality follows from~\ref{axiom:coherence:homogeneity}, and the strict inequality from the assumption that~\(\multprocess\) is a supermartingale multiplier, and~\ref{axiom:coherence:bounds}.
Hence, \(\multprocess'\) is a supermartingale multiplier as well.
Now, consider the real process \(\test'\) defined by
\begin{equation}\label{eq:strict:multiplier:supermartingale}
\test'(\sit)\coloneqq\mint[\multprocess](\sit)\frac12\dfrac{\dist{\sit}+2}{\dist{\sit}+1}
\text{ for all \(\sit\in\sits\)}.
\end{equation}
Then \(\test'\) is non-negative since \(\mint[\multprocess]\) is, and \(\test'(\init)=1\), so \(\test'\) is a test process.
Also,
\begin{multline*}
\test'(\sit\cdot)
=\mint[\multprocess](\sit\cdot)\frac12\dfrac{\dist{\sit}+3}{\dist{\sit}+2}
=\mint[\multprocess](\sit)\multprocess(\sit)\frac12\dfrac{\dist{\sit}+3}{\dist{\sit}+2}\\
=\test'(\sit)\multprocess(\sit)\dfrac{(\dist{\sit}+1)(\dist{\sit}+3)}{(\dist{\sit}+2)^2}
=\test'(\sit)\multprocess'(\sit)
\text{ for all \(\sit\in\sits\)},
\end{multline*}
which shows that \(\test'\) is the test supermartingale generated by the supermartingale multiplier~\(\multprocess'\).
Equation~\eqref{eq:supermartingale:multiplier:differences} now tells us that
\begin{equation*}
\uex_{\frcstsystem(\sit)}(\adddelta\test'(\sit))
=\test'(\sit)\sqgroup[\big]{\uex_{\frcstsystem(\sit)}(\multprocess'(\sit))-1}
\text{ for all~\(\sit\in\sits\)}.
\end{equation*}
The same Equation~\eqref{eq:supermartingale:multiplier:differences} also guarantees that if \(\test'(\sit)=0\), then also~\(\adddelta\test'(\sit)=0\), and therefore it follows readily from Equation~\eqref{eq:strict:multiplier:supermartingale:condition} that \(\test'\) satisfies the strict supermartingale condition~\eqref{eq:strict:supermartingale}.
Finally, Equation~\eqref{eq:strict:multiplier:supermartingale} guarantees that \(\test'\) and \(\mint[\multprocess]\) become unbounded on the same paths.
\end{proof}

\begin{proof}[Proof of Proposition~\ref{prop:properties:of:global:expectations}]
We begin by proving that \(\inf f\leq\uglobal(f)\leq\sup f\).
Conjugacy will then imply that also \(\inf f\leq\lglobal(f)\leq\sup f\), and therefore that both \(\lglobal(f)\) and \(\uglobal(f)\) are real numbers.
This fact will then be used further on.
The remainder of statement~\ref{axiom:lower:upper:bounds} will be proved further below.
Since all constant real processes are supermartingales [by~\ref{axiom:coherence:bounds}], we infer from Equation~\eqref{eq:tree:upper:expectation} that, almost trivially,
\begin{equation*}
\uglobal(f)
\leq\inf\cset{\alpha\in\reals}{\alpha\geq f(\pth)\text{ for all \(\pth\in\pths\)}}
=\sup f.
\end{equation*}
For the other inequality, consider any supermartingale \(\supermartin\in\supermartins\) such that \(\liminf\supermartin\geq f\).
We derive from Equation~\eqref{eq:supermartingale} and~\ref{axiom:coherence:bounds} that \(\supermartin(\sit)\geq\min\set{\supermartin(\sit0),\supermartin(\sit1)}\) for all~\(\sit\in\sits\).
This implies that there is some path~\(\altpth\in\pths\) such that \(\supermartin(\init)\geq\supermartin(\altpthto{n})\) for all \(n\in\naturalswithzero\),\footnote{This argument requires the axiom of dependent choice.} and therefore also that \(\supermartin(\init)\geq\liminf\supermartin(\altpth)\geq f(\altpth)\geq\inf f\).
Equation~\eqref{eq:tree:upper:expectation} then guarantees that, indeed,
\begin{equation*}
\uglobal(f)
=\inf\cset[\big]{\supermartin(\init)}{\supermartin\in\supermartins\text{ and }\liminf\supermartin\geq f}
\geq\inf f.
\end{equation*}
In particular, we find for~\(f=0\) that
\begin{equation}\label{eq:zero:expectation}
\lglobal(0)=\uglobal(0)=0.
\end{equation}

\ref{axiom:lower:upper:subadditivity}.
We prove the third and fourth inequalities; the remaining inequalities will then follow from conjugacy.
For the fourth inequality, we consider any real~\(\alpha\) and~\(\beta\) such that \(\alpha>\uglobal(f)\) and \(\beta>\uglobal(g)\).
Then it follows from Equation~\eqref{eq:tree:upper:expectation} that there are supermartingales~\(\supermartin_1,\supermartin_2\in\supermartins\) such that \(\liminf\supermartin_1\geq f\), \(\liminf\supermartin_2\geq g\), \(\alpha>\supermartin_1(\init)\) and \(\beta>\supermartin_2(\init)\).
But then \(\supermartin\coloneqq\supermartin_1+\supermartin_2\) is a supermartingale for~\(\frcstsystem\) with
\begin{equation*}
\liminf\supermartin
=\liminf(\supermartin_1+\supermartin_2)
\geq\liminf\supermartin_1+\liminf\supermartin_2
\geq f+g,
\end{equation*}
and we therefore infer from Equation~\eqref{eq:tree:upper:expectation} that
\begin{equation*}
\uglobal(f+g)
\leq\supermartin(\init)
=\supermartin_1(\init)+\supermartin_2(\init)
<\alpha+\beta.
\end{equation*}
Since this inequality holds for all real \(\alpha>\uglobal(f)\) and \(\beta>\uglobal(g)\), and since we have proved above that upper expectations of gambles are real-valued, we find that, indeed, \(\uglobal(f+g)\leq\uglobal(f)+\uglobal(g)\).

For the third inequality, observe that \(g=(f+g)-f\), so we infer from the inequality we have just proved that
\begin{equation*}
\uglobal(g)=\uglobal((f+g)-f)\leq\uglobal(f+g)+\uglobal(-f)=\uglobal(f+g)-\lglobal(f),
\end{equation*}
whence, indeed, \(\uglobal(f+g)\geq\lglobal(f)+\uglobal(g)\), since we have already proved above that lower and upper expectations are real-valued.

\ref{axiom:lower:upper:homogeneity}.
We prove the second equality; the first equality then follows from conjugacy.
It follows from Equation~\eqref{eq:zero:expectation} that we may assume without loss of generality that \(\lambda>0\).
The desired equality now follows at once from Equation~\eqref{eq:tree:upper:expectation} and the equivalences \(\supermartin\in\supermartins\ifandonlyif\lambda^{-1}\supermartin\in\supermartins\) and \(\liminf\supermartin\geq\lambda f\ifandonlyif\liminf\lambda^{-1}\supermartin\geq f\).

\ref{axiom:lower:upper:bounds}.
It is only left to prove that \(\lglobal(f)\leq\uglobal(f)\).
Since \(f-f=0\), we infer from~\ref{axiom:lower:upper:subadditivity} and Equation~\eqref{eq:zero:expectation} that \(0=\uglobal(f-f)\leq\uglobal(f)+\uglobal(-f)=\uglobal(f)-\lglobal(f)\).
The desired inequality now follows from the fact that lower and upper expectations are real-valued, as proved above.

\ref{axiom:lower:upper:constantadditivity}.
We prove the first equality; the second will then follow from conjugacy.
Infer from~\ref{axiom:lower:upper:bounds} that \(\lglobal(\mu)=\uglobal(\mu)=\mu\), and then~\ref{axiom:lower:upper:subadditivity} indeed leads to
\begin{equation*}
\lglobal(f)+\mu=\lglobal(f)+\lglobal(\mu)\leq\lglobal(f+\mu)\leq\lglobal(f)+\uglobal(\mu)=\lglobal(f)+\mu.
\end{equation*}

\ref{axiom:lower:upper:monotonicity}.
We prove the first implication; the second will then follow from conjugacy.
Assume that \(f\leq g\), then \(\inf(g-f)\geq0\), so we infer from~\ref{axiom:lower:upper:bounds} and~\ref{axiom:lower:upper:subadditivity} that, indeed,
\begin{equation*}
0\leq\inf(g-f)\leq\lglobal(g-f)\leq\lglobal(g)+\uglobal(-f)=\lglobal(g)-\lglobal(f).
\end{equation*}
The desired inequality now follows from the fact that lower and upper expectations are real-valued, as proved above.
\end{proof}

\begin{proof}[Proof of Proposition~\ref{prop:computable:simplified}]
The `if' part is immediate when we let \(e(\sit,N)\coloneqq N\) for all~\(\sit\in\sits\) and \(N\in\naturalswithzero\), so we proceed to the `only if' part.
That \(\process\) is {\comp} means that there is some recursive net of rational numbers~\(\rprime[\sit,n]\) and some recursive map~\(e\colon\sits\times\naturalswithzero\to\naturalswithzero\) such that \(n\geq e(\sit,N)\) implies that \(\abs{\rprime[\sit,n]-F(\sit)}\leq2^{-N}\) for all~\(\sit\in\sits\) and \(N\in\naturalswithzero\).
The net of rational numbers defined by~\(r_{\sit,n}\coloneqq\rprime[\sit,e(\sit,n)]\) for all~\(\sit\in\sits\) and \(n\in\naturalswithzero\) is recursive because the function~\(e\) is recursive, and it indeed satisfies \(\abs{r_{\sit,n}-\process(\sit)}=\abs{\rprime[\sit,e(\sit,n)]-\process(\sit)}\leq2^{-n}\) for all~\(\sit\in\sits\) and \(n\in\naturalswithzero\).
\end{proof}

\begin{proof}[Proof of Proposition~\ref{prop:computable:upper:lower}]
We begin with the `if' part.
Assume that \(\process\) is both lower and upper {\scomp}.
This implies that there are two recursive nets of rational numbers~\(\lowr[\sit,n]\) and \(\uppr[\sit,n]\) such that \(\lowr[\altsit,n]\nearrow\process(\altsit)\) and \(\uppr[\altsit,n]\searrow\process(\altsit)\) for any fixed~\(\altsit\in\sits\).
Consider the recursive nets of rational numbers defined by~\(\delta_{\,\sit,n}\coloneqq\uppr[\sit,n]-\lowr[\sit,n]\geq0\) and \(r_{\sit,n}\coloneqq\group{\lowr[\sit,n]+\uppr[\sit,n]}/2\).
For any fixed~\(\altsit\in\sits\), the sequence~\(\delta_{\,\altsit,n}\searrow0\), which implies that for any~\(N\in\naturalswithzero\) there is some natural number~\(e(\altsit,N)\) such that \(\delta_{\,\altsit,n}\leq2^{-N}\) for all~\(n\geq e(\altsit,N)\).
Clearly, the map~\(e\colon\sits\times\naturalswithzero\to\naturalswithzero\) can be defined recursively, and we see that \(n\geq e(\sit,N)\) also implies that \(\abs{\process(\sit)-r_{\sit,n}}\leq\abs{\uppr[\sit,n]-\lowr[\sit,n]}=\delta_{\,\sit,n}\leq2^{-N}\), for all~\(\sit\in\sits\) and \(n,N\in\naturalswithzero\).
Hence, the real process~\(\process\) is also {\comp}.

We continue with the `only if' part.
Assume that \(\process\) is {\comp}, so there is a recursive net of rational numbers~\(r_{\sit,n}\) and a recursive map~\(e\colon\sits\times\naturalswithzero\to\naturalswithzero\) such that
\(n\geq e(\sit,N)\) implies that \(\abs{r_{\sit,n}-\process(\sit)}\leq2^{-N}\) for all~\(\sit\in\sits\) and \(n,N\in\naturalswithzero\).
We prove that \(\process\) is {\lscomp}; the proof that \(\process\) is {\uscomp} is completely similar.
Consider the recursive net of rational numbers defined by~\(\rprime[\sit,n]\coloneqq r_{\sit,e(\sit,n+2)}-3\cdot2^{-\group{n+2}}\) for all \(\sit\in\sits\) and \(n\in\naturalswithzero\).
Then we know that \(\abs{\rprime[\sit,n]+3\cdot2^{-\group{n+2}}-\process(\sit)}\leq2^{-\group{n+2}}\) and therefore also
\begin{equation*}
-2^{-n}
=-2^{-\group{n+2}}-3\cdot2^{-\group{n+2}}
\leq\rprime[\sit,n]-\process(\sit)
\leq2^{-\group{n+2}}-3\cdot2^{-\group{n+2}}
=-2^{-\group{n+1}}
\leq2^{-n},
\end{equation*}
for all~\(\sit\in\sits\) and \(n\in\naturalswithzero\), which then tells us that \(\rprime[\sit,n]\leq\process(\sit)-2^{-\group{n+1}}\leq\rprime[\sit,n+1]\) and that \(\abs{\rprime[\sit,n]-\process(\sit)}\leq2^{-n}\), again for all~\(\sit\in\sits\) and \(n\in\naturalswithzero\).\footnote{Note that this implies that we can always assume without loss of generality from the outset for our original net~\(r_{\sit,n}\) that it is non-decreasing as a function of~\(n\), that \(r_{\sit,n}<\process(\sit)\) and that \(e(\sit,n)=n\) for all~\(\sit\in\sits\) and~\(n\in\naturalswithzero\).}
Hence, we find for the recursive net of rational numbers~\(\rprime[\sit,n]\) that \(\rprime[\sit,n]\nearrow\process(\sit)\) for all~\(\sit\in\sits\), which implies that \(\process\) is indeed {\lscomp}.
\end{proof}

\begin{proof}[Proof of Proposition~\ref{prop:IcompIffGammaComp}]
Obviously, the constant real processes~\(\lconstantfrcstsystem[I](\sit)\coloneqq\lp\) and \(\uconstantfrcstsystem[I](\sit)\coloneqq\up\) are {\comp} if and only if their constant values~\(\lp\) and~\(\up\) are.
\end{proof}

\begin{proof}[Proof of Proposition~\ref{prop:computable:from:delta}]
We only prove the third statement.
The proof for the first and second statements are similar to the proof of the `if' part of the third one, but simpler.
There are a number of ways to prove the third statement, but we will use Proposition~\ref{prop:computable:simplified}.

For the `if' part, we assume that \(\process(\init)\) and \(\adddelta\process\) are {\comp}.
Proposition~\ref{prop:computable:simplified} then implies that there are a recursive sequence of rational numbers~\(r_{\init,n}\) and two recursive nets of rational numbers~\(r_{\sit,n}^x\) such that \(\abs{\process(\init)-r_{\init,n}}\leq2^{-n}\) and \(\abs{\adddelta\process(\sit)(x)-r_{\sit,n}^x}\leq2^{-n}\) for all~\(\sit\in\sits\), \(n\in\naturalswithzero\) and \(x\in\outcomes\).
We now define the recursive net of rational numbers~\(r_{\sit,n}\) as follows: for any~\(\sit\in\sits\) and any~\(n\in\naturalswithzero\), let
\begin{equation*}
r_{\sit,n}\coloneqq r_{\init,n}+\smashoperator{\sum_{k=1}^{\dist{\sit}}}r_{\sitto{k-1},n}^{\sitat{k}}.
\end{equation*}
Then, since also
\begin{equation*}
\process(\sit)=\process(\init)+\smashoperator{\sum_{k=1}^{\dist{\sit}}}\adddelta\process(\sitto{k-1})(\sitat{k}),
\end{equation*}
we see that
\begin{equation*}
\abs{\process(\sit)-r_{\sit,n}}
\leq\abs{\process(\init)-r_{\init,n}}
+\smashoperator{\sum_{k=1}^{\dist{\sit}}}\abs[\big]{\adddelta\process(\sitto{k-1})(\sitat{k})-r_{\sitto{k-1},n}^{\sitat{k}}}
\leq(\dist{\sit}+1)2^{-n},
\end{equation*}
so if we define the (clearly) recursive map~\(e\) by
\begin{equation*}
e(\sit,N)\coloneqq N+\dist{\sit}\geq N+\log_2(\dist{\sit}+1)
\text{ for all \(\sit\in\sits\) and \(N\in\naturalswithzero\)},
\end{equation*}
then \(n\geq e(\sit,N)\) implies that \(\abs{\process(\sit)-r_{\sit,n}}\leq2^{-N}\) for all~\(\sit\in\sits\) and \(n\in\naturalswithzero\).
Hence, \(\process\) is {\comp}.

For the `only if' part, assume that \(\process\) is {\comp}.
Then definitely in particular also its value \(\process(\init)\) in the initial situation~\(\init\) is {\comp}, so it only remains to prove that the process difference \(\adddelta\process\) is {\comp}.
Consider, to this effect, any~\(x\in\outcomes\).
It follows from the {\compy} of~\(\process\) and Proposition~\ref{prop:computable:simplified} that there is some recursive net of rational numbers~\(\rprime[\sit,n]\) such that \(\abs{\process(\sit)-\rprime[\sit,n]}\leq2^{-n}\) and \(\abs{\process(\sit x)-\rprime[\sit x,n]}\leq2^{-n}\) and therefore also
\begin{equation*}
\rprime[\sit x,n]-\rprime[\sit,n]-2^{-\group{n-1}}
\leq\process(\sit x)-\process(\sit)
\leq\rprime[\sit x,n]-\rprime[\sit,n]+2^{-\group{n-1}}
\text{ for all~\(\sit\in\sits\) and \(n\in\naturalswithzero\).}
\end{equation*}
If we now let \(r_{\sit,n}^x\coloneqq\rprime[\sit x,n+1]-\rprime[\sit,n+1]\), then this defines a recursive net of rational numbers~\(r_{\sit,n}^x\) that satisfies \(\abs{\adddelta\process(\sit)(x)-r_{\sit,n}^x}\leq2^{-n}\) for all~\(\sit\in\sits\) and all~\(n\in\naturalswithzero\).
Hence, the real process~\(\adddelta\process(\cdot)(x)\) is {\comp} by Proposition~\ref{prop:computable:simplified}, and so is, therefore, the process difference~\(\adddelta\process\).
\end{proof}

\begin{proof}[Proof of Proposition~\ref{prop:computable:from:multiplier}]
We only give the proof for the first statement.
The proof for the second statement is similar but simpler, and the third statement then follows readily from the first and the second, and Propositions~\ref{prop:computable:upper:lower} and~\ref{prop:computable:from:delta}.

Assume that the multiplier process \(\multprocess\) is {\lscomp}.
This implies that there are two recursive nets of rational numbers~\(r_{\sit,n}^{x}\) such that \(r_{\sit,n}^x\nearrow\multprocess(\sit)(x)\), for~\(x\in\outcomes\).
Since \(\multprocess(\sit)(x)\geq0\), we may assume without loss of generality that \(r_{\sit,n}^x\geq0\) too [otherwise replace this recursive net of rational numbers with the recursive net of rational numbers \(\max\set{0,r_{\sit,n}^{x}}\)].
We now construct a recursive net of rational numbers~\(r_{\sit,n}\) as follows: for any~\(\sit\in\sits\) and for any~\(n\in\naturalswithzero\), we let \(\smash{r_{\sit,n}\coloneqq\prod_{k=0}^{\dist{\sit}-1}r_{\sitto{k},n}^{\sitat{k+1}}}\).
Then, since also \(\smash{\mint(\sit)=\prod_{k=0}^{m-1}\multprocess(\sitto{k})(\sitat{k+1})}\) and \(\smash{r_{\sitto{k},n}^{\sitat{k+1}}\nearrow\multprocess(\sitto{k})(\sitat{k+1})}\) for all~\(k\in\set{0,1,\dots,\dist{\sit}-1}\), we find that \(r_{\sit,n}\nearrow\mint(\sit)\) for all~\(\sit\in\sits\), so \(\mint\) is indeed {\lscomp}.
\end{proof}

\begin{proof}[Proof of Proposition~\ref{prop:computablemultiplier:from:process}]
Since \(\mint\) is positive, it follows trivially that \(\multprocess\) is positive as well.
Consider now any~\(\xval\in\outcomes\).
Since \(\mint\) is {\comp}, it follows from Proposition~\ref{prop:computable:from:delta} that \(\adddelta\mint\) is {\comp}, and therefore, we know that \(\adddelta\mint(\sit)(x)\) is {\comp} as well.
Hence, since \(\mint\) is {\comp} and positive, and
\begin{equation*}
\multprocess(\sit)(x)
=\frac{\mint(\sit)+\adddelta\mint(\sit)(x)}{\mint(\sit)}
=1+\frac{\adddelta\mint(\sit)(x)}{\mint(\sit)}
\text{ for all \(\sit\in\sits\)},
\end{equation*}
we find that the real process \(\multprocess(\sit)(x)\), \(\sit\in\sits\) is {\comp}, and so is therefore \(\multprocess\).

For the second statement, consider any positive {\comp} real process \(\process\), and let
\begin{equation*}
\multprocess(\sit)(x)
\coloneqq\frac{\process(\sit x)}{\process(\sit)}>0
\text{ for all \(\sit\in\sits\) and \(x\in\outcomes\)},
\end{equation*}
then \(\multprocess\) is clearly a positive multiplier process with \(\mint=\process/\process(\init)\).
This implies that \(\mint\) is {\comp}.
The first part of the proposition now implies that \(\multprocess\) is {\comp}.
\end{proof}

\begin{proof}[Proof of Proposition~\ref{prop:computably:unbounded:various}]
We prove that~\ref{it:computably:unbounded}\(\then\)\ref{it:computably:unbounded:real}\(\then\)\ref{it:computably:unbounded:fraction}\(\then\)\ref{it:computably:unbounded}.

\ref{it:computably:unbounded}\(\then\)\ref{it:computably:unbounded:real}.
Trivial, because for any growth function \(\ordering\), the map~\(\realordering\coloneqq\ordering\) is a real growth function.

\ref{it:computably:unbounded:real}\(\then\)\ref{it:computably:unbounded:fraction}.
Fix any~\(r\in\naturals\).
Because the real growth function~\(\realordering\) is non-decreasing and unbounded, there is some~\(m_r\in\naturalswithzero\) such that \(\realordering(m)>\nicefrac2r\) for all natural~\(m\geq m_r\).
For any such \(m\geq m_r\), it follows from the assumption that there is some natural~\(n_{r,m}\geq m\) for which \(\mu(n_{r,m})>\realordering(n_{r,m})-\nicefrac1r\), and therefore, since also \(\realordering(n_{r,m})\geq\realordering(m)>\nicefrac2r\),
\begin{equation*}
\frac{\mu(n_{r,m})}{\realordering(n_{r,m})}>1-\frac1{r\realordering(n_{r,m})}>\frac12.
\end{equation*}
This implies that, indeed, \(\limsup_{n\to\infty}\nicefrac{\mu(n)}{\realordering(n)}>0\).

\ref{it:computably:unbounded:fraction}\(\then\)\ref{it:computably:unbounded}.
We begin by showing that there is some real growth function \(\realordering'\) such that \(\limsup_{n\to\infty}\sqgroup{\mu(n)-\realordering'(n)}>0\).
It follows from the assumption that there is some~\(r\in\naturals\) such that \(\limsup_{n\to\infty}\nicefrac{\mu(n)}{\realordering(n)}>\nicefrac1r\), and also that there is some~\(n_o\in\naturalswithzero\) such that~\(\realordering(n)>0\) for all \(n\geq n_o\).
If we now let \(\realordering'\coloneqq\nicefrac{\realordering}{2r}\) for all~\(n\in\naturalswithzero\), then it is clear that \(\realordering'\) is a real growth function, and that \(\limsup_{n\to\infty}\nicefrac{\mu(n)}{\realordering'(n)}>2\).
This implies that for all \(n\in\naturalswithzero\) with~\(n\geq n_o\), there is some~\(m_n\geq n\) in~\(\naturalswithzero\) such that \(\nicefrac{\mu(m_n)}{\realordering'(m_n)}>2\), and therefore also
\begin{equation*}
\mu(m_n)-\realordering'(m_n)
>\realordering'(m_n)
\geq\realordering'(n)
\geq\realordering'(n_o).
\end{equation*}
This implies that, indeed, \(\limsup_{n\to\infty}[\mu(n)-\realordering'(n)]\geq\realordering'(n_o)>0\).
Because \(\realordering'\) is computable, we know from Proposition~\ref{prop:computable:simplified} [after identifying the countably infinite sets~\(\sits\) and~\(\naturalswithzero\)] that there is some recursive net of rational numbers \(\rprime[k,n]\) such that
\begin{equation*}
\abs{\rprime[k,n]-\realordering'(k)}\leq2^{-n}
\text{ for all~\(k,n\in\naturalswithzero\)}.
\end{equation*}
If we now define the sequence of rational numbers \(r_k\coloneqq\rprime[k,k+1]-3\cdot2^{-k}\), then this sequence is clearly a recursive sequence of rational numbers for which \(\abs{r_k+3\cdot2^{-k}-\realordering'(k)}\leq 2^{-(k+1)}<2^{-k}\) and therefore also
\begin{equation}\label{eq:computably:unbounded:real}
\realordering'(k)-4\cdot2^{-k}
<r_k
<\realordering'(k)-2\cdot2^{-k}
\text{ for all~\(k\in\naturalswithzero\)}.
\end{equation}
Hence also
\begin{equation*}
r_{k+1}
>\realordering'(k+1)-4\cdot2^{-(k+1)}
\geq\realordering'(k)-2\cdot2^{-k}
>r_k
\text{ for all~\(k\in\naturalswithzero\)},
\end{equation*}
where the strict inequalities follow from Equation~\eqref{eq:computably:unbounded:real}, and the weak inequality from the non-decreasing character of the real growth function~\(\realordering'\).
This tells us that the sequence~\(r_k\) is increasing.
Equation~\eqref{eq:computably:unbounded:real} tells us that it is also unbounded, because \(\realordering'\) is.
If we therefore define the map \(\ordering'\colon\naturalswithzero\to\integers\) by letting \(\ordering'(k)\coloneqq\floor{r_k}\) for all~\(k\in\naturalswithzero\), then this map is recursive because \(r_k\) is a recursive sequence of rational numbers, non-decreasing because the sequence~\(r_k\) is increasing, and unbounded because the sequence~\(r_k\) is.
Since
\begin{equation*}
-\ordering'(k)=-\floor{r_k}=\ceil{-r_k}\geq-r_k>-\realordering'(k)+2\cdot2^{-k},
\end{equation*}
where we have used Equation~\eqref{eq:computably:unbounded:real}, we find that
\begin{equation*}
\mu(k)-\ordering'(k)
\geq\mu(k)-\realordering'(k)+2\cdot2^{-k}
\text{ for all~\(k\in\naturalswithzero\)},
\end{equation*}
and therefore \(\limsup_{n\to\infty}\sqgroup{\mu(n)-\ordering'(n)}\geq\limsup_{n\to\infty}\sqgroup{\mu(n)-\realordering'(n)}>0\).
The same inequality of course also holds if we replace~\(\ordering'\) by the growth function~\(\ordering\coloneqq\max\set{0,\ordering'}\).
\end{proof}

\begin{proof}[Proof of Proposition~\ref{prop:computably:unbounded:implies:unbounded}]
Consider any real \(R>0\).
Since \(\mu\) is computably unbounded, there is some growth function~\(\ordering\) such that Equation~\eqref{eq:computably:unbounded} holds.
Since \(\ordering\) is unbounded, there is some~\(m_R\in\naturalswithzero\) such that \(\ordering(m_R)>R\), and then Equation~\eqref{eq:computably:unbounded} implies that there is some natural~\(n_R\geq m_R\) such that \(\mu(n_R)>\ordering(n_R)\geq\ordering(m_R)>R\) [the weak inequality follows from the non-decreasing character of the growth function~\(\ordering\)].
Hence, \(\mu\) is unbounded above.
\end{proof}

\begin{proof}[Proof of Proposition~\ref{prop:computably:unbounded:product}]
Assume {\itshape ex absurdo} that both \(\mu_1\) and \(\mu_2\) are not computably unbounded.
That the product~\(\mu_1\mu_2\) is computably unbounded implies, by Proposition~\ref{prop:computably:unbounded:various}\ref{it:computably:unbounded:fraction}, that there is some real growth function~\(\realordering\) and some natural number~\(r>0\) such that
\begin{equation}\label{eq:computably:unbounded:product:start}
(\forall m\in\naturalswithzero)(\exists n_m\geq m)
\frac{\mu_1(n_m)\mu_2(n_m)}{\realordering(n_m)}>\frac{1}{r^2}.
\end{equation}
If we consider the real growth functions \(\realordering_1\) and \(\realordering_2\) defined by~\(\realordering_1(n)=\realordering_2(n)\coloneqq\sqrt{\realordering(n)}\) for all~\(n\in\naturalswithzero\), then clearly \(\realordering(n)=\realordering_1(n)\realordering_2(n)\) for all~\(n\in\naturalswithzero\), and therefore Equation~\eqref{eq:computably:unbounded:product:start} guarantees that
\begin{equation}\label{eq:computably:unbounded:product:intermediate}
(\forall m\in\naturalswithzero)(\exists n_m\geq m)
\frac{\mu_1(n_m)}{\realordering_1(n_m)}\frac{\mu_2(n_m)}{\realordering_2(n_m)}>\frac{1}{r^2}.
\end{equation}
But the assumption that \(\mu_1\) and \(\mu_2\) are not computably unbounded, in combination with Proposition~\ref{prop:computably:unbounded:various}\ref{it:computably:unbounded:fraction}, guarantees in particular that there are \(m_1,m_2\in\naturalswithzero\) such that
\begin{equation*}
(\forall n\geq m_1)\frac{\mu_1(n)}{\realordering_1(n)}\leq\frac{1}{r}
\text{ and }
(\forall n\geq m_2)\frac{\mu_2(n)}{\realordering_2(n)}\leq\frac{1}{r},
\end{equation*}
and therefore
\begin{equation*}
(\forall n\geq\max\set{m_1,m_2})\frac{\mu_1(n)}{\realordering_1(n)}\frac{\mu_2(n)}{\realordering_2(n)}\leq\frac{1}{r^2},
\end{equation*}
contradicting Equation~\eqref{eq:computably:unbounded:product:intermediate}.
\end{proof}

\end{document}